\documentclass[12pt,reqno]{amsart}

\usepackage{amssymb, amsmath, amsthm, amsfonts}
\usepackage{mathabx}
\usepackage{thmtools}
\usepackage[centering, margin=1.1in]{geometry}
\usepackage{tikz}
\usetikzlibrary{positioning,arrows.meta,matrix}
\usepackage{mathtools}
\usepackage{enumitem}
\usepackage[space]{cite}

\usepackage[colorlinks, linkcolor=black, citecolor=black, hypertexnames=false]{hyperref}
\usepackage[capitalize, nameinlink]{cleveref}

\usepackage[alphabetic, nobysame]{amsrefs}

\AtBeginDocument{\def\MR#1{}}

\declaretheorem[name=Theorem, numberwithin=section]{theorem}

\declaretheorem[name=Lemma, sibling=theorem]{lemma}
\declaretheorem[name=Proposition, sibling=theorem]{prop}
\declaretheorem[name=Corollary, sibling=theorem]{corollary}

\declaretheorem[name=Remark, sibling=theorem]{remark}

\declaretheorem[name=Claim, numbered=no]{claim*}

\crefname{section}{Section}{Sections}
\Crefname{section}{Section}{Sections}

\crefname{appendix}{Appendix}{Appendices}
\Crefname{appendix}{Appendix}{Appendices}

\newcommand*{\bbone}{\text{\usefont{U}{bbold}{m}{n}1}}

\DeclareMathOperator*{\Res}{Res}

\DeclareMathOperator*{\GL}{GL}

\DeclareMathOperator*{\ra}{rad}

\renewcommand{\Im}{\operatorname{Im}}
\renewcommand{\Re}{\operatorname{Re}}  

\newcommand{\ord}{\operatorname{ord}}

\newcommand{\Z}{\mathbb{Z}}
\newcommand{\Q}{\mathbb{Q}}
\newcommand{\R}{\mathbb{R}}

\newcommand{\C}{\mathbb{C}}
\newcommand{\N}{\mathbb{N}}

\newcommand{\ep}{\varepsilon}

\newcommand{\sumtwo}{\operatorname*{\sum\sum}}
\newcommand{\sumthree}{\operatorname*{\sum\sum\sum}}

\newcommand{\ec}{\check{e}}
\newcommand{\f}{\omega}
\newcommand{\m}{\mathfrak m}
\renewcommand{\pmod}[1]{\ (\mathrm{mod}\ #1)}
\newcommand{\ls}{\lesssim}

\newcommand{\quartic}[2]{\Big(\frac{#1}{#2}\Big)_4}
\newcommand{\quadrat}[2]{\Big(\frac{#1}{#2}\Big)_2}
\newcommand{\pfrac}[2]{\left(\frac{#1}{#2}\right)}

\DeclareFontFamily{U}{wncy}{}
\DeclareFontShape{U}{wncy}{m}{n}{<->wncyr10}{}
\DeclareSymbolFont{mcy}{U}{wncy}{m}{n}
\DeclareMathSymbol{\Sha}{\mathord}{mcy}{"58} 

\numberwithin{equation}{section}
\mathtoolsset{showonlyrefs} 

\makeatletter
\@namedef{subjclassname@2020}{\textup{2020} Mathematics Subject Classification}
\makeatother

\begin{document}

\author{Cruz Castillo}
\address{Department of Mathematics, University of Illinois, Urbana, USA}
\email{ccasti30@illinois.edu}

\author{Alexandre de Faveri}
\address{EPFL SB MATH, Station 10, 1015 Lausanne, Switzerland}
\email{alexandre.defaveri@epfl.ch}

\author{Alexander Dunn}
\address{School of Mathematics, Georgia Institute of Technology,
    Atlanta, USA}
\email{adunn61@gatech.edu}

\title[Non-vanishing for quartic Hecke $L$-functions and ranks of elliptic curves]{Non-vanishing for quartic Hecke $L$-functions \\ and ranks of elliptic curves}

\begin{abstract}
    We show that a positive proportion of Hecke $L$-functions attached to the quartic residue symbols $\big( \tfrac{\cdot}{q} \big)_4$ for squarefree $q \in \Z[i]$ do not vanish at the central point. Our method also extends to the Hecke characters associated to quartic twists of the congruent number curve $E: y^2=x^3-x$. In particular, we prove that the elliptic curve $E^{(q)} : y^2=x^3-qx$ has Mordell--Weil rank $0$ over $\Q(i)$ for a positive proportion of squarefree $q \in \Z[i]$ ordered by norm.
 \end{abstract}

\maketitle

\tableofcontents


\section{Introduction}

The non-vanishing problem for central values of $L$-functions is a fundamental question of great arithmetic significance. For example, the Birch and Swinnerton-Dyer (BSD) conjecture \cite{BSD} asserts that the rank of an elliptic curve is equal to the order of vanishing of its $L$-function at the central point $s=\frac{1}{2}$ (when analytically normalized). Substantial progress towards this conjecture has been made in the case of CM elliptic curves.
For an elliptic curve $E_{/K}$ that has complex multiplication by the ring of integers of an imaginary quadratic field $K$, Coates and Wiles~\cite{CW76} proved that if $L(1/2, E_{/K}) \ne 0$, then the abelian group of $K$-rational points $E(K)$ is finite. 
Further extensions and important progress were obtained in \cites{Art78, GZ86, Rub87, Kol90, Rubin91}.

In this paper we study quartic twists of the congruent number curve $E: y^2 = x^3-x$ over $\Q(i)$. More specifically, for squarefree $q \in \Z[i]$, consider the family of elliptic curves 
\begin{equation*}
    E^{(q)}: y^2 = x^3 - qx,    
\end{equation*}
which have complex multiplication by $\Z[i]$. We prove the following result on $E^{(q)}(\Q(i))$.

\begin{theorem}[Positive proportion of rank $0$]\label{thm:elliptic}
    For a positive proportion of squarefree $q \in \Z[i]$ ordered by norm, the elliptic curve $E^{(q)}: y^2 = x^3-qx$ satisfies $E^{(q)}(\Q(i)) \simeq \Z/2\Z$.
\end{theorem}

It suffices to establish Mordell--Weil rank $0$ above, since \cite[Theorem 2 (i)]{Naj10} implies $E^{(q)}(\Q(i))_{\text{tors}} \simeq \Z/2\Z$, generated by $(0, 0)$, for all but finitely many squarefree $q \in \Z[i]$. Since $E^{(q)}_{/\Q(i)}$ has CM, its Hasse--Weil $L$-function is a product of Hecke $L$-functions. For $\lambda := 1+i$ and $q \equiv 1 \pmod{\lambda^7}$, a computation analogous to \cite[Theorem 7, p.~310]{IR90} gives
\begin{equation}\label{eq:Eq_factor}
    L\big(s,E^{(q)}_{/\Q(i)}\big) = L(s,\chi_q \psi) \cdot L(s,\overline{\chi_q \psi}),
\end{equation}
where $\chi_q(m) := \big(\frac{m}{q}\big)_4$ is the quartic residue symbol and $\psi((m)) := \frac{\overline{m}}{|m|}$ for $m \equiv 1 \pmod{\lambda^3}$. Then \cref{thm:elliptic} is a consequence of the Coates--Wiles theorem~\cite{CW76} and the following more general result for Hecke $L$-functions of (twisted) quartic characters. 

\begin{theorem}[Positive proportion of non-vanishing]\label{mainthm} 
    There exists a constant $c > 0$ such that for every $\omega \in \Z$ and sufficiently large $X >X_0(\omega)$ we have 
    \begin{align*}
        \sum_{\substack{q \in \Z[i]\\ q \equiv 1 \pmod{\lambda^7}\\ N(q) \le X}} \mu^2(q) \cdot \bbone_{L(1/2,\, \nu_{q,\omega})\ne 0} \ge c \sum_{\substack{q \in \Z[i]\\ q \equiv 1 \pmod{\lambda^7}\\ N(q) \le X }} \mu^2(q),
    \end{align*}
    where $\nu_{q,\omega}((m))$ is the primitive Hecke character equal to $\chi_q(m)(\tfrac{\overline{m}}{|m|})^{\omega}$ for $m \equiv 1 \pmod{\lambda^3}$.
\end{theorem}

\begin{remark}
    A value for the constant $c>0$ may be computed from our work, but we do not specify it (and did not pursue optimality of the proportion). This leads to many technical simplifications, and more importantly allows us to use only convex inputs and the large sieve. For this reason, our method should generalize to higher order characters. Consult \cref{sec:context_limitations} for a detailed discussion of these matters. 
\end{remark}

For each fixed $\omega \in \Z$, the family of Hecke $L$-functions
\begin{equation*}
    \big\{ L(s,\nu_{q,\omega}): q \in \Z[i], \quad q \equiv 1 \pmod{\lambda^7}, \quad \text{and} \quad \mu^2(q)=1 \big\}
\end{equation*}
has unitary symmetry type. Hence we expect $L(s, \nu_{q, \omega}) \neq 0$ for $100\%$ of the members of the family, by the conjectures of Katz and Sarnak~\cite{KatzSarnak}. Alternatively, the root number of $L(s,E^{(q)}_{/\Q(i)})$ is always $1$, so for $\omega=1$ we arrive at the same prediction via Goldfeld's minimalist conjecture \cite{Gold79}. Theorem~\ref{mainthm} is the first unconditional result to establish a positive proportion of non-vanishing in this quartic family. 

\subsection{Related works}

\subsubsection{Fixed order characters}

\cref{mainthm} with $\omega = 0$ gives non-vanishing of $L(1/2, \chi_q)$ for a positive proportion of squarefree $q \in \Z[i]$. Previously, non-vanishing for infinitely many such $q$ was proved\footnote{The setup is slightly different, but this follows from their methods. It is also implicit in \cite{FHL, Dia}.} by Blomer, Goldmakher, and Louvel~\cite{BGL}. Under GRH, Gao and Zhao~\cite{GZ20} obtained a positive proportion of non-vanishing in a similar family. 

This situation is typical of the non-vanishing problem for families of fixed order characters: in many cases, non-vanishing for infinitely many members of the family has been established (by computing the first moment), while under GRH a positive proportion of non-vanishing is known (through either $n$-level density or sharp conditional bounds for moments). For characters of order $n = 2$ or $3$, examples of the former include \cite{Luo04, BY10, G25}, and of the latter include \cite{OS99, DG22, Gz22, GY24}.

In a fundamental result, Soundararajan \cite{sound00} proved a positive proportion of non-vanishing for the family of primitive quadratic Dirichlet\footnote{An open folklore conjecture predicts that $L(1/2,\chi) \neq 0$ for all primitive Dirichlet characters $\chi$.} characters over $\Q$. Recently, a positive proportion of non-vanishing for a family of primitive cubic Hecke characters over $\Q(\zeta_3)$ was proved by David, Stucky, and the second and third authors \cite{DDDS24}.

Less is known about non-vanishing over families of higher order characters. Multiple Dirichlet series have been used to compute moments (with certain arithmetic weights) of $L$-functions of $n$-th order Hecke characters over number fields containing $\Q(\zeta_n)$, resulting in strong error terms. Using this method, the first moment was computed by Friedberg, Hoffstein, and Lieman \cite{FHL}, and the second moment by Diaconu \cite{Dia}. The moments considered in these works are not sieved down to primitive characters, unlike the results of this paper. However, they imply non-vanishing for infinitely many primitive $n$-th order characters. This was also proved in the aforementioned work of Blomer, Goldmakher, and Louvel \cite{BGL}, using the $n$-th order large sieve developed by them (based on work of Heath-Brown \cite{HB}).

Over function fields, it is known for any $n \ge 2$ that a positive proportion of $L$-functions associated to $n$-th order characters with squarefree modulus do not vanish at the central point \cites{BF18,DFL21,DFL25}. We direct the interested reader to \cite{DDDS24} and \cite{Sound23} for more detailed expositions on moments of $L$-functions and non-vanishing.

\subsubsection{Twists of elliptic curves}

The non-vanishing of central $L$-values for twists of elliptic curves has been intensely studied. Goldfeld's conjecture \cite{Gold79} states that $50\%$ of the quadratic twists of an elliptic curve $E_{/\mathbb{Q}}$ should have analytic rank $0$ (respectively $1$). Recently, Smith~\cite{Smith25} showed that Goldfeld's conjecture follows from BSD. Moreover, much unconditional progress has been made due to his breakthrough work \cites{Smith26I, Smith26II} on Selmer groups. For instance, Burungale and Tian \cite{BurTia26} proved a $p$-converse theorem which combined with Smith's results on $2^\infty$-Selmer groups implies Goldfeld's conjecture for the congruent number family $y^2 = x^3-q^2x$, for $q \in \Z$ squarefree.

Goldfeld's minimalist philosophy applied to higher order twists over number fields predicts that their analytic ranks are typically as small as permitted by their root numbers. For higher order twists over $\Q$, Alp{\"o}ge, Bhargava, and Shnidman \cite{ABS24} proved a positive proportion of rank $0$ (respectively rank $1$) in the cubic families $x^3 + y^3 = q$ and $x^3 + y^3 = q^2$, for $q \in \Z$. Conjecturally, each of the ranks $0$ and $1$ occur $50\%$ of the time in this setting. See also \cite{KL19, BES20, KS24} and references therein.

For twists over number fields, Burungale and Tian \cite{BurTia26} showed that if $3$ is not inert\footnote{Note that $K=\Q(i)$ is not covered by the result.} in the imaginary quadratic field $K$ and the elliptic curve $E_{/K}$ has CM by an order of $K$, then a positive proportion of quadratic twists of $E_{/K}$ associated to squarefree $q \in \mathcal{O}_K$ have analytic rank $0$. An important input is the work of Bhargava, Klagsbrun, Lemke Oliver, and Shnidman \cite{BKLS19}. Conjecturally, $100\%$ of quadratic twists over $K$ have rank $0$.

All of the results mentioned so far study Selmer ranks (using various techniques) and apply $p$-converse theorems to convert to analytic ranks. In contrast, our approach tackles analytic ranks directly, providing further information on the $L$-values through their moments. This strategy was used with great success in many early results on ranks of twists of elliptic curves over $\Q$. Examples include works of Bump, Friedberg, and Hoffstein \cite{BFH90}, Murty and Murty \cite{MM91}, and Iwaniec \cite{Iwa90} for quadratic twists\footnote{This was an important input for the seminal work of Kolyvagin \cite{kol88}.}, Lieman \cite{Lie94} for cubic twists, and Diaconu \cite{Dia} for quartic twists. These results obtain non-vanishing for infinitely many twists. A similar strategy was also used by Diaconu and Tian \cite{DT05} to show that the twisted Fermat curve $x^p + y^p = q$ has no $F$-rational points for infinitely many $q \in F^\times/F^{\times p}$, where $p$ is an odd prime and $F$ is totally real with $[F(\zeta_p) : F] = 2$.

In the analytic results above, the relevant $L$-functions have degree $2$ over the base field. Due to the sparseness of the twists considered, at present this seems to limit our knowledge to their first moment (with power-saving error terms). For this reason, obtaining a positive proportion of non-vanishing (with these methods) seems difficult. The starting point of this paper is the observation that in our context, the relevant twisted $L$-functions factor into terms of degree $1$, as in \eqref{eq:Eq_factor}. Thus, while we are still limited to computing a first moment of $L(1/2, E^{(q)}_{/\Q(i)})$, this allows us to compute a ``half-moment'' as well. These two moments can then be mollified to obtain a positive proportion of non-vanishing.

Our methods should also imply that a positive proportion of curves in the cubic family $x^3 + y^3 = q$ with $q \in \Z[\zeta_3]$ squarefree have analytic rank $0$. A good proportion can be obtained by closely following the argument in \cite{DDDS24}. We leave the details to the interested reader.

\subsection{Results on moments}

The proof of \cref{mainthm} goes through asymptotics for the first and second mollified moments. To present our results, let us introduce some notation.  

Recall that $\lambda:=1+i$ denotes the unique ramified prime in $\Z[i]$. Let $\chi_q(n) := \big(\frac{n}{q}\big)_4$ be the quartic residue symbol, defined in \eqref{eq:quar_char}, and $\xi$ be the unique primitive (unitary) Hecke character with $\xi(n) = \big(\frac{\overline{n}}{|n|}\big)^{\omega}$ for $n \equiv 1 \pmod{\lambda^3}$, which is described in \eqref{eq:inf_hecke}. For $q \in \Z[i]$ with $q \equiv 1 \pmod{\lambda^7}$ and $\omega \in \Z$, define the mollifier
\begin{equation} \label{mollifier}
    \mathcal{M}_{\omega}(q):=\sum_{\substack{0 \neq \mathfrak{b} \unlhd \Z[i] \\ N(\mathfrak{b}) \leq M}} \lambda_{\omega}(\mathfrak{b}) \sqrt{N(\mathfrak{b})} \nu_{q,\omega}(\mathfrak{b}),
\end{equation}
where $\nu_{q,\omega} := \chi_q \cdot \xi$ is a primitive Hecke character detailed in \eqref{eq:nu_q}, and the coefficients $\boldsymbol{\lambda}_{\omega}:= (\lambda_{\omega}(\mathfrak b))_{0 \ne b \unlhd \Z[i]}$ are supported on squarefree $\mathfrak{b}$ coprime with $2$ and will be chosen later. Additionally, we assume that $\lambda_{\omega}(\mathfrak b)\ll_{\ep} N(\mathfrak b)^{-1+\ep}$ uniformly with respect to $\omega$.

For any $\mathbb{C}$-valued sequence $\boldsymbol{\beta}$ and smooth $f: \R \to \C$ compactly supported on $(1,2)$, let
\begin{equation}
\mathcal{S}(\boldsymbol{\beta};f) = \mathcal{S}_X(\boldsymbol{\beta};f):=\sum_{\substack{q \in \mathbb{Z}[i] \\ q \equiv 1 \pmod{\lambda^7} }} 
\mu^2(q) \beta_q f \Big( \frac{N(q)}{X} \Big).
\end{equation}
Write $\mu^2(q)=M_Y(q)+R_Y(q)$ for a parameter $Y \geq 1$, where
\begin{equation} \label{MYRYdef}
    M_Y(q):=\sum_{\substack{ \mathfrak{l}^2 \mid q \mathbb{Z}[i] \\ N(\mathfrak{l}) \leq Y  }} \mu(\mathfrak{l}) \qquad \text{and}
    \qquad R_Y(q):=\sum_{\substack{ \mathfrak{l}^2 \mid q \mathbb{Z}[i] \\ N(\mathfrak{l}) > Y  }} \mu(\mathfrak{l}).
\end{equation}
Define
\begin{align} \label{SMdef}
    \mathcal{S}_M(\boldsymbol{\beta};f)=\mathcal{S}_{M,X,Y}(\boldsymbol{\beta};f)
    :=\sum_{\substack{q \in \mathbb{Z}[i] \\ q \equiv 1 \pmod{\lambda^7}}} M_Y(q) \beta_q f \Big( \frac{N(q)}{X} \Big)
\end{align}
and
\begin{equation} \label{SRdef}
    \mathcal{S}_R(\boldsymbol{\beta};f) =\mathcal{S}_{R,X,Y}(\boldsymbol{\beta};f)
    :=\sum_{\substack{q \in \mathbb{Z}[i] \\ q \equiv 1 \pmod{\lambda^7}}}  |R_Y(q) \beta_q| f \Big( \frac{N(q)}{X} \Big).
\end{equation}
If $f$ is non-negative, it follows that
\begin{equation} \label{nonnegproperty}
    \mathcal{S}(\boldsymbol{\beta};f)=\mathcal{S}_{M}(\boldsymbol{\beta};f)+
    O(\mathcal{S}_R(\boldsymbol{\beta};f)).
\end{equation}

Let $F: \R \to [0,1]$ be smooth and compactly supported on $(1,2)$. Our aim is to evaluate the mollified moments
\begin{equation*}
    \mathcal{S}(L(1/2, \nu_{q,\omega})\mathcal{M}_{\omega}(q);F) \qquad \text{and} \qquad \mathcal{S}(|L(1/2, \nu_{q,\omega})\mathcal{M}_{\omega}(q)|^2 ;F).
\end{equation*}
For $U \geq 1$, the first moment will lead to 
\begin{align} \label{SMexpand}
    & \mathcal{S}_M\Big(\mathcal{M}_{\omega}(q) \Big[B_{\omega,U}(q)+\ep(\omega)\xi(q)\widetilde{g}_4(q) \cdot \widetilde{B}_{\omega,U^{-1}}(q)\Big];F\Big) \nonumber \\
    &=\sum_{\substack{ 0 \neq  \mathfrak{b} \unlhd \mathbb{Z}[i]}} \lambda_\omega(\mathfrak{b}) \sqrt{N(\mathfrak{b})} \mathcal{S}_M \Big(  \nu_{q,\omega}(\mathfrak{b}) \Big[B_{\omega,U}(q)+\ep(\omega)\xi(q)\widetilde{g}_4(q) \cdot \widetilde{B}_{\omega,U^{-1}}(q)\Big]; F \Big), \qquad
\end{align}
where $\widetilde{g}_4(q)$ is the normalized quartic Gauss sum given in \eqref{eq:normalized_gauss_sum}, $\ep(\omega) =\pm 1$ is given by \eqref{eq:r_num_f}, $B_{\omega,U}(q)$ is defined in \eqref{BUdef}, and $\widetilde{B}_{\omega,U}(q)$ in \eqref{BUtildedef}.
Similarly, the second moment will lead to a multiple of
\begin{align} \label{SMexpand2}
    & \mathcal{S}_M\big(|\mathcal{M}_{\omega}(q)|^2 A_{\omega}(q);F\big) \nonumber \\
    &=\mathop{\sum \sum}_{\substack{ 0 \neq  \mathfrak{b}_1,\mathfrak{b}_2 \unlhd \mathbb{Z}[i]}} \lambda_{\omega}(\mathfrak{b}_1) \overline{ \lambda_{\omega}(\mathfrak{b}_2)} \sqrt{N(\mathfrak{b}_1 \mathfrak{b}_2 )}
    \mathcal{S}_M \Big(\nu_{q,\omega}(\mathfrak b_1) \overline{\nu_{q,\omega}(\mathfrak b_2)} A_{\omega}(q); F \Big),
\end{align}
where $A_{\omega}(q)$ is defined in \eqref{Aomegadef}. Then it suffices to evaluate the terms in \eqref{SMexpand} and \eqref{SMexpand2}.

For the first moment we have the following result, where $\check{F}(w) := \int_0^{\infty}F(t) t^{w} dt$.

\begin{prop}\label{prop:first_moment}
    Let $0 \neq \mathfrak b \unlhd \Z[i]$ be squarefree and coprime with $2$,
    with $N(\mathfrak{b}) \sim B$ for some $1 \le B \le X^{100}$.
    Let $\omega \in \mathbb{Z}$, $1 \le Y, U \le X^{100}$, and $\varepsilon>0$. Assume that 
    \begin{equation} \label{firstmomentassump}
        (1+|\omega|) B  Y^2 U \le X^{1/2-\nu}
    \end{equation}
    for some fixed $0<\nu<\frac{1}{2}$. Then
    \begin{align*}
        & \mathcal{S}_M \Big( \nu_{q,\omega}(\mathfrak{b}) \Big[B_{\f,U}(q) + \ep(\f)\xi(q)\widetilde{g}_4(q) \cdot \widetilde{B}_{\f,U^{-1}}(q) \Big]; F \Big)  \\
        & = C_{\omega}X\check{F}(0)\frac{\xi(\mathfrak{b}^4) r_{\omega}(\mathfrak b)}{N(\mathfrak b)^{3/2}}+O_\varepsilon\Big(\frac{X}{Y B^{3/2}}+\frac{ X^{7/8+\ep}}{(1+|\omega|)^{1/4} U^{1/4} B^{3/4}} \Big) + R_{\omega,U}(\mathfrak b),
    \end{align*}
    where
    \begin{equation}\label{eq:C_f}
        C_{\omega}:= \frac{\pi }{48 \sqrt 2 \cdot \zeta_{\Q(i)}(2) \cdot (\sqrt{2}-\xi(\lambda)) } \prod_{\substack{\mathfrak p \textnormal{ prime}\\ (\mathfrak p,2)=1\\q:=N(\mathfrak p)}}\Big(1+ \frac{q}{(q+1)(q^2 \overline{\xi(\mathfrak p^4)}-1)}\Big).
    \end{equation}
    The multiplicative function $r$ is given, for $\mathfrak{p}$ prime, $q:=N(\mathfrak{p})$, and $k \ge 1$, by
    \begin{equation}\label{eq:def r}
        r(\mathfrak p^k):= \frac{q^{3}}{q^3+q^2-\xi(\mathfrak{p}^4)} = 1 + O\Big(\frac{1}{q}\Big),
    \end{equation}
    and there exists an absolute constant $H \ge 1$ such that 
    \begin{equation} \label{Rfirstbd}
        R_{\omega,U}(\mathfrak b) \ll_{F,\nu, \varepsilon} X^{\ep} \cdot \big(BY(1+|\omega|)\big)^{H} \cdot \Big(  \bbone_{\omega=0} \cdot \frac{X^{17/16}}{U^{5/8}}
        + \frac{X^{7/8}}{U^{3/4}} \Big).
    \end{equation}
\end{prop}

Similarly, we obtain the following asymptotic formula for the second moment.

\begin{prop}\label{prop:second_moment}
    Let $0 \neq \mathfrak{b}_1,\mathfrak{b}_2 \unlhd \Z[i]$ be squarefree and coprime with $2$. Suppose that $N(\mathfrak{b}_1) \sim B_1$ and $N(\mathfrak{b}_2)\sim B_2$ for $1 \leq B_1, B_2 \leq X^{100}$, and denote $B:=\max(B_1,B_2)$. Set $\mathfrak{b} = (\mathfrak{b}_1,\mathfrak{b}_2)$ and $\mathfrak{b}_i=\mathfrak{a}_i \mathfrak{b}$ for $i \in \{1, 2\}$. For any $\omega \in \Z$, $1 \le Y \le X^{100}$, and $\ep >0$, we have
    \begin{align*}
        & \mathcal{S}_M\big(\nu_{q,\omega}(\mathfrak{b}_1)\overline{\nu_{q,\omega}(\mathfrak{b}_2)}A_{\omega}(q);F\big) \\
        &= D_{\omega}\check{F}(0)X\frac{g_\f(\mathfrak{b}) h_{\omega}(\mathfrak{a}_1) \overline{h_{\omega}(\mathfrak{a}_2)}}{\sqrt{N(\mathfrak{a}_1 \mathfrak{a}_2)}}\Big[\log \Big(\frac{(1+|\f|)^2 X}{N(\mathfrak{a}_1 \mathfrak{a}_2)}\Big)+\mathcal O_\omega(\mathfrak{b}_1,\mathfrak{b}_2)\Big]  \\
        & +O_{\ep}\Big(\frac{(1+|\omega|)^{\ep}X^{1+\ep}}Y+\frac{X^{3/4+\ep}}{(1+|\omega|)^{1/2-\varepsilon}N(\mathfrak{a}_1 \mathfrak{a}_2)^{1/4}}\Big)+ \mathcal R_{\omega}(\mathfrak{b}_1,\mathfrak{b}_2),
    \end{align*}
    where
    \begin{equation}\label{eq:D_omega}
        D_{\omega}:= \frac{\pi^2}{768\cdot\zeta_{\Q(i)}(2)\cdot|\sqrt 2 -\xi(\lambda)|^2}\prod_{\substack{\mathfrak p \textnormal{ prime}\\ (\mathfrak p,2)=1\\q:=N(\mathfrak p)}} \Big( 1 -\frac{1}{q(q+1)} +  2 \Re \Big(\frac{q}{(q+1)(q^2\xi(\mathfrak{p}^4)-1)}\Big)\Big).
    \end{equation}
    The multiplicative functions $g_\f$ and $h_\f$ are given,
    for $\mathfrak{p}$ prime, $q:=N(\mathfrak{p})$, and $k \ge 1$, by
    \begin{equation}\label{eq:g_xi_1}
        g_{\omega}(\mathfrak{p}^k) := \frac{q^2(q^2+1)(q+1)}{q^5+2q^4+q^3+(1-2\Re{\xi(\mathfrak{p}^4)})q^2+1} = 1 + O\Big(\frac{1}{q}\Big)
    \end{equation}
    and
    \begin{equation}\label{eq:h_xi}
        h_{\omega}(\mathfrak{p}^k) := \frac{q^3(q^2+q(\xi(\mathfrak{p}^4)+1) + 1)}{q^5+2 q^4+q^3+(1-2\Re{\xi(\mathfrak{p}^4)})q^2+1}  = 1 + O\Big(\frac{1}{q}\Big).
    \end{equation}
    Furthermore,
    \begin{equation}\label{eq:O(b1,b2)}
        \mathcal{O}_{\omega}(\mathfrak{b}_1,\mathfrak{b}_2) := C_{\omega,F} + \sum_{i=1}^2 \sum_{\substack{\mathfrak{p} \textnormal{ prime} \\ \mathfrak{p} \mid \mathfrak{a}_i}} D_{i, \omega}(\mathfrak{p}) \frac{\log{N(\mathfrak{p})}}{N(\mathfrak{p})} + \sum_{\substack{\mathfrak{p} \textnormal{ prime} \\ \mathfrak{p} \mid \mathfrak{b}}} D_{3, \omega}(\mathfrak{p}) \frac{\log{N(\mathfrak{p})}}{N(\mathfrak{p})},
    \end{equation}
    where $C_{\omega,F} \ll_{F} 1$ and $D_{i, \omega}(\mathfrak{p}) \ll 1$ for $i \in \{1, 2, 3\}$. Finally, there exists an absolute constant $H \geq 1$
    such that
    \begin{equation} \label{Rbound}
        \mathcal{R}_{\omega}(\mathfrak{b}_1,\mathfrak{b}_2) \ll_{F, \varepsilon}  \big(B Y (1+|\omega|) \big)^H \cdot X^{1-\frac{1}{70}+\varepsilon}.
    \end{equation}
\end{prop}

The proofs of Propositions~\ref{prop:first_moment} and \ref{prop:second_moment} are given in Sections~\ref{sec:first_moment} and \ref{sec:second_moment}, respectively. Finally, to deal with the $\mathcal{S}_R$ terms we will also need the following estimates.

\begin{prop}\label{SRestimate1}
    Let $\omega \in \mathbb{Z}$, $1 \le Y, U \le X^{100}$, and $\varepsilon>0$. Then
    \begin{align*}
        & \mathcal{S}_R (|B_{\f,U}(q)|;F ) + \mathcal{S}_R (|\widetilde{B}_{\f,U^{-1}}(q)|; F ) \nonumber \\
        &  \ll_{\ep} ((1+|\omega|)X)^{\ep} \Big( \bbone_{\omega=0} \cdot U^{1/2} X^{1/2}+(1+|\omega|)^{1/2} \frac{X^{3/4}}{Y^{1/4}}+(1+|\omega|)^{1/3} \frac{X}{Y^{5/6}} \Big).
    \end{align*}
\end{prop}

\begin{prop} \label{SRestimate2}
    Let $\omega \in \mathbb{Z}$, $1 \leq Y \leq X^{100}$, and $\varepsilon > 0$. Then
    \begin{equation*} 
        \mathcal{S}_R(|A_{\omega}(q)|; F) \ll_{\varepsilon} ((1+|\omega|)X)^{\varepsilon} \Big(  (1+|\omega|) X^{3/4} +  (1+|\omega|)^{2/3} \frac{X}{Y^{2/3}} \Big).
    \end{equation*}
\end{prop}

Proofs of Propositions~\ref{SRestimate1} and \ref{SRestimate2} are given in \cref{sec:S_R_bounds}. To prove \cref{mainthm}, we choose a mollifier in \cref{mollifier_section} and assemble the four results above to obtain asymptotics for the first and second mollified moments. In particular, we obtain power-saving asymptotics for the second moment of $L(1/2, \nu_{q, \omega})$, which is of independent interest.

\begin{corollary}\label{thm:second_moment}
    There exists $\delta > 0$ such that the following holds. For every $\omega \in \Z$ and smooth $F: \R \to [0,1]$ with compact support in $(1, 2)$, we have
    \begin{equation*}
        \sum_{\substack{q \in \Z[i] \\ q \equiv 1 \pmod{\lambda^7} \\ N(q) \leq X}} \mu^2(q) |L(1/2, \nu_{q, \omega})|^2 F\Big(\frac{N(q)}{X}\Big) = 2 D_{\omega}\check{F}(0)X \big({\log{X} + D_{F, \omega}} \big) +O_{F, \omega}(X^{1-\delta})
    \end{equation*}
    as $X \to \infty$, where $D_\omega$ is given in \eqref{eq:D_omega} and $D_{F, \omega} = 2\log(1+|\omega|) + C_{F, \omega}$ with $C_{F, \omega} \ll_F 1$.
\end{corollary}

\begin{proof}
    By \cref{lem:afe_2} we have $|L(1/2, \nu_{q, \omega})|^2 = 2A_\omega(q)$, where $A_\omega(q)$ is defined in \eqref{Aomegadef}. Then apply \eqref{nonnegproperty}, and treat the two resulting terms using Propositions \ref{prop:second_moment} and \ref{SRestimate2} with $\mathfrak{b}_1 = \mathfrak{b}_2 = (1)$ and $Y = X^\theta$, where $\theta = \frac{1}{100 H}$ for $H \geq 1$ given in \eqref{Rbound}.
\end{proof}

An analogous result holds for the first moment, using Propositions \ref{prop:first_moment} and \ref{SMexpand2}.

\subsection*{Acknowledgments}

We thank Alexandra Florea, Peter Koymans, and Alexander Smith for helpful comments on an earlier draft of this paper. Cruz Castillo was partially supported by the Alfred P.\ Sloan Foundation’s MPHD Program and by the NSF Graduate Research Fellowship Grant DGE 21-46756. Alexander Dunn was supported by the NSF Standard Grant DMS-2452303, an AMS-Simons Travel Grant, and the Richard A.\ Duke Endowed Fund at the Georgia Tech School of Mathematics.


\section{Conventions} \label{conventions}

For $n \in \mathbb{N}$ and $N>0$, the notation $n \sim N$ means $N \leq n < 2N$, and $n \asymp N$ means that there exist constants $c_1,c_2>0$ such that $c_1 N \leq n \leq c_2 N$. Dependence of implied constants on parameters will be indicated (as subscripts of the notation) in statements of results, but suppressed throughout the body of the paper (i.e.\ in proofs). In particular, implied constants in the body of the paper are allowed to depend on $\varepsilon>0$ (which is possibly different in each instance) and on the implicit constants in $\asymp$ or $\ll$ notation.

We use the special notation $f \ls g$ to denote that there exists an absolute constant $H \geq 0$ such that
\begin{equation} \label{specnotation}
    |f| \ll_\varepsilon X^{\varepsilon} \cdot \big(B Y (1+|\omega|)\big)^{H} \cdot |g| \qquad \text{for all } \varepsilon > 0.
\end{equation}

Every ideal $0 \neq \mathfrak{n} \unlhd \mathbb{Z}[i]$ can be uniquely decomposed as $\mathfrak{n}=\lambda^k c \mathbb{Z}[i]$ with $k \in \mathbb{Z}_{\geq 0}$ and $c \equiv 1 \pmod{\lambda^3}$. If $\mathfrak{n}$ is coprime with $\lambda$, we call $c$ its \emph{primary} generator. We pass between ideals and their generators freely in this paper.

Given $0 \neq \mathfrak{d},\mathfrak{n} \unlhd \mathbb{Z}[i]$, the notation $\mathfrak{d} \mid \mathfrak{n}$ means there exists $\mathfrak{a} \unlhd \mathbb{Z}[i]$ such that $\mathfrak{n}=\mathfrak{a} \mathfrak{d}$. Similarly, given $0 \neq d, n \in \Z[i]$, the notation $d \mid n$ means $(d) \mid (n)$. For $a, b \equiv 1 \pmod{\lambda^3}$, the notation $a \mid b^{\infty}$ means that if $\pi \equiv 1 \pmod{\lambda^3}$ is prime and $\pi \mid a$, then $\pi \mid b$.


\section{Overview of the argument} \label{hlsketch}

Here we sketch our argument for power-saving asymptotics of the mollified first and second moments. Recall that $\xi$ is a Hecke character with frequency $\omega \in \Z$. For simplicity, in this sketch we fix $\omega$, assume coprimality of all relevant variables, suppress smooth functions, ignore units and powers of the ramified prime $\lambda$ in $\mathbb{Z}[i]$, and remove congruence conditions with fixed modulus.

\subsection{Second mollified moment}

We first use the (balanced) approximate functional equation of \cref{lem:afe_2} for $|L(1/2,\nu_{q,\omega})|^2$, which has root number $1$, and then sieve out the squarefree condition on $q$ using \eqref{nonnegproperty}. The error term $\mathcal{S}_R (|A_{\omega}(q)|;F )$ from the sieve is estimated in \cref{SRestimate2} using a technical but standard application of the quartic large sieve \cite{BGL}, which we omit from this discussion. Its contribution is negligible as long as $\log{Y} \gggtr \log{B} \gg \log{X}$.

It then suffices to obtain asymptotics for $\mathcal{S}_M \big(\nu_{q,\omega}(\mathfrak{b}_1)\overline{\nu_{q,\omega}(\mathfrak{b}_2)}A_{\omega}(q);F \big)$ with power-saving error terms, as in \cref{prop:second_moment}. Recall that $Y$ is the parameter in the sieve \eqref{MYRYdef}, and that $N(\mathfrak{b}_1), N(\mathfrak{b}_2) \ll B$. We eventually choose $Y$ and $B$ to be arbitrarily small (but fixed) powers of $X$, hence factors of the form $(BY)^{O(1)}$ will be negligible. The special notation $f \ls g$ given in \eqref{specnotation} removes the need to track such factors.

We use Poisson summation on the sum over $q$ and remove the main term (of magnitude $X \log{X}$), corresponding to the frequency $r=0$. The remaining expression is
\begin{equation} \label{start}
     \ls  \sum_{\substack{0 \neq r \in \mathbb{Z}[i] \\ N(r) \ls 1}} \Big|{ \sum_{\substack{n_1 \in \mathbb{Z}[i] \\ n_1 \equiv 1 \pmod{\lambda^3} \\ N(n_1) \sim N_1 }}  \xi(n_1) \widetilde{g}_4(-r, n_1) } \Big| \cdot \Big| \sum_{\substack{ n_2 \in \mathbb{Z}[i] \\ n_2 \equiv 1 \pmod{\lambda^3} \\ N(n_2) \sim N_2 }} \xi(n_2) \widetilde{g}_4(r, n_2) \Big|, 
\end{equation}
where $\widetilde{g}_4(\mu,c)$ is the normalized quartic Gauss sum (typically of absolute value $1$) defined in \eqref{eq:normalized_gauss_sum}, and the ranges $N_1, N_2 \gg 1$ satisfy $X \ls N_1 N_2 \ls  X$.

We need to understand twisted sums of quartic Gauss sums. Performing Perron summation and a contour shift to the critical line, we pass over (at most) a simple pole at $s=\frac{3}{4}$ (and only when $\omega=0$). For $0 \neq a \in \Z[i]$, this gives
\begin{align} \label{pattersonsketch}
    \sum_{\substack{n \in \mathbb{Z}[i]  \\ n \equiv 1 \pmod{\lambda^3} \\ N(n) \sim N }} \xi(n) \widetilde{g}_4(a, n) \approx \bbone_{\omega=0} \cdot  \tau_4(a) \cdot N^{3/4} + \int_{\mathcal{C}_{\varepsilon}} \widetilde{\psi}(a,s,\xi)  N^{s} \frac{ds}{s}, 
\end{align}
where $\mathcal{C}_{\varepsilon}$ denotes the line segment $\Re(s) = \frac{1}{2} + \varepsilon$ and $\left|\Im(s)\right| \leq X^\varepsilon$, while
\begin{equation*}
    \widetilde{\psi}(a,s,\xi):=\sum_{\substack{c \in \mathbb{Z}[i] \\ c \equiv 1 \pmod{\lambda^3} }} \frac{\xi(c) \widetilde{g}_4(a,c)}{N(c)^s} \qquad \text{and} \qquad \tau_4(a):=\Res_{s=\frac{3}{4}} \widetilde{\psi}(a,s,\mathbf{1}).
\end{equation*}
 
Crucially, Kubota \cite{Kub} showed that $\widetilde{\psi}(a,s,\xi)$ has meromorphic continuation to $s \in \mathbb{C}$ and satisfies a functional equation, despite not having an Euler product. This leads to the convexity bound (of $\GL_1$ type in $a$)
\begin{equation} \label{convexsketch}
\widetilde{\psi} (a, \tfrac{1}{2} + \varepsilon + it, \xi ) \ll_{\omega, \varepsilon} N(a)^{\frac{1}{4}+\varepsilon} \cdot (1+|t|)^{\frac{3}{4}+\varepsilon}.
\end{equation}

The quantity $\tau_4(a) N(a)^{1/8}$ is the $a$-th Fourier coefficient of Kubota's quartic theta function \cite{Kub}. Finding a general closed formula for the $\tau_4(a)$ is one of the central open problems in the theory of metaplectic forms \cite{Del80}. This is in stark contrast to the cubic case, where Patterson \cite{Pat1} used a tour de force Hecke converse argument to compute the Fourier coefficients of the cubic theta function, obtaining essentially cubic Gauss sums. Patterson's exact evaluation \cite{Pat1} was fully exploited by the second and third authors, together with David and Stucky \cite{DDDS24}, to obtain asymptotics for the second moment of cubic Hecke $L$-functions with (nearly optimal) error term $O_{\ep}(X^{5/6+\ep})$.

In the quartic case, Patterson \cite{Pat2}, with refinements by Eckhardt--Patterson \cite[Conjecture~2.11]{EckPat}, conjectured that $\tau_4(\pi)^2 N(\pi)^{1/4}$ is proportional to $\overline{\widetilde{g}_4(\pi)}$ for $\pi$ prime. Outside of a handful of special relations \cites{Suz1, Suz93}, nothing more is known about the exact evaluation of $\tau_4(a)$. However, the convexity bound for $\widetilde{\psi}(a, s, \xi)$ implies 
\begin{equation} \label{convexthetasketch}
    \tau_4(a) \ll_\varepsilon N(a)^{1/8+\varepsilon}.  
\end{equation} 

The lack of knowledge in the quartic and higher order cases is concerning, but fortunately we only need to consider fairly small twists (with $N(r) \ls 1$) in \eqref{start}. This allows us to establish a small (but fixed) power-saving error term for $\mathcal{S}_M \big(\nu_{q,\omega}(\mathfrak{b}_1)\overline{\nu_{q,\omega}(\mathfrak{b}_2)}A_{\omega}(q);F \big)$ using \emph{only} the convex inputs \eqref{convexsketch} and \eqref{convexthetasketch}. Indeed, using \eqref{pattersonsketch} to evaluate both sums of Gauss sums in \eqref{start}, we can apply \eqref{convexsketch} and \eqref{convexthetasketch}, then trivially estimate the sum over $r$. We deduce that \eqref{start} is $\ls (N_1 N_2)^{3/4} \ls X^{3/4}$.

\subsection{Context and limitations for the second moment}\label{sec:context_limitations}

An advantage of the approach above is that it readily applies to higher order characters, as it does not rely on special evaluations of theta coefficients nor subconvex bounds for $\widetilde{\psi}(a, s, \xi)$. Furthermore, it vastly simplifies the exposition, since in reality several auxiliary variables arise in the proof from lack of coprimality, M\"{o}bius inversion, congruence conditions, and decompositions into prescribed prime factorizations. We treat these auxiliary variables in a crude manner, leading to acceptable losses of small powers of $X$.

It is possible (though very technically involved) to bound many of the auxiliary variables optimally, as was done in \cite{DDDS24} for cubic characters. Moreover, while we do not have optimal pointwise bounds for $\tau_4(a)$ and $\widetilde{\psi}(a, s, \xi)$, it is possible to prove optimal bounds \emph{on average} for both quantities. These average bounds are as good as the pointwise ones for the application at hand. However, this would still not be enough to obtain optimal error terms in the second moment, for two reasons. First, \cite{DDDS24} also requires an alternative argument (via the large sieve) exploiting the twisted multiplicativity of $\tau_3$, which is not available for $\tau_4$ due to our inability to determine it. Finally, the quartic large sieve is likely not optimal in its current form, unlike its cubic counterpart \cites{DR}. For those reasons, while the error terms resulting from a careful treatment of auxiliary variables would likely be quite good, they would \emph{not} capture a meaningful threshold (such as the conjectured lower order main term $X^{3/4} \log{X}$). 

\subsection{First mollified moment}

Let $U \geq 1$ be a parameter. We first use the approximate functional equation of \cref{lem:afe_2} (with unbalanced lengths $U \sqrt{X}$ and $\sqrt{X}/U$) for $L(1/2,\nu_{q,\omega})$, and then sieve out the squarefree condition on $q$ using \eqref{nonnegproperty}. The error terms from sieving are estimated in \cref{SRestimate1}, using the quartic large sieve as before.

The remaining terms are sums of the form
\begin{equation} \label{term1}
    \sum_{\substack{0 \neq n \in \Z[i] \\ N(n) \ll U \sqrt{X} }} \frac{\xi(bn)}{N(n)^{1/2}} \sum_{\substack{\ell \in \Z[i]  \\ N(\ell) \sim L}} \mu(\ell) \chi_{\ell^2}(bn) \sum_{ \substack{ m \in \Z[i] \\ N(m) \sim X/L^2 }} \chi_{m}(b n), 
\end{equation}
where $1 \ll L \ll Y$ and $N(b) \ll B$, and the dual sum
\begin{equation} \label{term2}
    \ep(\f) \xi(b) \sum_{\substack{0 \neq n \in \Z[i] \\ N(n) \ll \sqrt{X}/U}} \frac{\overline{\xi(n)}}{N(n)^{1/2}} \sum_{\substack{ m \in \Z[i]  \\ N(m) \sim X}}  \xi(m) \widetilde{g}_4(b^3n,m).
\end{equation}
Like before, we choose $Y$ and $B$ to be arbitrarily small (but fixed) powers of $X$.

To handle \eqref{term1}, we first apply Poisson summation in $m$ and extract the zero frequency, which gives the main term. The length of the remaining sum over frequencies $k \neq 0$ is 
\begin{equation*}
    0 \neq N(k) \ll \frac{B U \sqrt{X}}{X/L^2} \ll \frac{B Y^2 U}{\sqrt{X}}.
\end{equation*} 
Note that $B Y^2 U/\sqrt{X} = o(1)$ whenever $BY^2U \leq X^{1/2-\nu}$ for $0<\nu<\frac{1}{2}$ fixed, explaining the constraint \eqref{firstmomentassump}.
Hence \eqref{term1} gives a negligible contribution to the error term.

To handle \eqref{term2}, we apply \eqref{pattersonsketch}, \eqref{convexsketch}, and \eqref{convexthetasketch} to the sum over $m$, and then estimate the sum over $n$ trivially. We deduce that \eqref{term2} is 
\begin{align*}
& \ls \sum_{\substack{0 \neq n \in \Z[i] \\ N(n) \ll \sqrt{X}/U }} \frac{1}{N(n)^{1/2}} \Big( \bbone_{\omega=0} \cdot N(n)^{1/8} X^{3/4}+N(n)^{1/4} X^{1/2} \Big) \ls \bbone_{\omega=0} \cdot \frac{X^{17/16}}{U^{5/8}}+\frac{X^{7/8}}{U^{3/4}}.
\end{align*}
This above display is acceptable if $U \gg X^{1/10+\delta}$ for any fixed $\delta > 0$, as $B$ and $Y$ are small powers of $X$.


\section{Preliminaries}

\subsection{Arithmetic functions}

Let $\Q(i)$ be the Gaussian imaginary quadratic field with ring of integers $\Z[i]$, and $N(x):=N_{\Q(i)/\Q}(x) = |x|^2$ be its norm. Let $\mu(n)$ denote the M{\"o}bius function and $\varphi(n)$ denote the Euler function on $\mathbb{Z}[i]$. For $0 \neq n \in \Z[i]$ let
\begin{equation*}
	\mathrm{rad}(n) = \lambda^{\bbone_{\lambda \mid n}} \cdot \prod_{\substack{\pi \text{ prime},\ \pi \mid n \\ \pi \equiv 1 \pmod{\lambda^3}}} \pi.
\end{equation*}

\subsection{Quartic residue symbol and reciprocity}

Let $\pi$ be a prime with $(\pi) \ne (\lambda)$. If $\pi \nmid \alpha$, the quartic residue symbol is defined as the unique $\big(\frac{\alpha}{\pi}\big)_4 \in \{\pm 1, \pm i\}$ satisfying
\begin{equation}\label{eq:quartic_sym}
    \quartic{\alpha}{\pi} \equiv \alpha^{\frac{N(\pi)-1}{4}}\pmod \pi.
\end{equation}
For $\alpha,\gamma\in \Z[i]$ with $(\alpha,\gamma) =1$ and $(\gamma,\lambda)=1$ we define
\begin{equation}
    \quartic{\alpha}{\gamma} := \prod_{\pi|\gamma}\quartic{\alpha}{\pi}.
\end{equation}
When $(\alpha,\gamma)\ne 1$, we take the symbol to be zero. Similarly, the quadratic residue symbol is given by $\big(\frac{\alpha}{\gamma}\big)_2  := \big(\frac{\alpha}{\gamma}\big)^2_4$. 

The law of biquadratic reciprocity \cite[\S 9.9, Theorem 2]{IR90} states that for $\alpha,\gamma \in \Z[i]$ with $\alpha \equiv \gamma \equiv 1 \pmod{\lambda^3}$ and $(\alpha,\gamma)=1$ we have
\begin{equation} \label{eq:quar_repr}
    \quartic{\alpha}{\gamma} = (-1)^{C(\alpha,\gamma)}\quartic{\gamma}{\alpha} , \quad \text{where} \quad C(\alpha,\gamma) = \frac{(N(\alpha)-1)}{4}\frac{(N(\gamma)-1)}{4}.
\end{equation}
There are supplementary laws for the ramified primes and units.
If
\begin{equation*}
    \gamma=1+a_4 \lambda^4+a_5 \lambda^5+a_6 \lambda^6+\cdots \qquad \text{with} \qquad a_i \in \{0,1\},
\end{equation*}
then
\begin{equation}\label{eq:ram&units_1}
    \Big(\frac{\lambda}{\gamma} \Big)_4=i^{-a_4+2a_6} \qquad \text{and} \qquad \Big( \frac{i}{\gamma} \Big)_4=(-1)^{a_4+a_5}.
\end{equation}
Otherwise, if
\begin{equation*}
    \gamma=1+\lambda^3+a_4 \lambda^4+a_5 \lambda^5+a_6 \lambda^6+\cdots \qquad \text{with} \qquad a_i \in \{0,1\},
\end{equation*}
then
\begin{equation}\label{eq:ram&units_2}
    \Big( \frac{\lambda}{\gamma} \Big)_4=-i^{-a_4+2a_6} \qquad \text{and} \qquad \Big( \frac{i}{\gamma} \Big)_4=i (-1)^{a_4+a_5}.
\end{equation}

\subsection{Gauss sums}

Let $e(x) := e^{2\pi ix}$ for $x \in \R$ and $\ec(z):=e(z+\bar{z})$ for $z \in \C$. For $c \in \Z[i]$ with $(c,\lambda) =1$ and $\nu \in \lambda^{-2} \mathbb{Z}[i]$,
the quartic and quadratic Gauss sums over $\Z[i]$ are defined respectively as
\begin{equation} \label{gaussdef}
    g_4(\nu,c) := \sum_{d \pmod c} \quartic{d}{c}\ec\Big(\frac{\nu d}c\Big) \qquad \text{and} \qquad g_2(\nu,c):= \sum_{d \pmod c}\quadrat{d}{c} \ec\Big(\frac{\nu d}c\Big).
\end{equation}
When $\nu =1$, we simply write $g_4(c)$ and $g_2(c)$, respectively. Changing variables and applying the Chinese remainder theorem and quartic reciprocity, we obtain the following.

\begin{lemma}\label{quartic-GS-lemma1}
    Let $c, c_1, c_2 \equiv 1 \pmod{\lambda^3}$, $\nu \in \mathbb{Z}[i]$, and $\mu \in \lambda^{-2} \mathbb{Z}[i]$.
    \begin{enumerate}[leftmargin=*]
        \item[(i)] \label{changevar}  If $(\nu, c)=1$, then
              \begin{equation*}
                  g_4(\nu \mu, c) = \overline{\chi_{c}(\nu)} g_4(\mu,c).
              \end{equation*}
        \item[(ii)] \label{twistmult}  If $(c_1, c_2)=1$, then
              \begin{equation*}
                  g_4(\mu, c_1 c_2) = \chi_{c_1}(c_2) \chi_{c_2} (c_1) g_4(\mu,c_1) g_4(\mu, c_2) = (-1)^{C(c_1,c_2)} g_4(c_2^2 \mu, c_1) g_4(\mu , c_2).
              \end{equation*}
    \end{enumerate}
\end{lemma}
In order to compute $g_4(\mu,c)$ for general parameters $\mu \in \lambda^{-2}\Z[i]$ and $c \equiv 1 \pmod{\lambda^3}$, by \cref{quartic-GS-lemma1} it suffices to compute $g_4(\pi^{k},\pi^{\ell})$ for $\pi \equiv 1 \pmod{\lambda^3}$ prime
and $k,\ell \in \mathbb{Z}_{\geq 0}$.
\begin{lemma} \label{localcomp}
    Let $k,\ell \in \mathbb{Z}_{\geq 0}$ and
    $\pi \in \mathbb{Z}[i]$ be prime and satisfy $\pi \equiv 1 \pmod{\lambda^3}$. Then
    \begin{equation} \label{rel4}
        g_4(\pi^k,\pi^{\ell})=
        \begin{cases}
            1 & \text{if } \ell=0,\\
            N(\pi)^k g_4(\pi) & \text{if } \ell=k+1, \quad k \equiv 0 \pmod{4},\\
            N(\pi)^k g_2(\pi) & \text{if } \ell=k+1, \quad k \equiv 1 \pmod{4}, \\
            N(\pi)^k \chi_{\pi}(-1) \overline{g_4(\pi)} & \text{if } \ell=k+1, \quad k \equiv 2 \pmod{4}, \\
            -N(\pi)^k & \text{if } \ell=k+1, \quad k \equiv 3 \pmod{4}, \\
            \varphi(\pi^{\ell}) & \text{if } k \geq \ell, \quad \ell \equiv 0 \pmod{4}, \\
            0 & \text{otherwise}.
        \end{cases}
    \end{equation}
\end{lemma}

If $\pi \equiv 1 \pmod{\lambda^3}$ is a degree $2$ prime in $\mathbb{Z}[i]$, i.e.\ satisfies $N(\pi)=p^2$ for a rational prime $p \equiv 3 \pmod{4}$, then
\begin{equation} \label{trivquadquar}
    \left|g_4(\pi)\right|=\left|g_2(\pi)\right|=p=N(\pi)^{1/2}.
\end{equation}
If $\pi \equiv 1 \pmod{\lambda^3}$ is a degree $1$ prime in $\mathbb{Z}[i]$, then
Gauss showed that
\begin{equation} \label{quad2}
    g_2(\pi)=\Big(\frac{-1}{\pi} \Big)_4 N(\pi)^{1/2},
\end{equation}
and in  that case we also have the fourth power formula \cite[Proposition~9.9.5]{IR90}
\begin{equation} \label{fourthpower}
    g_4(\pi)^4= \pi^3 \overline{\pi}
\end{equation}
and the square formula \cite[Proposition~9.10.1]{IR90}
\begin{equation} \label{squarepower}
    g_4(\pi)^2=-\Big( \frac{-1}{\pi} \Big)_4\Big( \frac{\overline{\pi}}{\pi}  \Big)_4^{-2}  N(\pi)^{1/2} \pi.
\end{equation}
Note that \cref{quartic-GS-lemma1}, \eqref{rel4}, \eqref{trivquadquar}, and \eqref{fourthpower}
imply that for all $a, c \equiv 1 \pmod{\lambda^3}$ we have
\begin{equation} \label{sqrootcancel}
    |g_4(c)|=\mu^2(c) \sqrt{N(c)}
\end{equation}
and
\begin{equation} \label{nicebd}
    |g_4(a, c)| \leq \sqrt{N(ac)}.
\end{equation}

For $c \in \mathbb{Z}[i]$ with $(c,\lambda)=1$ and $\nu \in \lambda^{-2} \mathbb{Z}[i]$,
the normalized quartic Gauss sum over $\mathbb{Z}[i]$
is defined as 
\begin{equation}\label{eq:normalized_gauss_sum}
    \widetilde{g}_4(\nu,c) := \frac{g_4(\nu,c)}{\sqrt{N(c)}} = \frac{1}{\sqrt{N(c)}} \sum_{d \pmod{c}} \Big( \frac{d}{c} \Big)_4 \check{e} \Big( \frac{\nu d}{c} \Big) \qquad \text{and} \qquad \widetilde{g}_4(c) := \widetilde{g}_4(1, c).
\end{equation}

We need to consider slightly more general exponential sums that are the finite Fourier transforms of quartic Hecke characters (not necessarily primitive).

\begin{lemma} \label{h3fourier}
    For $c=c_1 c^2_2c_3^3 \in \mathbb{Z}[i]$ with $c_1,c_2,c_3 \in \mathbb{Z}[i]$ satisfying $c_1,c_2,c_3 \equiv 1 \pmod{\lambda^3}$ and $\mu^2(c_1c_2)=1$, and for $\mu \in \lambda^{-2}\Z[i]$, define
    \begin{equation} \label{hdef}
        \widetilde{h}_4(\mu,\chi_c):=\frac{1}{N(c_1 c_2c_3)^{1/2}} \sum_{\substack{ x \pmod{c_1 c_2c_3} \\ (x,c_1 c_2 c_3)=1}} \chi_c( x) \check{e} \Big( \frac{\mu x}{c_1 c_2 c_3} \Big).
    \end{equation}

    \begin{enumerate}[leftmargin=*]
        \item[(i)] If $c_1,c_2,c_3$ are pairwise coprime, then
              \begin{equation*}
                  \widetilde{h}_4(\mu,\chi_c)=\chi_{c_1}(c_2c_3)\chi_{c_2}^2(c_1c_3)\overline{\chi_{c_3}(c_1c_2)}\widetilde{g}_4(\mu,c_1)\widetilde{g}_2(\mu,c_2)\overline{\widetilde{g}_4(\mu,c_3)}.
              \end{equation*}

        \item[(ii)] If $\nu \in \mathbb{Z}[i]$ and $(\nu,c)=1$, then
              \begin{equation*}
                  \widetilde{h}_4(\mu \nu,\chi_c)=\overline{\chi_c(\nu)} \: \widetilde{h}_4(\mu,\chi_c).
              \end{equation*}
    \end{enumerate}
\end{lemma}

\begin{proof}
    Write $\chi_c= \chi_{c_1}\chi_{c_2}^2\overline{\chi_{c_3}}$. Then $(i)$ follows from applying the Chinese remainder theorem to \eqref{hdef}, while $(ii)$ follows from the change of variable $x \mapsto \nu x$ in \eqref{hdef}.
\end{proof}

\subsection{Hecke characters over $\Q(i)$}

\subsubsection{Generalities}

Let $I$ be the multiplicative group of non-zero fractional ideals in $\Q(i)$. For $\mathfrak{m} \unlhd \Z[i]$ denote
\begin{equation*}
    I_{\m}:=\{\mathfrak{n} \in I: (\mathfrak{n},\m)=1\},
\end{equation*}
where $(\mathfrak{n}, \m)=1$ means $\mathfrak{n} = \mathfrak a \mathfrak b^{-1}$ for $\mathfrak a, \mathfrak b  \unlhd \Z[i]$ with $(\mathfrak a \mathfrak b, \m)=1$. A (unitary) Hecke character of $\Q(i)$ is a continuous homomorphism $\psi: I_{\m} \to S^1$ for which there exist two characters
\begin{equation}\label{eq:classical_characters}
    \chi_{\infty}(\alpha) = \Big(\frac{\overline \alpha}{|\alpha|}\Big)^{\omega} \text{ for } \alpha \in \C^\times \qquad \text{and} \qquad \chi: (\Z[i]/\mathfrak{m})^\times\to S^1
\end{equation}
such that for every $\alpha \in \Z[i]$ with $(\alpha,\mathfrak{m})=1$ we have
\begin{equation}
    \psi((\alpha)) = \chi(\alpha)\chi_{\infty}(\alpha).
\end{equation}
We call $\omega$ the \emph{frequency} and $\m$ the \emph{modulus} of $\psi$. Since $\psi((u))= \psi((1)) =  1$ for any unit $u \in \Z[i]$, we necessarily have
\begin{equation}\label{eq:unit_condition}
    \chi(u)\chi_{\infty}(u) =1.
\end{equation}
Any choice of $\chi$ and $\chi_\infty$ satisfying \eqref{eq:unit_condition} defines a (unique) Hecke character $\psi$. The \emph{conductor} of $\psi$ is the smallest divisor $\mathfrak{f}$ of $\m$ such that $\psi$ is the restriction to $I_\m$ of a Hecke character of modulus $\mathfrak{f}$. If $\mathfrak{f} = \m$ we call $\psi$ \emph{primitive}.

\subsubsection{Quartic characters}

For $q \in \Z[i]$ with $q \equiv 1 \pmod{\lambda^3}$, define the quartic Dirichlet character
\begin{equation}\label{eq:quar_char}
    \chi_q(\alpha) := \quartic{\alpha}{q}.
\end{equation}
The supplements to quartic reciprocity show that this is a (quartic) Hecke character (with frequency zero), i.e.\ trivial on units, precisely when $N(q) \equiv 1 \pmod{16}$. In that case the notation $\chi_q(\mathfrak{n})$ for $\mathfrak{n} \unlhd \Z[i]$ is unambiguous. 

We may (uniquely) decompose $q = q_1 q_2^2 q_3^3 q_4^4 q_5^4$ for $q_i \equiv 1 \pmod {\lambda^3}$ with $\mu^2(q_1q_2q_3q_4)=1$ and $q_5 \mid (q_1 q_2 q_3 q_4)^\infty$, so that $\chi_q = \chi_{q_1}\chi_{q_2}^2\overline{\chi}_{q_3}\mathbf{1}_{q_4}$, where $\mathbf{1}_{q_4}$ is the trivial character modulo $q_4 \Z[i]$. The conductor of $\chi_q$ is then $\ra(q) \Z[i]$, and $\chi_q$ is primitive if and only if $q_4 =1$. Hence a quartic Dirichlet character $\chi$ is a primitive Hecke character if and only if $\chi = \chi_q$ for some $q \in \mathcal{C}_4$, where
\begin{equation*}
    \mathcal{C}_4:= \big\{q_1q_2^2q_3^3: q_i\in \Z[i], q_i \equiv 1 \pmod{\lambda^3}, \mu^2(q_1q_2q_3) =1, \ \text{and} \ N(q_1q_2^2q_3^3) \equiv 1 \pmod{16}\big\}.
\end{equation*}
We will often restrict to the family
\begin{equation}\label{eq:fam}
    \mathcal{F}_4:= \big\{q_1q_2^2q_3^3: q_i\in \Z[i], q_i \equiv 1 \pmod{\lambda^3}, \mu^2(q_1q_2q_3) =1, \ \text{and} \ q_1q_2^2q_3^3 \equiv 1 \pmod{\lambda^7}\big\} \quad
\end{equation}
and its subfamily
\begin{equation}\label{eq:subfam}
    \mathcal{F}'_4:= \big\{1 \neq q\in \Z[i]: \mu^2(q) =1\  \text{and} \ q \equiv 1 \pmod{\lambda^7}\big\}.
\end{equation}

\subsubsection{Twists}

We are also interested in twisting $\chi_q$ at the infinite place. For that we may use the characters $\xi$ described in \cite[\S 3, Example 1]{IK}. First, fix an integer $\omega$  and set\footnote{We correct a typo in \cite[Example 1]{IK} when $\omega \equiv 2 \pmod{4}$.}
\begin{equation*}
    \mathfrak{m}_{\omega}:= \begin{cases}
        (\lambda^3) \quad &\text{if} \quad \omega \equiv 1 \text{ or } 3 \pmod{4},\\
        (2) \quad &\text{if} \quad \omega \equiv 2 \pmod{4},\\
        (1)\quad & \text{otherwise.}
    \end{cases}
\end{equation*}
Consider the Hecke character $\xi:I_{\m_{\omega}}\to \C^\times$ given by
\begin{equation}\label{eq:inf_hecke}
    \xi(\mathfrak{n}) = \Big(\frac{\overline{n}}{|n|}\Big)^{\omega},
\end{equation}
where $n$ is the canonical generator of $\mathfrak{n}$, i.e.\ $n = \lambda^k n'$ for $k \in \Z_{\geq 0}$ (necessarily $k=0$ if $4 \nmid \omega$) and $n' \equiv 1 \pmod{\lambda^3}$. We define $\xi$ to be zero on ideals not coprime to $\mathfrak m_{\omega}$. The character $\xi$ is induced by the characters $(\chi_{\xi},\xi_{\infty})$ given by
\begin{equation}\label{eq:xi_classical_decomp}
    \chi_{\xi}(n) = u^\omega \text{ for } (n, \m_\omega)=1 \qquad \text{and} \qquad \xi_{\infty}(n) = \Big(\frac{\overline n}{|n|}\Big)^{\omega} \text{ for } n\in \C^\times,
\end{equation}
where $u \in \Z[i]$ is any unit such that $n \equiv u \pmod{\m_\omega}$. Note that $\chi_\xi$ is the trivial character when $4 \mid \omega$. We will write $\xi(n)$ for $\xi((n))$ when $n \equiv 1 \pmod{\lambda^3}$.

For $q \in \Z[i]$ with $q \equiv 1 \pmod{\lambda^7}$ we define
\begin{equation}\label{eq:nu_q}
    \nu_{q,\omega}(\mathfrak{n}) := \chi_q(\mathfrak{n}) \xi(\mathfrak{n}) = \Big(\frac{\mathfrak{n}}{q}\Big)_4 \xi(\mathfrak{n}),
\end{equation}
which is a Hecke character of frequency $\omega$ and modulus $\ra(q)\mathfrak{m_{\omega}}$. If in addition $q\in \mathcal{F}_4$, then $\nu_{q, \omega}$ is primitive.

\subsection{Hecke \texorpdfstring{$L$}{}-functions over \texorpdfstring{$\Q(i)$}{}}

Let $\mathfrak{m} \unlhd \Z[i]$ and let $\psi \pmod{\mathfrak{m}}$ be a Hecke character of $\Q(i)$ with frequency $\omega$. Thus there exist two characters
\begin{equation}
    \chi_{\infty}(\alpha) = \Big(\frac{\overline \alpha}{|\alpha|}\Big)^{\omega} \text{ for }  \alpha \in \C^\times \qquad \text{and} \qquad \chi: (\Z[i]/\mathfrak{m})^\times\to S^1
\end{equation}
such that for every $\alpha \in \Z[i]$ with $(\alpha,\m)=1$ we have $ \psi((\alpha)) = \chi(\alpha)\chi_{\infty}(\alpha)$.

The Hecke $L$-function attached to $\psi$ is given by
\begin{equation} \label{Lspsi}
    L(s,\psi):=\sum_{0 \neq \mathfrak{n} \unlhd \mathbb{Z}[i]} \frac{\psi(\mathfrak{n})}{N(\mathfrak{n})^{s}}, \qquad \Re(s)>1.
\end{equation}
Note that we put $\psi(\mathfrak{n})=0$ whenever ($\mathfrak{n}, \mathfrak{m})\neq 1$. Let $\mathfrak{c}_{\psi} \unlhd \mathbb{Z}[i]$ denote the conductor of $\psi$ and $d_{\Q(i)} = -4$ denote the discriminant. The local factor at the infinite place is
\begin{equation} \label{eq:Lspsiinf}
    L_{\infty}(s,\psi):= (|d_{\mathbb{Q}(i)}|N(\mathfrak{c}_\psi))^{s/2}(2\pi)^{-s}\Gamma(s+\tfrac{1}{2} |\omega|),
\end{equation}
and the completed Hecke $L$-function of $\psi$ is 
\begin{equation} \label{completed}
    \Lambda(s,\psi):=L_{\infty}(s,\psi)L(s,\psi).
\end{equation}

\begin{prop}[{\cite[Theorem 3.8]{IK}}] \label{funceq}
    The completed $L$-function $\Lambda(s,\psi)$ is entire, provided that $\psi$ is primitive and non-trivial. Furthermore, it satisfies the functional equation
    \begin{equation*}
        \Lambda(s,\psi)=W(\psi) \Lambda(1-s,\overline{\psi}),
    \end{equation*}
    where if $\chi_\infty$ and $\chi$ denote the components of $\psi$ as in \eqref{eq:classical_characters}, and $\mathfrak{c}_{\psi}=m \mathbb{Z}[i]$, we have
    \begin{equation} \label{Wpsi}
        W(\psi):=\frac{i^{\omega} \chi_{\infty}(2m)}{N(\mathfrak{c}_{\psi})^{1/2}} \sum_{\substack{ x \pmod{\mathfrak{c}_{\psi}} \\ (x,\mathfrak{c}_{\psi})=1 }} \chi(x) \check{e} \Big(\frac{x}{2 m}  \Big).
    \end{equation}
\end{prop}

\begin{remark}
    By \eqref{eq:unit_condition}, $W(\psi)$ is independent of the choice of generator $m\in \Z[i]$ for $\mathfrak{c}_\psi$.
\end{remark}

\begin{lemma}\label{lem:Wpsi}
    Let $\omega \in \Z$, $q= q_1q_2^2q_3^3 \in \mathcal{F}_4$ with $q_i \equiv 1 \pmod{\lambda^3}$ and $\mu^2(q_1 q_2 q_3)=1$, and $\nu_{q,\omega}$ be as defined in \eqref{eq:nu_q}, which has frequency $\omega$ and conductor $q_1q_2q_3\m_{\omega}$. Then
    \begin{equation*}
        W(\nu_{q,\omega}) = \ep(\mathfrak{\omega}) \xi(q_1q_2q_3)\chi_{q_1}(q_2q_3)\chi_{q_2^2}(q_1q_3)\overline{\chi_{q_3}(q_1q_2)}\widetilde{g}_4(q_1)\widetilde{g}_2(q_2)\overline{\widetilde{g}_4(q_3)}
    \end{equation*}
    for
    \begin{equation}\label{eq:r_num_f}
        \ep(\mathfrak{\omega}) = 
        \begin{cases}
            (-1)^{\frac{\omega+1}{2}} (\frac{2}{\omega})_2 & \text{if } 2\nmid \omega, \\
            1 & \text{otherwise}.
        \end{cases}
    \end{equation}
\end{lemma}

\begin{proof}
    We consider the case $2 \nmid \omega$. The other cases are simpler and follow in a similar fashion. First, set $\tilde{q}:= \ra(q) = q_1q_2q_3$. By \eqref{Wpsi}, recalling that $\xi = \chi_\xi \cdot \xi_\infty$ as in \eqref{eq:xi_classical_decomp}, we obtain
    \begin{equation*}
        W(\nu_{q,\omega}) = \frac{i^{\omega}\xi_{\infty}(2\lambda^3\tilde{q})}{N(\lambda^3\tilde{q})^{1/2}} \sum_{\substack{x \pmod{\lambda^3\tilde{q}}\\(x,\lambda^3\tilde{q})=1}}\chi_q(x)\chi_{\xi}(x)\ec \Big({\frac{x}{2\lambda^3 \tilde{q}}}\Big).
    \end{equation*}
    Apply the Chinese remainder theorem and write $x \pmod{\lambda^3\tilde{q}}$ as $x = a\tilde{q} +\lambda^3b$ with $a \pmod{\lambda^3}$ and $b \pmod{\tilde{q}}$, so that
    \begin{equation}\label{eq:W_psi_1}
        W(\nu_{q,\omega}) = \frac{i^{\omega}\xi_{\infty}(2\lambda^3\tilde{q})}{N(\lambda^3\tilde{q})^{1/2}} \Big(\sum_{b \pmod{ \tilde{q}}} \chi_q(\lambda^3 b) \ec\Big(\frac{b}{2 \tilde{q}} \Big)\Big)\Big(\sum_{\substack{a \pmod{\lambda^3}\\(a,\lambda) =1}} \chi_{\xi}(a)\ec\Big(\frac{a}{2\lambda^3}\Big)\Big).
    \end{equation}
    We evaluate each sum on the right hand side of \eqref{eq:W_psi_1} separately. A short computation shows that
    \begin{equation}
        \frac{\xi_{\infty}(2\lambda^3)}{N(\lambda^3)^{1/2}}\sum_{\substack{a \pmod{\lambda^3}\\(a,\lambda) =1}} \chi_{\xi}(a)\ec\Big(\frac{a}{2\lambda^3}\Big) = (-1)^{\frac{\omega^2-1}{8}}i = \quadrat{2}{\omega} i.
    \end{equation}
    Since $q \in \mathcal F_4$ and $\tilde{q}=q_1q_2q_3$, we have $\chi_q(\lambda) = \chi_q(i) = 1$ and by \cref{h3fourier} conclude that
    \begin{equation*}
        \frac{1}{N(\tilde{q})^{1/2}}\sum_{b \pmod {\tilde{q}}} \chi_q(\lambda^3 b) \ec\Big(\frac{b}{2\tilde{q}} \Big) =  \widetilde{h}_4(1,\chi_q) =\chi_{q_1}(q_2q_3)\chi_{q_2^2}(q_1q_3)\overline{\chi_{q_3}(q_1q_2)} \widetilde{g}_4(q_1)\widetilde{g}_2(q_2)\overline{\widetilde{g}_4(q_3)}.
    \end{equation*}
    Thus
    \begin{equation*}
        W(\nu_{q,\omega}) = (-1)^{\frac{\mathfrak{\omega}+1}{2}}\quadrat{2}{\mathfrak{\omega}}\xi(q_1q_2q_3)\chi_{q_1}(q_2q_3)\chi_{q_2^2}(q_1q_3)\overline{\chi_{q_3}(q_1q_2)} \widetilde{g}_4(q_1)\widetilde{g}_2(q_2)\overline{\widetilde{g}_4(q_3)}.
    \end{equation*}
\end{proof}

\begin{corollary}\label{cor:root_num}
    If $q \in \mathcal{F}_4^{\prime}$, then $W(\nu_{q,\omega})=\varepsilon(\omega) \xi(q) \widetilde{g}_4(q)$.
\end{corollary}

\subsection{Approximate functional equations}

The lemma below records the standard approximate function equation for $L(s,\nu_{q,\omega})$.

\begin{lemma}[{\cite[Theorem 5.3]{IK}}] \label{lem:afe}
    Let $\mathcal{F}_4$ be as in \eqref{eq:fam}, $U > 0$, $\omega \in \Z$, and $q=q_1q_2^2q_3^3\in \mathcal{F}_4$ with $q_i \equiv 1 \pmod{\lambda^3}$ and $\mu^2(q_1 q_2 q_3)=1$. Let $\nu_{q,\omega}$ be as in \eqref{eq:nu_q}, and assume it is non-trivial, i.e.\ $(q, \omega) \neq (1, 0)$. Then for $s\in \C$ with $0 \le \Re(s) \le 1$ we have 
    \begin{align}
        L(s,\nu_{q,\omega})= \sum_{0 \ne \mathfrak n \unlhd \Z[i] } \frac{\nu_{q,\omega}(\mathfrak n)}{N(\mathfrak n )^s}V_{s,\omega}\Big(\frac{N(\mathfrak n)}{2U\sqrt{N(q_1 q_2 q_3 \mathfrak m_{\omega})}}\Big) + W(\nu_{q,\omega})(4N(q_1 q_2 q_3 \mathfrak m_{\omega}))^{\frac{1}{2} -s} \nonumber \\
        \times \, (2\pi)^{2s-1} \frac{\Gamma(1-s + \frac{|\omega|}{2})}{\Gamma(s+ \frac{|\omega|}{2})} \sum_{0 \ne \mathfrak n \unlhd \Z[i]} \frac{\overline{\nu_{q,\omega}(\mathfrak n)}}{N(\mathfrak n)^{1-s}}V_{1-s,\omega}\Big(\frac{N(\mathfrak n) U}{2 \sqrt{N(q_1 q_2 q_3 \mathfrak m_{\omega})}}\Big), \qquad \label{AFEeq}
    \end{align}
    where $W(\nu_{q, \omega})$ is given in \cref{lem:Wpsi} and 
    \begin{equation}\label{eq:V_s_f}
        V_{s,\omega}(y):= \frac{1}{2\pi i }\int_{2-i\infty}^{2+i\infty}(2\pi)^{-w}y^{-w}e^{w^2}\frac{\Gamma(s+\frac{|\omega|}{2}+w)}{\Gamma(s+\frac{|\omega|}{2})}\frac{dw}{w}.
    \end{equation}
\end{lemma}

We now record an estimate for $V_s(y)$ in the our degree $2$ setting, which follows from Stirling's formula and the rapid decay of $e^{w^2}$ in vertical strips, as in \cite[Proposition~5.4]{IK}.

\begin{lemma} \label{decaylem}
    Suppose that $\Re(s) \geq \alpha>0$. Then for every $y>0$ and $A, k \in \mathbb{Z}_{\geq 0}$ we have
    \begin{align*}
        y^{k} V^{(k)}_{s,\omega}(y) \ll_{A,\alpha,k} \Big(1+\frac{y}{1+ |\omega|+ |s|} \Big)^{-A} .
    \end{align*}
\end{lemma}

It will be convenient to work with slightly different expressions at $s = \frac{1}{2}$. For any $q \in \Z[i]$ with $q \equiv 1 \pmod{\lambda^3}$, denote
\begin{equation} \label{BUdef}
    B_{\omega,U}(q):= \sum_{0 \ne \mathfrak n \unlhd \Z[i] } \frac{\nu_{q,\omega}(\mathfrak n)}{N(\mathfrak n )^{1/2}}V_{\omega}\Big(\frac{N(\mathfrak n)}{2U\sqrt{N(q \mathfrak m_{\omega})}}\Big),
\end{equation}
\begin{equation} \label{BUtildedef}
    \widetilde{B}_{\omega,U}(q) := \sum_{ 0 \ne \mathfrak n \unlhd \Z[i]} \frac{\overline{\nu_{q,\omega}(\mathfrak n)}}{N(\mathfrak n)^{1/2}}V_{\omega}\Big(\frac{N(\mathfrak n)}{2 U \sqrt{N(q \mathfrak m_{\omega})}}\Big),
\end{equation}
and
\begin{equation} \label{Aomegadef}
    A_{\omega}(q) := \sumtwo_{0 \neq \mathfrak{n}_1, \mathfrak{n}_2 \unlhd \Z[i]} \frac{\nu_{q,\omega}(\mathfrak{n}_1) \overline{\nu_{q,\omega}( \mathfrak{n}_2)}}{ N( \mathfrak{n}_1 \mathfrak{n}_2)^{1/2}}
    W_{\omega} \Big( \frac{N(\mathfrak{n}_1 \mathfrak{n}_2)}{4 N(q \mathfrak{m}_{\omega})} \Big).
\end{equation}
Here $V_\omega(w) := V_{\frac{1}{2},\omega}(w)$, which is real hence in fact $\widetilde{B}_{\omega,U}(q) = \overline{B_{\omega,U}(q)}$, and
\begin{equation}\label{eq:W_f}
    W_{\omega}(y):= \frac{1}{2\pi i }\int_{2-i\infty}^{2+i\infty}(2\pi)^{-2w}y^{-w}e^{w^2}\frac{\Gamma(\frac{1}{2}+\frac{|\omega|}{2}+w)^2}{\Gamma(\frac{1}{2}+\frac{|\omega|}{2})^2}\frac{dw}{w}.
\end{equation}
As in \cref{decaylem}, for every $y > 0$ and $A, k \in \Z_{\geq 0}$ we have
    \begin{equation}\label{W_f_bound}
        y^k W_{\omega}^{(k)}(y) \ll_{A,k}\Big(1+\frac{y}{(1+|\omega|)^2}\Big)^{-A}.
    \end{equation}

The main advantage of the expressions above is that their lengths involve $N(q)$ instead of $N(\ra(q))$, which is analytically more convenient. Since $q = \ra(q)$ for $q \in \mathcal{F}'_4$, by \cite[Theorem 5.3]{IK} and \cref{cor:root_num} we obtain the following.

\begin{lemma} \label{lem:afe_2}
    Let $\mathcal{F}_4'$ be as in \eqref{eq:subfam}, $\omega \in \Z$, and $U > 0$.
    If $q\in \mathcal{F}_4'$ then
    \begin{equation*}
        L(1/2,\nu_{q,\omega}) = B_{\omega,U}(q) + \varepsilon(\omega) \xi(q) \widetilde{g}_4(q) \widetilde{B}_{\omega,U^{-1}}(q)
    \end{equation*}
    and
    \begin{equation*}
        |L(1/2,\nu_{q,\omega})|^2 = 2 A_{\omega}(q).
    \end{equation*}
\end{lemma}

\subsection{Poisson Summation} This section gathers some straightforward consequences of the Poisson summation formula.

\begin{lemma}\label{lem:pois_1}
    Let $V : \R^2 \to \C$ be a Schwartz function. By an abuse of notation, set $V(x+iy)= V(x,y)$. Then we have
    \begin{equation}\label{eq:pois_1}
        \sum_{m \in \Z[i]}V(m) = \sum_{k\in \Z[i]}\int_{\R^2}V(x,y)\ec\Big(\frac{k(x+iy)}{2}\Big)dx dy.
    \end{equation}
\end{lemma}
\begin{proof}
    View $\Z[i]$ as a self-dual lattice in $\C$. Apply Poisson summation to find
    \begin{gather*}
        \sum_{m \in \Z[i]}V(m) =\sum_{k_1+ik_2 \in \Z[i]} \int_{\R^2}V(x,y)e(k_1x+k_2y)dxdy.
    \end{gather*}
    The result follows after applying the identity $\Re[(k_1-ik_2)(x+iy)] = k_1 x + k_2 y$, using the fact that $\Z[i]$ is closed under conjugation, and the definition of $\ec(z)$.
\end{proof}
\begin{lemma}\label{lem:pois_2}
    For $q \in \Z[i]$, let $\psi: \Z[i]\to \C$ be $q$-periodic and $V:\R^2 \to \C$ be a Schwartz function. Then
    \begin{equation*}
        \sum_{m \in \Z[i]}\psi(m)V(m) = \frac{1}{N(q)} \sum_{k \in \Z[i]}\dot{\psi}(k)\dot{V}\Big(\frac{k}{q}\Big),
    \end{equation*}
    where
    \begin{align*}
        \dot{\psi}(k):= \sum_{t \pmod q} \psi(t)\ec\Big({-\frac{kt}{2 q}}\Big)  \qquad \text{and} \qquad \dot{V}(r):= \int_{\R^2} V(x,y)\ec\Big(\frac{r(x+iy)}{2}\Big)dxdy.
    \end{align*}
\end{lemma}
\begin{proof}
    Sum over congruence classes modulo $q$ and apply Lemma~\ref{lem:pois_1} to see that
    \begin{gather}\label{eq:per_pois}
        \sum_{m \in \Z[i]} \psi(m)V(m) = \sum_{t \pmod q} \psi(t)  \sum_{k \in \Z[i]}\int_{\R^2} V((u+iv)q+t)\ec\Big(\frac{k(u+iv)}{2}\Big)dudv.
    \end{gather}
    A linear change of variables shows that
    \begin{gather}
        \int_{\R^2} V((u+iv)q+t)\ec\Big(\frac{k(u+iv)}{2}\Big) du dv = \frac{1}{N(q)}\int_{\R^2}V(x,y)\ec\Big(\frac{k(x+iy)}{2q}\Big)\ec\Big({-\frac{kt}{2q}}\Big) dx dy.
    \end{gather}
    The result follows from summing over $t \pmod q$.
\end{proof}
\begin{lemma}\label{lem:pois}
    Let $q,c \in \Z[i]$ with $q\equiv 1 \pmod {\lambda^3}$,  $\psi : \Z[i] \to \C$ be $q$-periodic, and $V:\R \to \C$ be a Schwartz function. Then for any $M > 0$ we have
    \begin{gather*}
        \sum_{\substack{m\in \Z[i] \\ m \equiv c \pmod{\lambda^7}}} \psi(m) V\Big(\frac{N(m)}{M}\Big) = \frac{\pi M}{64 N(q)} \sum_{k \in \Z[i]}\ec\Big({-\frac{kcq^3}{2\lambda^7}}\Big)\ddot{\psi}(k)\ddot{V}\Big(\frac{k\sqrt{M}}{q}\Big),
    \end{gather*}
    where
    \begin{gather}\label{eq:ddotV}
        \ddot{\psi}(k):=  \sum_{b \pmod q} \psi(2\lambda^7b)\ec\Big({-\frac{kb}{q}}\Big) \qquad \text{and} \qquad \ddot{V}(u):= \int_{0}^{\infty}rV(r^2)J_0\Big(\frac{\pi r|u|}{4\sqrt{2}}\Big)dr. \qquad
    \end{gather}
\end{lemma}
\begin{proof}
    We closely follow the proof of \cite[Lemma 3.10]{DDDS24}. By Lemma ~\ref{lem:pois_2} we obtain
    \begin{equation*}
        \sum_{\substack{m\in \Z[i]\\ m \equiv c \pmod{\lambda^7}}}\psi(m) V\Big(\frac{N(m)}{M}\Big) = \frac{1}{N(\lambda^7 q)} \sum_{k\in \Z[i]}\dot{\psi}(k) \int_{\R^2}V\Big(\frac{x^2+y^2}{M}\Big)\ec\Big(\frac{k(x+iy)}{2\lambda^7q}\Big)dx dy,
    \end{equation*}
    where
    \begin{equation*}
        \dot \psi(k) = \sum_{\substack{t \pmod {\lambda^7q} \\ t\equiv c \pmod {\lambda^7}}} \psi(t) \ec\Big({-\frac{kt}{2\lambda^7q}}\Big).
    \end{equation*}
    The integral term can be computed as in \cite[(4.6)]{DR}. Let $\alpha \in [-\pi,\pi)$ be determined by $\frac{k}{2\lambda^7q} = \frac{|k|}{16\sqrt 2|q|}e^{-i\alpha}$. Changing to polar coordinates we have
    \begin{align}
        \int_{\R^2}V\Big(\frac{x^2+y^2}{M}\Big)\ec\Big(\frac{k(x+iy)}{2\lambda^7q}\Big)dxdy& = M\int_{\R} r V(r^2) \int_{0}^{2\pi}\ec\Big(\frac{r|k|\sqrt{M}}{16\sqrt2|q|}e^{i(\theta-\alpha)}\Big)d\theta dr \nonumber \\
        & = M\int_{\R} r V(r^2) \int_{0}^{2\pi}\exp\Big(i\frac{\pi r|k|\sqrt{M}}{4\sqrt 2|q|}\cos(\theta-\alpha)\Big)d\theta dr \nonumber.
    \end{align}
    By \cite[(10.9.2)]{DLMF} we have $J_0(x) = \frac{1}{2\pi} \int_0^{2\pi} e^{-{ix\cos(\theta)}}d\theta$, so the display above is equal to
    \begin{gather*}
        2\pi M \int_0^{\infty}rV(r^2)J_0\Big(\frac{\pi r|k|\sqrt{M}}{4\sqrt{2}|q|}\Big) dr = 2 \pi M \ddot{V}\Big(\frac{k\sqrt{M}}{q}\Big).
    \end{gather*}
    We now evaluate $\dot\psi(k)$. By the Chinese remainder theorem, decompose $t = aq +\lambda^7b$ for $a \pmod{\lambda^7}$ and $b \pmod q$. From $q \equiv 1 \pmod {\lambda^3}$ we have $q^4 \equiv 1 \pmod{\lambda^7}$, so $cq^3 \equiv a \pmod {\lambda^7}$.  By the $q$-periodicity of $\psi$, we conclude that
    \begin{align*}
        \dot \psi(k) &= \bigg(\sum_{\substack{a \pmod{\lambda^7} \\ a \equiv cq^3 \pmod{\lambda^7}}}\ec\Big({-\frac{ka}{2\lambda^7}}\Big)\bigg)\bigg(\sum_{b \pmod q} \psi(\lambda^7b)\ec\Big({-\frac{kb}{2q}}\Big)\bigg)  \\
        & = \ec\Big({-\frac{kcq^3}{2\lambda^7}}\Big)\bigg( \sum_{b \pmod q} \psi(2\lambda^7b)\ec\Big({-\frac{kb}q}\Big)\bigg)
    \end{align*}
    and the result follows.
\end{proof}

\subsection{Quartic large sieve}

The quartic large sieve of Blomer, Goldmakher, and Louvel \cite{BGL} will be an important tool in the next section.

\begin{theorem}[{\cite[Theorem~1.3]{BGL}}] \label{quartic-large-sieve}
    Let $\varepsilon>0$, $A, B \geq 1$, $\boldsymbol{\lambda}=(\lambda_b)$ be a $\mathbb{C}$-valued sequence supported on $b \in \mathbb{Z}[i]$ such that $b \equiv 1 \pmod{\lambda^3}$, and
    $\chi_a(\cdot)$ denote the quartic residue symbol $\big( \frac{\cdot}{a} \big)_4$ for $a \in \mathbb{Z}[i]$ with $a \equiv 1 \pmod{\lambda^3}$.
    Then
    \begin{align*}
        \sum_{\substack{ a \in \mathbb{Z}[i] \\ a \equiv 1 \pmod{\lambda^3} \\ N(a) \leq A  }} \mu^2(a) \Bigg | \sum_{\substack{ b \in \mathbb{Z}[i] \\ b \equiv 1 \pmod{\lambda^3} \\ N(b) \leq B }}
        \mu^2(b) \lambda_b \chi_a(b)  \Bigg |^2
        \ll_{\varepsilon} (AB)^\varepsilon ( A + B + (AB)^{2/3} ) \cdot \| \boldsymbol{\mu}^2 \boldsymbol{\lambda} \|_2^2.
    \end{align*}
\end{theorem}


\section{Sieving error terms}\label{sec:S_R_bounds}

In this section we prove Propositions \ref{SRestimate1} and \ref{SRestimate2}. It is convenient to prove them simultaneously, since many of the technical steps are similar.

\begin{proof}[Proof of Propositions \ref{SRestimate1} and \ref{SRestimate2}]
    Recall that $1 \leq U \leq X^{100}$, and also the definitions of $B_{\omega,U}(q)$, $\widetilde{B}_{\omega,U}(q)$, and $A_{\omega}(q)$ given in \eqref{BUdef}, \eqref{BUtildedef}, and \eqref{Aomegadef} respectively. Let $\mathcal{D}_{\omega} \in \{B_{\omega,U}, \widetilde{B}_{\omega,U^{-1}},A_{\omega} \}$. The first part of this proof relates $\mathcal{D}_{\omega}(q)$ back to a Dirichlet polynomial of the correct length, replacing $N(q)$ with $N(\ra(q))$. The second part of the proof is an application of the quartic large sieve.

    \subsection{Dirichlet polynomials}

    It suffices to estimate
    \begin{equation} \label{eq:S_R_1}
        \mathcal{S}_{R}(|\mathcal{D}_\omega(q)|;F)
        =\sum_{\substack{q \in \Z[i]\\ q \equiv 1 \pmod {\lambda^7} }} |R_Y(q) \mathcal{D}_{\omega}(q)| F\Big(\frac{N(q)}{X} \Big).
    \end{equation}
    From \eqref{MYRYdef} and the divisor bound, it follows that $|R_Y(q)| \ll_\varepsilon N(q)^{\varepsilon}$ and that $R_Y(q)=0$ unless $q \mathbb{Z}[i]=\mathfrak{l}^2 \mathfrak{r}$, where $\mathfrak{r}$ is squarefree and $N(\mathfrak{l})>Y$ (note that $\mathfrak{l}$ here is not to be confused with the $\mathfrak{l}$ in \eqref{MYRYdef}). This is equivalent to $q=\ell^2 r$, where $\ell, r \in \Z[i]$ satisfy $\ell, r \equiv 1 \pmod{\lambda^3}$ with $r$ squarefree, $N(\ell)>Y$, and $\ell^2 r \equiv 1 \pmod{\lambda^7}$.
    Therefore \eqref{eq:S_R_1} is
    \begin{align} \label{SRintermed}
        \ll_\varepsilon X^{\varepsilon} \sum_{\substack{ \ell \in \Z[i] \\ \ell \equiv 1 \pmod{\lambda^3} \\ N(\ell)>Y }} \sum_{\substack{r \in \Z[i] \\ \ell^2 r \equiv 1 \pmod{\lambda^7}}} \mu^2(r) | \mathcal{D}_{\omega}(\ell^2 r)| F \Big( \frac{N(\ell^2 r)}{X} \Big).
    \end{align}

    Let $L=2^{\alpha} \gg 1$ and $R=2^{\beta} \gg 1$, where $\alpha, \beta \in \Z$.
    We insert a dyadic partition $N(\ell) \sim L$ and $N(r) \sim R$ in \eqref{SRintermed}. It suffices to consider
    parameters $L$ and $R$ such that $L^2 R \asymp X$ and $L \gg Y$. Thus, recalling that $0 \leq F \leq 1$, it follows that \eqref{SRintermed} is
    \begin{equation} \label{SRdyadic}
        \ll_\varepsilon X^{\varepsilon} \sum_{\substack{\alpha \in \Z \\ L=2^{\alpha}  \\ Y  \ll L \ll \sqrt{X} }} \sum_{\substack{ \beta \in \Z \\ R=2^{\beta} \\ R \asymp X/L^2 }}  Z_{L,R,\mathcal{D}_{\omega}}
    \end{equation}
    for
    \begin{equation} \label{ZLRdef}
        Z_{L,R,\mathcal{D}_{\omega}} := \sum_{\substack{ \ell \in \Z[i] \\ \ell \equiv 1 \pmod{\lambda^3} \\ N(\ell) \sim L }} \sum_{\substack{r \in \Z[i] \\ \ell^2 r \equiv 1 \pmod{\lambda^7} \\ N(r) \sim R }} \mu^2(r) | \mathcal{D}_{\omega}(\ell^2 r)|.
    \end{equation}

    Let $e:= (r, \ell)$. Thus $r = e r^{\prime}$ and $\ell = e \ell'$ for $r^{\prime}, \ell^{\prime} \in \mathbb{Z}[i]$ with $e, r^{\prime}, \ell^{\prime} \equiv 1 \pmod{\lambda^3}$ and $(r^{\prime},\ell^{\prime})=1$. Moreover, since $r$ is squarefree, both $e$ and $r^{\prime}$ are squarefree, and $(r^{\prime},e \ell^{\prime} )=1$. We further uniquely factorize $\ell^{\prime} = a b^2 c^3 d^4$ with $a,b,c,d \in \mathbb{Z}[i]$ such that $a,b,c,d \equiv 1 \pmod{\lambda^3}$ and $\mu^2(abc)=1$. The condition $(r^{\prime},e \ell^{\prime} )=1$ and the factorization above guarantee that $(r^{\prime},abcde)=1$. Thus
    \begin{equation} \label{chareqn_new}
        \nu_{\ell^2 r,\omega }=\nu_{a^2 b^4 c^6 d^8 e^3 r^{\prime},\omega}=\nu_{a^2 c^2 e^3 r^{\prime},\omega} \cdot \mathbf{1}_{bd},
    \end{equation}
    where $\mathbf{1}_{u}$ denotes the principal character modulo $u \Z[i]$. Denoting $f := a^2c^2 e^3$, note that the condition $\ell^2 r \equiv 1 \pmod{\lambda^7}$ is equivalent to $fr' \equiv 1 \pmod{\lambda^7}$, where we use that if $k \in \Z[i]$ satisfies $k \equiv 1 \pmod{\lambda^3}$, then $k^4 \equiv 1 \pmod{\lambda^7}$. 
    Observe also that $\nu_{f r^{\prime},\omega}$ is a primitive character with frequency $\omega$ and conductor $\ra(f)r^{\prime}\mathfrak{m}_\omega$.
    
    We now separate variables in $D_{\omega}(\ell^2 r)$ using
    \eqref{eq:V_s_f} and \eqref{eq:W_f},
    and interchange the integral with the summations by absolute convergence, obtaining 
    \begin{align} \label{BUintermed}
        B_{\omega,U}(\ell^2 r) = \frac{1}{2\pi i }\int_{2-i\infty}^{2+i\infty} (2 \pi)^{-w} U^{w}
        (4N(a^2 b^4 c^6 d^8 e^3 r^{\prime} \m_{\omega}))^{w/2} e^{w^2} \frac{\Gamma(\frac{1+|\omega|}{2}+w)}{\Gamma(\frac{1+|\omega|}{2})} & \nonumber \\
         \times \, \mathcal{E}_{\omega}(w,fr',bd) L(1/2+w, \nu_{f r^{\prime},\omega}) \, \frac{dw}{w}, & \qquad
    \end{align}
    \begin{align} \label{BUtildeintermed}
        \widetilde{B}_{\omega,U^{-1}}(\ell^2 r) =  \frac{1}{2\pi i }\int_{2-i\infty}^{2+i\infty} (2 \pi)^{-w}
        U^{-w} (4N(a^2 b^4 c^6 d^8 e^3 r^{\prime} \m_{\omega}))^{w/2} e^{w^2} \frac{\Gamma(\frac{1+|\omega|}{2}+w)}{\Gamma(\frac{1+|\omega|}{2})} & \nonumber \\
        \times \, \overline{\mathcal{E}_{\omega}(\overline{w},fr',bd)} \cdot \overline{L(1/2+\overline{w}, \nu_{f r^{\prime},\omega})} \, \frac{dw}{w}, & \qquad
    \end{align}
    and 
    \begin{align} \label{Aintermed}
        &A_{\omega}(\ell^2 r) =  \frac{1}{2\pi i }\int_{2-i\infty}^{2+i\infty} (2 \pi)^{-2w}
        (4N(a^2 b^4 c^6 d^8 e^3 r^{\prime} \m_{\omega}))^{w} e^{w^2}\frac{\Gamma(\frac{1+|\omega|}{2}+w)^2}{\Gamma(\frac{1+|\omega|}{2})^2} \nonumber \\
        &\times \mathcal{E}_{\omega}(w,fr',bd) \cdot  \overline{\mathcal{E}_{\omega}(\overline{w},fr',bd)} \cdot L(1/2+w, \nu_{f r^{\prime},\omega}) \cdot \overline{L(1/2+\overline{w}, \nu_{f r^{\prime},\omega})} \, \frac{dw}{w}, \qquad
    \end{align}
    where
    \begin{equation*}
        \mathcal{E}_{\omega}(w,fr',bd):=\prod_{\substack{\pi \text{ prime} \\ \pi \equiv 1 \pmod{\lambda^3} \\ \pi \mid bd,\ \pi \nmid fr'}} \big(1-\nu_{f r^{\prime},\omega}(\pi) N(\pi)^{-1/2-w} \big).
    \end{equation*}
    In particular, for all $\varepsilon > 0$ we have (uniformly in $\omega$)
    \begin{equation} \label{mathcalEbd}
        \mathcal{E}_{\omega}(w,fr',bd) \ll_\varepsilon X^{\varepsilon} \quad \text{for} \quad \Re(w) \geq \varepsilon.
    \end{equation}

    Note that $\nu_{fr', \omega}$ is the trivial character if and only if  $\omega = 0$ and $f = r' = 1$. This last condition is equivalent to $\ell^2 r =a^2 b^4 c^6 d^8 e^3 r^{\prime} = g^4$ for some $g \equiv 1 \pmod{\lambda^3}$. In that case, since $N(g^4) \asymp X$, a trivial bound using \cref{decaylem} and \eqref{W_f_bound} gives
    \begin{equation} \label{residuecont}
        |\mathcal{D}_\omega(\ell^2 r)| = |\mathcal{D}_0(g^4)| \ll_\varepsilon X^{\varepsilon} \cdot
        \begin{cases}
            U^{1/2} X^{1/4}  & \quad \text{if} \quad \mathcal{D}_{\omega}=B_{\omega,U}, \\
            U^{-1/2} X^{1/4} & \quad \text{if} \quad \mathcal{D}_{\omega}=\widetilde{B}_{\omega,U^{-1}}, \\
            X^{1/2} & \quad \text{if} \quad \mathcal{D}_{\omega}=A_{\omega}.
        \end{cases}
    \end{equation}
    For the non-trivial $\nu_{fr', \omega}$, we move the contour in each of the three integrals above to $\Re(w)=\varepsilon$. Then use Stirling's formula, the convexity bound for $L(s, \nu_{fr', \omega})$, and \eqref{mathcalEbd} to truncate those integrals to $\left|\Im(w)\right| \leq ((1+|\omega|)X)^{\varepsilon}$, up to error $\ll_\varepsilon ((1+|\omega|)X)^{-1000}$, and then to estimate the remaining integrands. From \eqref{residuecont}, the fact that $U \geq 1$, and Cauchy--Schwarz we deduce that
    \begin{align} \label{BUcombineintermed2}
        &|B_{\omega,U}(\ell^2 r)|+|\widetilde{B}_{\omega,U^{-1}}(\ell^2 r)| \ll_\varepsilon ((1+|\omega|)X)^{-750} + \bbone_{\omega=0} \cdot \bbone_{\ell^2 r = g^4} \cdot U^{1/2} X^{1/4+\varepsilon}  \nonumber \\
        &+ (1- \bbone_{\omega=0} \cdot \bbone_{\ell^2 r = g^4}) \cdot ((1+|\omega|) X)^{\ep} \cdot \int_{\varepsilon-i ((1+|\omega|)X)^{\varepsilon}}^{\varepsilon+i ((1+|\omega|)X)^{\varepsilon}}  |L(1/2+w, \nu_{f r^{\prime},\omega})| \, |dw|  \qquad
    \end{align}
    and
    \begin{align} \label{Aintermed2}
        & |A_{\omega}(\ell^2 r)| \ll_\varepsilon  ((1+|\omega|)X)^{-750}+\bbone_{\omega=0} \cdot \bbone_{\ell^2 r = g^4} \cdot X^{1/2+\varepsilon}  \nonumber \\
        &+(1- \bbone_{\omega=0} \cdot \bbone_{\ell^2 r = g^4}) \cdot ((1+|\omega|) X)^{\ep} \cdot \int_{\varepsilon-i ((1+|\omega|)X)^{\varepsilon}}^{\varepsilon+i ((1+|\omega|)X)^{\varepsilon}}
        |L(1/2+w, \nu_{f r^{\prime},\omega})|^2 \, |dw|. \qquad
    \end{align}
    Substituting \eqref{BUcombineintermed2} and \eqref{Aintermed2} into \eqref{ZLRdef}, we deduce that
    \begin{align} \label{ZRB}
        &Z_{L,R,B_{\omega,U}}+Z_{L,R,\widetilde{B}_{\omega,U^{-1}}} \ll_\varepsilon  ((1+|\omega|)X)^{-500}+ \bbone_{\omega=0} \cdot U^{1/2} X^{1/2+\varepsilon} + ((1+|\omega|)X)^{\varepsilon} \nonumber \\
        & \times \sup_{\substack{|t| \leq ((1+|\omega|)X)^{\varepsilon} \\ t \in \R }}
        \sum_{\substack{ a,b,c,d,e \in \Z[i] \\ a,b,c,d,e \equiv 1 \pmod{\lambda^3} \\ N(ab^2c^3d^4e) \sim L }} \mu^2(abc) \sum_{\substack{r^{\prime} \in \Z[i] \\ f r^{\prime} \equiv 1 \pmod{\lambda^7} \\ N(r^{\prime}) \sim R/N(e) \\ (r^{\prime},abcde)=1 \\ \omega=0 \implies f r^{\prime} \neq 1  }} \mu^2(e r^{\prime})  |L(1/2+\varepsilon+ it, \nu_{f r^{\prime},\omega})| \qquad
    \end{align}
    and
    \begin{align} \label{ZRA}
        Z_{L,R,A_{\omega}} \ll_\varepsilon  ((1+|\omega|)X)^{-500}+ \bbone_{\omega=0} \cdot X^{3/4+\varepsilon} + ((1+|\omega|)X)^{\varepsilon} \sup_{\substack{|t| \leq ((1+|\omega|)X)^{\varepsilon} \\ t \in \R }} & \nonumber \\
        \sum_{\substack{ a,b,c,d,e \in \Z[i] \\ a,b,c,d,e \equiv 1 \pmod{\lambda^3} \\ N(ab^2c^3d^4e) \sim L }} \mu^2(abc)
        \sum_{\substack{r^{\prime} \in \Z[i] \\ f r^{\prime} \equiv 1 \pmod{\lambda^7} \\ N(r^{\prime}) \sim R/N(e) \\ (r^{\prime},abcde)=1 \\ \omega=0 \implies f r^{\prime} \neq 1  }} \mu^2(er^{\prime})
        |L(1/2+\varepsilon + it, \nu_{f r^{\prime},\omega})|^2. & \qquad
    \end{align}

    Now apply the (balanced) approximate functional equation from \cref{lem:afe} to express $L(1/2+\varepsilon+it, \nu_{f r^{\prime},\omega})$, for $|t| \leq ((1+|\omega|)X)^{\varepsilon}$, using Dirichlet polynomials of the appropriate length. In both terms on the right side of \eqref{AFEeq}, write (uniquely) $\mathfrak{n}=\lambda^g n_1 n_2^2$ for $g \in \Z_{\geq 0}$ and $n_1,n_2 \equiv 1 \pmod{\lambda^3}$ with $\mu^2(n_1) = 1$. Then apply the triangle inequality, Stirling's formula, and the relation $\overline{V_{s, \omega}(y)} = V_{\overline{s}, \omega}(y)$ to deduce that
    \begin{align} \label{Lintermed}
        |L(1/2+\varepsilon+it, \nu_{fr^{\prime},\omega})| \ll_\varepsilon ((1+|\omega|)X)^{\varepsilon} \sum_{\pm} \sum_{g \in \Z_{\geq 0}} \frac{1}{2^{g(1/2 \pm \varepsilon)}} \sum_{\substack{n_2 \in \Z[i] \\ n_2 \equiv 1 \pmod{\lambda^3} }} \frac{1}{N(n_2)^{1 \pm 2 \varepsilon}} & \nonumber \\
        \times \, \Bigg |  \sum_{\substack{n_1 \in \Z[i] \\ n_1 \equiv 1 \pmod{\lambda^3} } } \frac{\mu^2(n_1) \nu_{f r^{\prime},\omega}(n_1)}{N(n_1)^{1/2 \pm \varepsilon + it}} V_{1/2 \pm \varepsilon + it,\omega} \Big(\frac{2^g N(n_1 n_2^2)}{2\sqrt{N(\ra(f r^{\prime}) \mathfrak m_{\omega})}}\Big) \Bigg |. & \qquad
    \end{align}
    Observe that $\text{rad} (fr^{\prime} ) \mid \text{rad}(\ell r)$, so by \cref{decaylem} we may truncate the sums in \eqref{Lintermed} up to $N(n_2^2) \leq \mathcal{N}(L,R)$ and  $N(n_1) \leq \mathcal{N}(L,R)$, incurring an error term $\ll_\varepsilon ((1+|\omega|)X)^{-1000}$, where
    \begin{equation} \label{Ndef}
        \mathcal{N}(L,R)= X^{\varepsilon}(1+|\omega|)^{1+\varepsilon} \sqrt{LR}.
    \end{equation}
    For simplicity let us separate variables by opening $V_{1/2 \pm \varepsilon + it,\omega}$ using \eqref{eq:V_s_f}, moving the contour to $\Re(w)=\varepsilon$, and applying Stirling's formula to truncate the integral to $\left|\Im(w)\right| \leq ((1+|\omega|) X)^{\varepsilon}$ with an error $\ll _\varepsilon ((1+|\omega|)X)^{-1000}$. By the triangle inequality, performing the sums over $g$ and $n_2$, we obtain
    \begin{align} 
        & |L(1/2+\varepsilon+it, \nu_{fr^{\prime},\omega})|  \nonumber \\
        & \ll_\varepsilon ((1+|\omega|)X)^{\varepsilon}  \int_{-((1+|\omega|)X)^{\varepsilon}}^{{((1+|\omega|)X)^{\varepsilon}}} \sum_{\pm} \Bigg |  \sum_{\substack{n_1 \in \Z[i] \\ n_1 \equiv 1 \pmod{\lambda^3} \\ N(n_1) \leq  \mathcal{N}(L,R) } } \frac{\mu^2(n_1) \nu_{fr^{\prime},\omega}(n_1)}{N(n_1)^{1/2 \pm \varepsilon +\varepsilon+i(t+y)}} \Bigg | |dy| +((1+|\omega|)X)^{-750}. \nonumber
    \end{align}
    By Cauchy--Schwarz, this implies
    \begin{align} \label{Lintermed3}
        & |L(1/2+\varepsilon+it, \nu_{fr^{\prime},\omega})|^2  \nonumber \\
        & \ll_\varepsilon ((1+|\omega|)X)^{\varepsilon}  \int_{-((1+|\omega|)X)^{\varepsilon}}^{{((1+|\omega|)X)^{\varepsilon}}} \sum_{\pm} \Bigg |  \sum_{\substack{n_1 \in \Z[i] \\ n_1 \equiv 1 \pmod{\lambda^3} \\ N(n_1) \leq \mathcal{N}(L,R) } } \frac{\mu^2(n_1) \nu_{f r^{\prime},\omega}(n_1)}{N(n_1)^{1/2 \pm \varepsilon +\varepsilon+i(t+y)}} \Bigg |^2 |dy| +((1+|\omega|)X)^{-750}.
    \end{align}

    \subsection{Application of the large sieve}

    We now apply Cauchy--Schwarz to deduce that the sum over $r^{\prime}$ in \eqref{ZRB} is
    \begin{equation} \label{rbd1}
        \ll \Big(\frac{R}{N(e)} \Big)^{1/2} \cdot \Bigg( \sum_{\substack{r^{\prime} \in \Z[i] \\ f r^{\prime} \equiv 1 \pmod{\lambda^7} \\ N(r^{\prime}) \sim R/N(e) \\ (r^{\prime},abcde)=1 \\ \omega=0 \implies f r^{\prime} \neq 1  }} \mu^2(er^{\prime})  |L(1/2+\varepsilon+ it, \nu_{fr^{\prime},\omega})|^2 \Bigg)^{1/2}.
    \end{equation}
    By \eqref{Lintermed3} we have
    \begin{align} \label{secmomentintermed}
        \sum_{\substack{r^{\prime} \in \Z[i] \\ f r^{\prime} \equiv 1 \pmod{\lambda^7} \\ N(r^{\prime}) \sim R/N(e) \\ (r^{\prime},abcde)=1 \\ \omega=0 \implies f r^{\prime} \neq 1  }} \mu^2(er^{\prime})  |L(1/2+\varepsilon+it, \nu_{f r^{\prime},\omega})|^2 \ll_\varepsilon ((1+|\omega|)X)^{-500} + ((1+|\omega|)X)^{\varepsilon} & \nonumber \\
        \times \sup_{\substack{y \in \R \\ |y| \leq ((1+|\omega|)X))^{\ep}}} \sum_{\pm}
        \sum_{\substack{r^{\prime} \in \Z[i] \\ r^{\prime} \equiv 1 \pmod{\lambda^3} \\  N(r^{\prime}) \sim R/N(e) }} \mu^2(r^{\prime})  \Bigg |  \sum_{\substack{n_1 \in \Z[i] \\ n_1 \equiv 1 \pmod{\lambda^3} \\ N(n_1) \leq \mathcal{N}(L,R) } } \frac{\mu^2(n_1) \nu_{f r^{\prime},\omega}(n_1)}{N(n_1)^{1/2 \pm \varepsilon +\varepsilon+i(t+y)}} \Bigg |^2, & \qquad
    \end{align}
    where we discarded the conditions $fr' \equiv 1 \pmod{\lambda^7}$, $\mu^2(e)= 1$, $(r^{\prime},abcde)=1$, and $\omega=0 \implies f r^{\prime} \neq 1$ on the right side of \eqref{secmomentintermed}
    by non-negativity.
    
    Recall that $\nu_{fr^{\prime},\omega}(n_1)=\chi_f (n_1) \xi(n_1) \chi_{r^{\prime}}(n_1)$, and that $f = a^2 c^2 e^3$. Apply the quartic large sieve (\cref{quartic-large-sieve}) to the sum over $r^{\prime}$ to deduce that \eqref{secmomentintermed} is
    \begin{align} \label{rsumest}
        & \ll_\varepsilon ((1+|\omega|)X)^{\ep} \Big(\frac{R}{N(e)}+ \mathcal{N}(L,R) + \Big( \frac{R}{N(e)} \mathcal{N}(L,R) \Big)^{2/3} \Big)  \nonumber \\
        & \ll_\varepsilon ((1+|\omega|)X)^{\ep} \Big(R+ (1+|\omega|) (LR)^{1/2} + (1+|\omega|)^{2/3} L^{1/3} R \Big),
    \end{align}
    where we used \eqref{Ndef}. Thus \eqref{rbd1} is
    \begin{equation} \label{rsumest2}
        \ll_\varepsilon ((1+|\omega|)X)^{\ep} \Big(R+ (1+|\omega|)^{1/2} L^{1/4} R^{3/4} + (1+|\omega|)^{1/3} L^{1/6} R \Big).
    \end{equation}
    Substitute \eqref{rsumest2} and \eqref{rsumest} respectively into \eqref{ZRB} and \eqref{ZRA} and estimate the sum over $a,b,c,d,e$ trivially by $X^{\ep} L$ using the divisor bound. This gives
    \begin{align} 
        Z_{L,R,B_{\omega,U}} +Z_{L,R,\widetilde{B}_{\omega,U^{-1}}}  &\ll_\varepsilon \bbone_{\omega=0} \cdot U^{1/2} X^{1/2+\varepsilon} \nonumber \\ 
        & \quad + ((1+|\omega|)X)^{\ep} \Big(LR+ (1+|\omega|)^{1/2} L^{5/4} R^{3/4} + (1+|\omega|)^{1/3} L^{7/6} R \Big) \nonumber
    \end{align}
    and
    \begin{align*} 
        Z_{L,R,A_{\omega}}
        \ll_\varepsilon \bbone_{\omega=0} \cdot X^{3/4+\varepsilon} + ((1+|\omega|)X)^{\varepsilon} \Big(LR+ (1+|\omega|) L^{3/2} R^{1/2} + (1+|\omega|)^{2/3} L^{4/3} R \Big).
    \end{align*}
    Using these two inequalities in \eqref{SRdyadic} finishes the proof of Propositions \ref{SRestimate1} and \ref{SRestimate2}.
\end{proof}


\section{Twisted sums of quartic Gauss sums} \label{Gauss_sums_section}

In this section we will frequently use \cref{quartic-GS-lemma1} and \cref{localcomp} without further reference. Let $\xi$ be the Hecke character defined in \eqref{eq:inf_hecke}. For any squarefree $\alpha \in \Z[i]$ with $\alpha\equiv 1 \pmod{\lambda^3}$, $r \in \lambda^{-2} \mathbb{Z}[i]$, and $v \equiv 1 \pmod{\lambda^3}$, denote
\begin{equation*}
    \psi_\alpha(r,s,\xi; v) := \sum_{\substack{c \in \Z[i] \\ c \equiv v \pmod{4} \\ (c, \alpha)=1}}  \frac{g_4(r, c) \xi(c)}{N(c)^s},
\end{equation*}
which converges absolutely for $\Re(s)>\frac{3}{2}$. Let
\begin{equation}\label{eq:zeta_def}
    \zeta_{\lambda}(s,\xi) := \sum_{\substack{c\in\Z[i] \\ c \equiv 1 \pmod{\lambda^3}}}  \frac{\xi(c)^{4}}{N(c)^s} = L(s, \xi^4) \cdot \Big(1 - \frac{\xi(\lambda)^4}{2^s}\Big),
\end{equation}
which converges absolutely for $\Re(s) > 1$. Also denote $\psi(r, s,\xi;v) := \psi_1(r, s,\xi;v)$,
\begin{equation*}
    \Delta_{\alpha}(s,\xi) := \prod_{\substack{\pi \text{ prime} \\ \pi\equiv 1 \pmod{\lambda^3} \\ \pi \mid \alpha}} \Big(1 - N(\pi)^{3 - 4s}\xi(\pi)^{4} \Big),
\end{equation*}
and
\begin{equation} \label{deltastardef}
    \Delta_{\alpha}^{*}(r, s,\xi):= \prod_{\substack{\pi \text{ prime} \\ \pi \equiv 1 \pmod{\lambda^3} \\ \pi \mid \alpha}}  \Big(1+ g_2\Big(\frac{r}{\pi},\pi\Big) N(\pi)^{1-2s} \xi(\pi)^2 \Big).
\end{equation}
By convention we set $g_2(\frac{r}{\pi}, \pi) = 0$ if $\pi \nmid r$, so $\Delta_{\alpha}^{*}(r, s,\xi)$ is a multiplicative function of $\alpha$. If $\xi$ is trivial then we omit it from the notations defined above. It is convenient to express $\psi_\alpha$ in terms of the simpler $\psi$.

\begin{lemma}[Reducing the coprimality conditions] \label{psi_coprimality_lemma}
    Let $s\in \C$ satisfy $\Re(s) > \frac{3}{2}$, $0 \neq r \in \lambda^{-2} \mathbb{Z}[i]$, and $v \equiv 1 \pmod{\lambda^3}$. For $\alpha, \beta \in \Z[i]$ such that $\alpha, \beta \equiv 1 \pmod{\lambda^3}$, $\mu^2(\alpha)=1$, and $(\alpha, \beta r) = 1$, we have the following. 
    \begin{enumerate}[leftmargin=*]
        \item[(i)] $\psi_{\alpha \beta}( \alpha^3 r, s,\xi;v) \Delta_\alpha(s,\xi) = \psi_{\beta} (\alpha^3 r, s,\xi;v)$.\\
        
        \item[(ii)]
              \begin{align*}
                  &\psi_{\alpha \beta}( \alpha^2 r, s,\xi;v) \Delta_\alpha(s,\xi) \\
                  &= \sum_{\substack{d \in \Z[i] \\ d \equiv 1 \pmod{\lambda^3} \\ d \mid \alpha}} \mu(d) (-1)^{C(d,v)}  \chi_d(-1)\xi(d)^3 N(d)^{2 -3s} \overline{g_4 \Big ( \frac{\alpha^2 r}{d^2}, d \Big)} \psi_{\beta} \Big ( \frac{\alpha^2 r}{d^2}, s,\xi; dv \Big).
              \end{align*}

        \item[(iii)] $\psi_{\alpha \beta}(\alpha r,s,\xi;v) \Delta^{*}_{\alpha}(\alpha r, s,\xi)=\psi_{\beta}(\alpha r,s,\xi;v)$. \\

        \item[(iv)] $\displaystyle\psi_{\alpha \beta}(r, s,\xi;v) \Delta_\alpha(s,\xi) = \sum_{\substack{d \in \Z[i] \\ d \equiv 1 \pmod{\lambda^3} \\ d \mid \alpha}}  \mu(d) (-1)^{C(d,v)} \xi(d) N(d)^{-s}  g_4(r, d) \psi_\beta(rd^2, s,\xi;dv).$
    \end{enumerate}
\end{lemma}

\begin{proof}
    Fix $r$ and $\beta$ as above. The proof is by induction on the number of prime factors of $\alpha$, with the base case $\alpha=1$ being trivial for each item.

    If $\gamma, \pi \equiv 1 \pmod{\lambda^3}$ and $\pi$ is a prime with $\pi \nmid \gamma$, then for any $\rho \in \Z[i]$ we have
    \begin{align} \label{sum}
        \psi_{\gamma} (\rho, s,\xi;v) & = \sum_{j=0}^\infty \sum_{\substack{c \in \Z[i] \\ c \pi^j \equiv v \pmod{4} \\ (c, \gamma \pi) = 1}} \frac{g_4(\rho, c \pi^j) \xi(c\pi^j)}{N(c\pi^j)^s}.
    \end{align}
    Write $\rho = \delta \pi^k$ for $\delta \in \lambda^{-2}\Z[i]$ with $(\delta, \pi)=1$. Using Lemma \ref{localcomp}, observe that for each $k \in \{0,1,2,3\}$, only the summands corresponding to
    $j \in\{ 0, k+1\}$ contribute to the sum in \eqref{sum}. Furthermore, from $c \pi^j \equiv v \pmod{4}$ and $(c,\pi)=1$ we have
    \begin{align*}
        g_4(\delta, c \pi) = (-1)^{C(\pi,v)} g_4(\delta, \pi) g_4(\delta \pi^2, c),
    \end{align*}
    \begin{align*}
        g_4(\delta \pi, c \pi^2) & = (-1)^{C(\pi^2, c)} g_4(\delta \pi, \pi^2) g_4(\delta \pi^5, c) \newline  \\
        & = \overline{\chi_{\pi^2}(\delta)} g_4(\pi,\pi^2) \overline{\chi_c(\pi^5)} g_4(\delta, c) \\
        &= N(\pi) g_2(\delta, \pi) g_4(\delta \pi, c),
    \end{align*}
    \begin{align*}
        g_4(\delta \pi^2, c \pi^3) & = (-1)^{C(\pi^3,c)} g_4(\delta \pi^2, \pi^3) g_4(\delta \pi^8, c) \\
        & = (-1)^{C(\pi,v)} \overline{\chi_{\pi^3}(\delta)} g_4(\pi^2, \pi^3) \overline{\chi_{c}(\pi^8)} g_4(\delta, c) \\
        &=(-1)^{C(\pi,v)} \chi_{\pi}(\delta) \chi_{\pi}(-1) N(\pi)^2 \overline{g_4(\pi)} g_4(\delta,c)  \\
        & =(-1)^{C(\pi,v)} \chi_{\pi}(-1) N(\pi)^2  \overline{g_4(\delta, \pi)} g_4(\delta,c),
    \end{align*}
    and
    \begin{align*}
        g_4(\delta \pi^3,c \pi^4)&=(-1)^{C(\pi^4, c)} g_4(\delta \pi^3,\pi^4) g_4(\delta \pi^{11},c) \\
        &=\overline{\chi_{\pi^4}(\delta)} g_4(\pi^3,\pi^4) \overline{\chi_{c}(\pi^{11})}g_4(\delta,c) \\
        &=-N(\pi)^3 g_4(\delta \pi^3,c).
    \end{align*}

    Using these observations and computations in \eqref{sum} leads respectively to
    \begin{align} \label{pi_0_transform}
        \psi_{\gamma}(\delta, s,\xi;v) = \psi_{\gamma \pi} (\delta, s,\xi;v) + (-1)^{C(\pi,v)} \xi(\pi) \frac{g_4(\delta, \pi)}{N(\pi)^s} \psi_{\gamma \pi} (\delta \pi^2, s,\xi; \pi v),
    \end{align}
    \begin{align}\label{pi_1_transform}
        \psi_{\gamma}(\delta \pi, s,\xi;v) &= \psi_{\gamma \pi} (\delta \pi, s,\xi;v) + \xi(\pi)^2\frac{N(\pi) g_2(\delta, \pi)}{N(\pi)^{2s}} \psi_{\gamma \pi} (\delta \pi, s,\xi;v) \nonumber \\
        &=\Delta_{\pi}^{*}(\delta \pi, s,\xi) \psi_{\gamma \pi} (\delta \pi, s,\xi;v),
    \end{align}
    \begin{align}\label{pi_2_transform}
        & \psi_{\gamma}(\delta \pi^2, s,\xi;v)  = \psi_{\gamma \pi} (\delta \pi^2, s,\xi;v)
        + (-1)^{C(\pi,v)} \xi(\pi)^3 \frac{\chi_{\pi}(-1) N(\pi)^2  \overline{g_4(\delta, \pi)}}{N(\pi)^{3s}} \psi_{\gamma \pi} (\delta, s,\xi; \pi v), \quad
    \end{align}
    and
    \begin{align} \label{pi_3_transform}
        \psi_{\gamma}(\delta \pi^3,s,\xi;v)&=\psi_{\gamma \pi}(\delta \pi^3,s,\xi;v)-\xi(\pi)^{4} \frac{N(\pi)^3}{N(\pi)^{4s}} \psi_{\gamma \pi}(\delta \pi^3,s,\xi;v) \nonumber \\
        & =\Delta_{\pi}(s,\xi) \psi_{\gamma \pi}(\delta \pi^3,s,\xi;v).
    \end{align}
    
    The inductive proofs of items $(i)$ and $(iii)$ follow respectively from \eqref{pi_3_transform} for $(\gamma, \delta) = (\alpha \beta, \alpha^3 r)$
    and \eqref{pi_1_transform} for $(\gamma,\delta)=(\alpha \beta, \alpha r)$.

    Now consider item $(ii)$. Note that $|g_4(\delta, \pi)|^2 = N(\pi)$ and $(-1)^{C(\pi, \pi v)} = (-1)^{C(\pi, v)} \chi_\pi(-1)$. Therefore combining  \eqref{pi_0_transform} (with $v \mapsto \pi v$) and \eqref{pi_2_transform} gives
    \begin{align} \label{remove_pi_1}
        \psi_{\gamma \pi}(\delta \pi^2,s,\xi;v) \Delta_{\pi}(s,\xi) &= \psi_{\gamma}(\delta \pi^2,s,\xi;v) \nonumber \\
        & \quad -(-1)^{C(\pi,v)}  \xi(\pi)^3
        \frac{\chi_{\pi}(-1) N(\pi)^2\overline{g_4(\delta,\pi)}}{N(\pi)^{3s}} \psi_{\gamma}(\delta,s,\xi; \pi v). \qquad
    \end{align}
    Then the proof of item $(ii)$ follows from \eqref{remove_pi_1} with $(\gamma,\delta)=(\alpha \beta,\alpha^2 r)$ and the induction hypothesis, 
    assembling the divisors of $\pi \alpha$ from those of the form $d$ and those of the form $d \pi$, for $d \mid \alpha$. Here it is convenient to use that $C(xy, z) \equiv C(x, z) + C(y, z) \pmod{2}$ for $x, y, z\equiv 1 \pmod{\lambda^3}$.

    Finally, consider item $(iv)$. Another combination of \eqref{pi_0_transform} and \eqref{pi_2_transform} (with $v \mapsto \pi v$) gives
    \begin{equation} \label{remove_pi_2}
        \psi_{\gamma \pi}(\delta,s,\xi;v) \Delta_{\pi}(s,\xi) =\psi_{\gamma}(\delta,s,\xi;v)-(-1)^{C(\pi,v)} \xi(\pi)\frac{g_4(\delta,\pi)}{N(\pi)^s} \psi_{\gamma}(\delta \pi^2,s,\xi;\pi v). 
    \end{equation}
    The proof of item $(iv)$ then follows inductively from \eqref{remove_pi_2} with $(\gamma,\delta)=(\alpha \beta, r)$, as in the previous items.
\end{proof}

\begin{corollary}[Removal of coprimality for twists]\label{general_coprimality_cor}
    Let $s\in \C$ satisfy $\Re(s) > \frac{3}{2}$, $0 \neq r \in \lambda^{-2} \mathbb{Z}[i]$, and $v \equiv 1 \pmod{\lambda^3}$. For $a, b, c, d \in \Z[i]$ such that $a, b, c, d \equiv 1 \pmod{\lambda^3}$, $\mu^2(abcd)=1$, and $(abcd, r) = 1$, we have
    \begin{align*}
        \psi_{abcd}(a b^2 c^3 r, s,\xi;v)
        = \Delta^{*}_a (acr, s, \xi)^{-1} \Delta_{bcd}(s, \xi)^{-1} \mathop{\sum \sum}_{\substack{e, f \in \Z[i] \\ e, f \equiv 1\pmod{\lambda^3} \\ e \mid b,\ f \mid d}} \mu(ef) (-1)^{C(ef,v)} \chi_{e}(-1) & \nonumber \\
        \times \, \xi(e^3f) \frac{N(e)^2}{N(e^3 f)^{s}} \overline{g_4 \Big(\frac{a b^2 c^3 r}{e^2}, e \Big)}  g_4 \Big( \frac{a b^2 c^3 r}{e^2}, f \Big) \psi \Big( \frac{ab^2 c^3 f^2 r}{e^2}, s, \xi;efv \Big). &
    \end{align*}
\end{corollary}

\begin{proof}
    By \cref{psi_coprimality_lemma} $(i)$ we have $\psi_{abcd}(ab^2 c^3 r, s,\xi;v) = \Delta_c(s,\xi)^{-1} \psi_{abd}(ab^2 c^3 r, s,\xi;v)$. Then \cref{psi_coprimality_lemma} $(ii)$ shows that $\psi_{abd}(ab^2 c^3 r, s,\xi;v)$ is equal to
    \begin{equation} \label{rel_2}
        \Delta_b(s,\xi)^{-1}
        \sum_{\substack{e \in\Z[i] \\ e \equiv 1 \pmod{\lambda^3} \\ e \mid b}} \mu(e) (-1)^{C(e,v)} \chi_e(-1) \xi(e)^3
        N(e)^{2-3s} \overline{g_4 \Big( \frac{a b^2 c^3 r}{e^2}, e \Big)} \psi_{ad} \Big( \frac{a b^2 c^3 r}{e^2}, s,\xi;e v \Big).
    \end{equation}
    Next, \cref{psi_coprimality_lemma} $(iii)$ gives $\psi_{ad} \big( \frac{a b^2 c^3 r}{e^2}, s,\xi;ev \big)=\Delta^{*}_a \big( \frac{a b^2 c^3 r}{e^2}, s,\xi \big)^{-1} \psi_d \big( \frac{a b^2 c^3 r}{e^2},s,\xi;ev \big)$. Since $(a,bcr)=1$ and $e \mid b$, we deduce from \eqref{deltastardef} that $\Delta^{*}_a \big( \frac{a b^2 c^3 r}{e^2}, s,\xi \big)=\Delta^{*}_a(a c r, s,\xi)$. Finally, \cref{psi_coprimality_lemma} $(iv)$ yields that $\psi_d \big ( \frac{a b^2 c^3 r}{e^2},s,\xi;ev \big)$ is equal to
    \begin{equation} \label{rel_4}
        \Delta_d(s,\xi)^{-1}
        \sum_{\substack{f \in \Z[i] \\ f \equiv 1 \pmod{\lambda^3} \\ f \mid d}}  \mu(f) (-1)^{C(f,v)} \xi(f) N(f)^{-s}
        g_4 \Big( \frac{ a b^2 c^3 r}{e^2}, f \Big) \psi \Big( \frac{a b^2 c^3 f^2 r}{e^2}, s,\xi;efv \Big).
    \end{equation}
    Note that $\Delta_b(s,\xi)^{-1}  \Delta_c(s,\xi)^{-1}  \Delta_d(s,\xi)^{-1}= \Delta_{bcd}(s,\xi)^{-1}$ since $\mu^2(bcd)=1$. Tracing back through the equations above yields the result.
\end{proof}

\begin{lemma}[Poles and convexity bound] \label{psi_convexity_lemma}
    Let $0 \neq r \in \lambda^{-2} \mathbb{Z}[i] $, $v \equiv 1 \pmod{\lambda^3}$, and $\xi$ be the Hecke character of $\Q(i)$ with frequency $\omega \in \Z$ given in \eqref{eq:inf_hecke}.
    The function $\psi(r, s,\xi;v)$ has meromorphic continuation in $s$ to all of $\C$.
    If $\omega \neq 0$, then $\psi(r, s,\xi;v)$ is holomorphic in the half-plane $\Re(s)>1$.
    If $\omega = 0$, then $\psi(r, s,\xi;v)$ has at most one (simple) pole in the half-plane $\Re(s)>1$,
    at $s=\frac{5}{4}$. Furthermore, for any $0<\varepsilon<\frac{1}{100}$ we have the convexity bound
    \begin{align} \label{convexbd}
        \psi(r,s,\xi;v) & \ll_{\varepsilon, \ord_{\lambda}(r)} \big(N(r)^{1/2} (|s|^3+|\omega|^3+1)\big)^{3/2-\Re(s)+\varepsilon}
    \end{align}
    uniformly for $1+\varepsilon \leq \Re(s) \leq \frac{3}{2}+\varepsilon$ and $\big|s- \frac{5}{4} \big| \geq \varepsilon$.
\end{lemma}

\begin{proof}
    Let
    \begin{equation*}
        G(s,\xi) := \Gamma_\C\Big(s + \frac{|\omega|}{2}-\frac{3}{4}\Big) \Gamma_\C\Big(s+\frac{|\omega|}{2}-\frac{1}{2}\Big) \Gamma_\C\Big(s+\frac{|\omega|}{2}-\frac{1}{4}\Big)
    \end{equation*}
    for $\Gamma_\C(s) := 2 (2\pi)^{-s} \Gamma(s)$, and
    \begin{equation*}
        \widetilde{\psi}(r, s,\xi;v) := G(s,\xi) \zeta_{\lambda}(4s-3) \psi(r, s,\xi;v).
    \end{equation*}
    By \cite[(2.14) and (2.37)]{Dia}, $\widetilde{\psi}(r, s,\xi;v)$ is a constant multiple of
    $\widehat{\Psi}_{i_1 1}(s,r,\xi)$ (if $v \equiv 1 \pmod{4}$) or $\widehat{\Psi}_{i_2 1}(s,r,\xi)$
    (if $v \equiv 1+\lambda^3 \pmod{4}$) defined in \cite[(2.46)]{Dia}. Their meromorphic continuation to $s\in\C$ and pole structure is given in \cite[Theorem 2.1]{Dia}. Then the meromorphic continuation and pole structure of $\psi(r, s,\xi;v)$ follow from the fact that $G(s,\xi) \zeta_{\lambda}(4s-3,\xi)$ is both holomorphic and zero-free in the half-plane $\Re(s)>1$.
    The convexity bound \eqref{convexbd} was computed using the Phragmen--Lindel\"{o}f principle
    in \cite[Proposition~4.3]{DDHL}.
\end{proof}

Let
\begin{equation}
    \tau_4(r;v):=\Res_{s=\frac{5}{4}} \psi(r, s;v).
\end{equation}
No general closed form formula is known for the $\tau_4(\delta;v)$. Patterson \cite{Pat2}, with refinements by Eckhardt--Patterson \cite[Conjecture~2.11]{EckPat}, conjectured essentially that $N(\pi)^{1/4} \tau_4(\pi;v)^2$ is proportional to $\overline{\widetilde{g}_4(\pi)}$ for $\pi \equiv 1 \pmod{4}$ prime. 

Suzuki \cite{Suz1} obtained partial information about these residues. In what follows, let $0 \neq \delta \in \lambda^{-2}\Z[i]$ and $m, v \in \Z[i]$ satisfy $m, v \equiv 1 \pmod{\lambda^3}$. Suzuki \cite[Theorem~3]{Suz1} established the periodicity relation
\begin{equation} \label{quartrel}
    \tau_4(m^4 \delta;v)=\tau_4(\delta;v).
\end{equation}
For $(m, \delta) = 1$ and $\mu^2(m)=1$, he also showed \cite[Theorems 4, 5, 6]{Suz1} that
\begin{equation} \label{cuberel}
    \tau_4(m^3 \delta;v)=0 \qquad \qquad \text{for } m \neq 1,
\end{equation} 
\begin{equation} \label{squarerel}
    \tau_4(m^2 \delta;v)= (-1)^{C(m, v)}  \cdot \frac{\overline{g_4(\delta,m)}}{N(m)^{3/4}} \cdot \tau_4(\delta;mv),
\end{equation}
and $\tau_4(m \delta;v) =0$ unless
\begin{equation} \label{linrel2} 
    g_2 \Big(\frac{m \delta}{m_1},m_1 \Big)=N(m_1)^{1/2} \qquad \text{ for all } m_1 \equiv 1 \pmod{\lambda^3} \text{ such that } m_1 \mid m.
\end{equation} 

We have the following convexity bound for the residues.

\begin{lemma} \label{thetaconvexity}
    Let $\varepsilon>0$, $v \equiv 1 \pmod{\lambda^3}$, and $0 \neq r \in \lambda^{-2} \Z[i]$. Then
    \begin{equation*}
        \tau_4(r;v) \ll_{\varepsilon} N(r)^{1/8+\varepsilon}.
    \end{equation*}
\end{lemma}
\begin{proof}
    This follows from an application of \eqref{convexbd} to a contour integral around $s = \frac{5}{4}$.
\end{proof}

\begin{lemma}[Evaluation of truncated twisted sums of Gauss sums] \label{truncated_Gauss_sums_lemma}
    Let $H :\R_{>0} \to \R_{\geq 0}$ be a smooth and compactly supported function. Let $M >0$,  $\omega \in \Z$, $\xi$ be as in \eqref{eq:inf_hecke},  $v \equiv 1 \pmod{\lambda^3}$, and $s = \sigma +it$ for $0\leq \sigma \leq 2$ and $t \in \R$. Then for any $0 \neq \eta, a, b, c, d \in \Z[i]$ satisfying $a, b, c, d \equiv 1 \pmod{\lambda^3}$, $\mu^2(abcd)=1$, and $\eta \mid \lambda^{\infty}$, we have
    \begin{align}  \label{twisted_Gauss_sum_display}
        & \sum_{\substack{m \in \Z[i] \\ m \equiv v \pmod 4}} \frac{\xi(m) \overline{\chi_m(\eta ab^2c^3 d^4)} g_4(m)}{N(m)^s} H\Big(\frac{N(m)}{M}\Big)  \\
        & \qquad \qquad \qquad \qquad \qquad =\mathcal{T}_M(\eta a b^2 c^3 d^4,s,\xi;v) +
        \mathcal{R}_M(\eta ab^2c^3 d^4,s,\xi;v), \nonumber
    \end{align}
    where
    \begin{align} \label{resstate}
        &  \mathcal{T}_M(\eta a b^2 c^3 d^4,s,\xi;v) = \bbone_{\omega=0} \cdot \bbone_{c=1} \cdot \widetilde{H}(\tfrac{5}{4}-s) M^{\frac{5}{4}-s} \cdot \Delta^{*}_a (a \eta, \tfrac{5}{4})^{-1} \Delta_{bd}(\tfrac{5}{4})^{-1} N(b)^{-\frac{1}{4}} \nonumber \\
        & \times \mathop{\sum \sum}_{\substack{e, f \in \Z[i] \\ e, f \equiv 1\pmod{\lambda^3} \\ e \mid b,\ f \mid d}} (-1)^{C(efv,b)} \chi_{f}(-1) \frac{\mu(ef)}{N(ef)}
        \overline{\widetilde{g}_4 \Big(\frac{a b^2 \eta}{e^2}, e \Big)}  \widetilde{g}_4 \Big( \frac{a b^2 \eta}{e^2}, f \Big) \overline{\widetilde{g}_4 \Big(a \eta,\frac{bf}{e} \Big)} \tau_4(a \eta; bv), \qquad
    \end{align}
    and for any $A \in \Z_{\geq 0}$ and $\varepsilon > 0$ we have
    \begin{align*}
        & \mathcal{R}_M(\eta ab^2c^3 d^4,s,\xi;v) \nonumber \\
        & \ll_{H, A, \varepsilon} M^{1-\sigma+\varepsilon} \mathop{\sum \sum}_{\substack{e, f \in \Z[i] \\ e, f \equiv 1 \pmod{\lambda^3} \\ e \mid b, \ f \mid d}} \frac{N(a)^\varepsilon}{N(ef)^{1/2+\varepsilon}}
        \int_{-\infty}^\infty \frac{\big|\psi\big(\frac{\eta ab^2 c^3 f^2}{e^2}, 1+\varepsilon + i(t+y), \xi;efv\big)\big|}{(1+|y|)^A} \, dy.
    \end{align*}
    Here $\widetilde{H}(w)$ denotes the Mellin transform of $H$, namely $\widetilde{H}(w) = \int_{0}^\infty H(x) x^{w} \frac{dx}{x}$.
\end{lemma}

\begin{proof}
    For $\Re(w) >\frac{3}{2}$ we have
    \begin{align*}
        \sum_{\substack{m \in \Z[i] \\ m \equiv v \pmod{\lambda^3}}} \frac{ \xi(m) \overline{\chi_m(\eta ab^2c^3 d^4)} g_4(m)}{N(m)^{w}} & = \sum_{\substack{m \in \Z[i] \\ m \equiv v \pmod{4} \\ (m, abcd)=1}} \frac{\xi(m) g_4(\eta ab^2c^3, m)}{N(m)^{w}} \nonumber \\
        & = \psi_{abcd}(\eta ab^2c^3, w,\xi;v).
    \end{align*}
    Applying \cref{general_coprimality_cor} and \cref{psi_convexity_lemma} to $\psi_{abcd}(\eta ab^2c^3, w,\xi;v)$, we deduce that it is holomorphic for $\Re(w) > 1$, except for at most a simple pole at $w=\frac{5}{4}$ which can only occur if $\omega = 0$.

    By Mellin inversion, the Dirichlet series in \eqref{twisted_Gauss_sum_display} is equal to
    \begin{equation}
        \frac{1}{2\pi i} \int_{2 -i\infty}^{2+i\infty} \psi_{abcd}(\eta a b^2 c^3 ,w,\xi;v)  \widetilde{H}(w-s) M^{w-s} dw,
    \end{equation}
    where the Mellin transform $\widetilde{H}$ is rapidly decaying uniformly on vertical strips. We shift the line of integration to $\Re(w) = 1+\varepsilon$. Only when $\xi$ is trivial do we pick up a possible simple pole at $w =\frac{5}{4}$, which by \cref{general_coprimality_cor} has residue
    \begin{align} \label{resintermed}
        &\bbone_{\omega=0} \cdot \widetilde{H}(\tfrac{5}{4}-s) M^{\frac{5}{4}-s} \cdot \Delta^{*}_a (a c \eta, \tfrac{5}{4})^{-1} \Delta_{bcd}(\tfrac{5}{4})^{-1}  \nonumber \\
        &\times \mathop{\sum \sum}_{\substack{e, f \in \Z[i] \\ e, f \equiv 1\pmod{\lambda^3} \\ e \mid b,\ f \mid d}} \mu(ef) (-1)^{C(ef,v)}  \chi_{e}(-1) \frac{ \overline{g_4 \big(\frac{a b^2 c^3 \eta}{e^2}, e \big)}  g_4 \big( \frac{a b^2 c^3 \eta}{e^2}, f \big)}{N(e)^{7/4} N(f)^{5/4}} \tau_4\Big( \frac{ab^2 c^3 f^2 \eta}{e^2};efv \Big).  \qquad 
    \end{align}
    Since $\mu^2(abcd)=1$ and $(abcd,\eta)=1$, we deduce from \eqref{cuberel} that $\tau_4 \big( \tfrac{ab^2 c^3 f^2 \eta}{e^2};efv \big)= 0$ unless $c=1$. Furthermore, \eqref{squarerel} implies 
    \begin{equation} \label{tau4intermed}
        \tau_4 \Big( \frac{ab^2 f^2 \eta}{e^2};efv \Big) = (-1)^{C\left(\frac{bf}{e}, efv\right)} \frac{\overline{g_4\big(a \eta,\frac{bf}{e}\big)}}{N\big(\frac{bf}{e}\big)^{3/4}} \tau_4(a\eta; bv).
    \end{equation}
    A quick computation gives $(-1)^{C(ef, v) + C\left(\frac{bf}{e}, efv\right)} = (-1)^{C(efv, b)}\chi_{ef}(-1)$. Substituting \eqref{tau4intermed} into \eqref{resintermed} and writing the expression in terms of normalized Gauss sums, we obtain $\mathcal{T}_M(\eta a b^2 c^3 d^4,\xi;v)$ given in \eqref{resstate}.

    For the remaining integral over $\Re(w) = 1+\varepsilon$, we simply use \cref{general_coprimality_cor} and bound the Gauss sums using \cref{quartic-GS-lemma1} $(i)$ and \eqref{sqrootcancel}, as $\mu^2(abcd)=1$. Since
    \begin{equation*}
        \Delta_{bcd}(1+\varepsilon+iy,\xi) \gg_\varepsilon 1, \qquad  \Delta^*_{a}(ac \eta, 1+\varepsilon+iy,\xi) \gg_\varepsilon N(a)^{-\varepsilon},
    \end{equation*}
    and for any $A \in \Z_{\geq 0}$ we have
    \begin{equation*}
        \widetilde{H}(1-\sigma + \varepsilon + iy) \ll_{H, A, \varepsilon} (1 + |y|)^{-A}
    \end{equation*}
    uniformly for $y \in \R$ and $0 \leq \sigma \leq 2$, we obtain the desired result.
\end{proof}


\section{Second moment asymptotics}\label{sec:second_moment}

\begin{proof}[Proof of \cref{prop:second_moment}]
    Recall the definitions of $A_{\omega}(q)$ in \eqref{Aomegadef} and of $\mathcal{S}_M$ in \eqref{SMdef}.
    Thus
    \begin{align} \label{eq:S_M_1}
        \mathcal{S}_M\big(\nu_{q,\omega}(\mathfrak{b}_1)\overline{\nu_{q,\omega}(\mathfrak{b}_2)}A_\omega(q);F\big) =\sum_{\substack{q \in \Z[i]\\ q \equiv 1 \pmod {\lambda^7} }}M_Y(q) \nu_{q,\omega}(\mathfrak{b}_1)\overline{\nu_{q,\omega}(\mathfrak{b}_2)}A_{\omega}(q)F\Big(\frac{N(q)}X\Big) \nonumber \\
        = \sum_{0 \ne \mathfrak{n}_1,\mathfrak{n}_2 \unlhd \Z[i]}N(\mathfrak{n}_1\mathfrak{n}_2)^{-1/2}\mathcal{S}_M( \nu_{q,\omega}(\mathfrak{b}_1\mathfrak{n}_1)\overline{\nu_{q, \omega}(\mathfrak{b}_2\mathfrak{n}_2)};F_{\mathfrak{n}_1,\mathfrak{n}_2,\omega}), \quad
    \end{align}
    where
    \begin{equation}\label{eq:Fn1n2}
        F_{\mathfrak{n}_1, \mathfrak{n}_2,\omega}(t):= F(t) W_\omega\Big(\frac{N(\mathfrak{n}_1\mathfrak{n}_2)}{4N(\mathfrak{m}_{\omega})Xt}\Big).
    \end{equation}
    By \eqref{eq:nu_q} and \eqref{MYRYdef} we have
    \begin{align*}
        & \mathcal{S}_M\big( \nu_{q,\omega}(\mathfrak{b}_1\mathfrak{n}_1)\overline{\nu_{q,\omega}(\mathfrak{b}_2\mathfrak{n}_2});F_{\mathfrak{n}_1,\mathfrak{n}_2,\omega}\big) \\
        &= \xi(\mathfrak b_1\mathfrak{n}_1)\overline{\xi(\mathfrak b_2 \mathfrak{n}_2)}\sum_{ \substack{\ell \in \Z[i] \\ \ell \equiv 1 \pmod {\lambda^3} \\ N(\ell) \le Y}}\mu(\ell) \sum_{\substack{m\in \Z[i] \\ \ell^2m\equiv 1 \pmod{\lambda^7}}}\chi_{\ell^2m}(\mathfrak b_1\mathfrak{n}_1) \overline{\chi_{\ell^2m}(\mathfrak b_2\mathfrak{n}_2)}F_{\mathfrak{n}_1,\mathfrak{n}_2,\omega}\Big(\frac{N(\ell^2m)}{X}\Big).
    \end{align*}
    For $i \in \{1,2\}$, we may write $\mathfrak{b}_i = b_i\Z[i]$ and $\mathfrak{n}_i =\lambda^{g_i} n_i \Z[i]$ for some $g_i \in \Z_{\ge 0}$ and $b_i,n_i\equiv 1 \pmod{\lambda^3}$. Since the $\mathfrak{b}_i$ are squarefree, their primary generators are as well.  Since $\ell^2m \equiv 1 \pmod {\lambda^7}$, observe that $\chi_{\ell^2m}(\lambda) =1$ and by quartic reciprocity
    \begin{align*}
        \mathcal{S}_M\big( \nu_{q,\omega}(\mathfrak{b}_1\mathfrak{n}_1) \overline{\nu_{q, \omega}(\mathfrak{b}_2\mathfrak{n}_2)}; F_{\mathfrak{n}_1,\mathfrak{n}_2,\omega}\big) = \xi(b_1\lambda^{g_1}n_1)\overline{\xi(b_2\lambda^{g_2}n_2)} \sum_{\substack{\ell \in \Z[i] \\ \ell \equiv 1 \pmod{\lambda^3}\\N(\ell) \le Y}}\mu(\ell) \quadrat{\ell}{b_1n_1b_2n_2}\\
        \times \sum_{\substack{m\in \Z[i] \\ \ell^2m \equiv 1 \pmod {\lambda^7}}} \chi_{b_1n_1}(m) \overline{\chi_{b_2n_2}(m)}F_{\lambda^{g_1}n_1, \lambda^{g_2}n_2,\omega}\Big(\frac{N(\ell^2m)}{X}\Big).
    \end{align*}
    Set $(b_1n_1,b_2n_2) =: d \equiv 1 \pmod{\lambda^3}$ and $r_i := \frac{b_in_i}{d}$ for $i\in \{1,2\}$. Then
    \begin{equation}
        \begin{split}
            \mathcal{S}_M\big(\nu_{q,\omega}(\mathfrak{b}_1\mathfrak{n}_1) \overline{\nu_{q,\omega}(\mathfrak{b}_2\mathfrak{n}_2)};F_{\mathfrak{n}_1,\mathfrak{n}_2,\omega}\big) = \xi(\lambda^{g_1}r_1)\overline{\xi(\lambda^{g_2}r_2)} \sum_{\substack{\ell \in \Z[i]\\\ell \equiv 1 \pmod{\lambda^3}\\N(\ell) \le Y}}\mu(\ell) \quadrat{\ell}{d^2r_1r_2}\\
            \times \sum_{\substack{ m \in \Z[i] \\ \ell^2m \equiv 1 \pmod {\lambda^7}}} \chi_{dr_1}(m) \overline{\chi_{dr_2}(m)}F_{\lambda^{g_1}n_1, \lambda^{g_2}n_2,\omega}\Big(\frac{N(\ell^2m)}{X}\Big).
        \end{split}
    \end{equation}
    Insert this into \eqref{eq:S_M_1} and use the relation $n_i = \frac{dr_i}{b_i}$ for $i \in \{1, 2\}$ to arrive at
    \begin{align}
        &\mathcal{S}_M\big(\nu_{q,\omega}(\mathfrak b_1)\overline { \nu_{q,\omega}(\mathfrak b_2)}A_{\omega}(q);F\big) \nonumber \\
        & = N(b_1b_2)^{1/2}\sumtwo_{g_1,g_2\in \Z_{\ge 0}}\frac{\xi(\lambda^{g_1})\overline {\xi(\lambda^{g_2})}}{N(\lambda^{g_1+g_2})^{1/2}} \sum_{\substack{d \in \Z[i]\\ d \equiv 1 \pmod{\lambda^3} }}\frac{1}{N(d)}  \sumtwo_{\substack{ r_1, r_2 \in \mathbb{Z}[i] \\ r_1, r_2 \equiv 1 \pmod{\lambda^3} \\ (r_1, r_2)=1 \\ b_i \mid d r_i }}
        \frac{\xi(r_1)\overline{ \xi(r_2)}}{N(r_1r_2)^{1/2}} \nonumber \\
        & \times \sum_{\substack{\ell \in \Z[i] \\ \ell \equiv 1 \pmod {\lambda^3} \\ (\ell, d)=1 \\ N(\ell) \leq Y}} \mu(\ell) \quadrat{\ell}{r_1 r_2} \sum_{\substack{ m \in \Z[i] \\  \ell^2 m \equiv 1 \pmod{\lambda^7}}} \psi_{dr_1,dr_2}(m)
        F_{\frac{\lambda^{g_1} d r_1}{b_1},\frac{\lambda^{g_2} d r_2}{b_2},\omega } \Big( \frac{N(\ell^2 m)}{X}  \Big), \label{eq:S_M_2}
    \end{align}
    where for $a,b \equiv 1 \pmod {\lambda^3}$ we define $\psi_{a,b}(\cdot) := \chi_{a}(\cdot)\overline{\chi_{b}}(\cdot)$. Note that the sum over $g_1$ and $g_2$ reduces to a single element $g_1=g_2=0$ unless $4 \mid \omega$.
    We may uniquely decompose $r_i =d_i m_i$ for $d_i,m_i \equiv 1 \pmod{\lambda^3}$ with $d_i\mid d^\infty$ and $(m_i,d) = 1$ for $i \in \{1,2\}$. Then $(m_1,m_2) = (d_1,d_2)=1$  and $b_i \mid dd_im_i$. In preparation to apply Poisson summation, note that
    \begin{equation}\label{eq:perpsi}
        \psi_{d r_1, d r_2} = \chi_{d r_1} \overline{\chi_{d r_2}} = \chi_{r_1} \overline{\chi_{r_2}}\mathbf{1}_d = \chi_{m_1} \overline{\chi_{m_2}} \chi_{d_1}\overline{\chi_{d_2}}\mathbf{1}_{e},
    \end{equation}
    where $e := \frac{\ra(d)}{\ra(d_1d_2)}$. Applying \cref{lem:pois} with period $m_1m_2d_1d_2e$ and observing that $\overline{\ell^2} \equiv \ell^2 \pmod{\lambda^7}$ we find that the $m$-sum in \eqref{eq:S_M_2} is equal to
    \begin{align} \label{poissonapp}
        \frac{\pi X}{64 N(\ell^2 m_1 m_2 d_1 d_2 e)} \sum_{k \in \mathbb{Z}[i]} \ddot{\psi}_{d d_1 m_1, d d_2 m_2}(k) \check{e}\Big({-\frac{k\ell^2 (m_1 m_2 d_1 d_2 e)^3}{2 \lambda^7}}\Big) & \nonumber \\
        \times \, \ddot{F}_{\frac{\lambda^{g_1} d d_1 m_1}{b_1},\frac{\lambda^{g_2} d d_2 m_2}{b_2},\omega }
        \Big( \frac{k \sqrt{X}}{\ell^2 m_1 m_2 d_1 d_2 e} \Big ).
    \end{align}
    Using the Chinese remainder theorem for the pairwise coprime moduli $m_1,m_2,$ and $d_1d_2e$, we compute from the definition \eqref{eq:ddotV} that
    \begin{equation}\label{eq:ddpsi_1}
        \begin{split}
            \ddot{\psi}_{d d_1 m_1, d d_2 m_2}(k) :=& \ \sum_{b \pmod{m_1m_2d_1d_2e}} \chi_{m_1}\overline{\chi_{m_2}} \chi_{d_1}\overline{\chi_{d_2}}\mathbf{1}_{e}(2\lambda^7b) \check e \Big({-\frac{kb}{m_1m_2d_1d_2e}}\Big) \\
            =& \ \chi_{m_1}(2\lambda^7m_2d_1d_2e) \overline{\chi_{m_2}(2\lambda^7m_1d_1d_2e)} \chi_{d_1}\overline{\chi_{d_2}}(2\lambda^7m_1m_2) g_4(-k,m_1)\overline{g_4(k,m_2)} \\
            & \ \times \sum_{a \pmod{d_1d_2e}}\chi_{d_1}\overline{\chi_{d_2}}\mathbf{1}_e(a)\check e \Big({-\frac{ka}{d_1d_2e}}\Big).
        \end{split}
    \end{equation}
    By quartic reciprocity, we have
    \begin{equation}
        \chi_{m_1}(m_2d_1d_2e) \overline{\chi_{m_2}(m_1d_1d_2e)} \chi_{d_1}\overline{\chi_{d_2}}(m_1m_2) = (-1)^{C(m_1,m_2,d_1,d_2)} \quadrat{d_1}{m_1}\quadrat{d_2}{m_2}\psi_{m_1, m_2}(e),
    \end{equation}
    where
    \begin{equation}\label{eq:rec}
        (-1)^{C(m_1,m_2,d_1,d_2)}  := (-1)^{C(m_1,m_2)+C(m_1,d_1)+C(m_1,d_2)+C(m_2,d_1)+C(m_2,d_2)}.
    \end{equation}
    Furthermore, $\chi_{m_1}\chi_{d_1} \overline{\chi_{m_2}} \overline{\chi_{d_2}} (2\lambda^7) = \psi_{m_1 d_1,m_2d_2}(-i\lambda)$. By the Chinese remainder theorem applied to the pairwise coprime moduli $d_1,d_2$, and $e$, quartic reciprocity, and the evaluation of Ramanujan sums (where $e$ is squarefree) analogous to \cite[Lemma 5.5]{DR}, we have
    \begin{align*}
        &\sum_{a \pmod{d_1d_2e}}\chi_{d_1}\overline{\chi_{d_2}}\mathbf{1}_e(a)\check e \Big({-\frac{ka}{d_1d_2e}}\Big) \\
        &= (-1)^{C(d_1,d_2)} \psi_{d_1, d_2}(e) g_4(-k,d_1)\overline{g_4(k,d_2)} \cdot \frac{\mu\big(\frac{e}{(e,k)}\big) \varphi(e)}{\varphi\big(\frac{e}{(e, k)}\big)}.
    \end{align*}
    Overall, again since $e$ is squarefree we conclude that
    \begin{equation}
        \begin{split}
            \ddot{\psi}_{dd_1m_1,dd_2m_2}(k) &= (-1)^{C(m_1,m_2,d_1,d_2) +C(d_1,d_2)} \psi_{m_1d_1,m_2d_2}(-i\lambda e)\quadrat{d_1}{m_1}\quadrat{d_2}{m_2} \\
            & \quad \times g_4(-k,m_1) g_4(-k,d_1)\overline{g_4(k,m_2)g_4(k,d_2)}\mu\Big(\frac{e}{(e,k)}\Big)\varphi((e,k)).
        \end{split}
    \end{equation}
    From \eqref{eq:quar_repr}, \eqref{eq:ram&units_1}, and \eqref{eq:ram&units_2}, the quantity in \eqref{eq:rec} and $\chi_{m_i}(-i\lambda)$ are both $\lambda^7$-periodic. Hence, to interchange the sum over $k$ and the sums over the $m_i$, we split the $m_i$ into congruence classes modulo $\lambda^7$. This gives
    \begin{equation}\label{eq:pre_split}
        \begin{split}
            & \mathcal{S}_M  \big(\nu_{q,\omega}(\mathfrak{b}_1)\overline{\nu_{q,\omega}(\mathfrak{b}_2)}A_{\omega}(q);F\big) = \frac{\pi X}{64}N(b_1b_2)^{1/2}\sumtwo_{g_1,g_2\in \Z_{\ge 0}} \frac{\xi(\lambda^{g_1})\overline{\xi(\lambda^{g_2})}}{2^{(g_1+g_2)/2}} \\
            & \times \sum_{\substack{d\in \Z[i]\\ d\equiv 1 \pmod{\lambda^3}}}\frac{1}{N(d)}
            \sumtwo_{\substack{d_1,d_2\in \Z[i] \\ d_1,d_2 \equiv 1 \pmod{\lambda^3} \\ (d_1,d_2)=1 \\ d_1d_2\mid d^{\infty}}} (-1)^{C(d_1,d_2)}  \frac{\xi(d_1) \overline{\xi(d_2)} \psi_{d_1,d_2}(e)}{{N(d_1d_2)^{3/2}}N(e)} \\
            & \times \sumtwo_{\substack{c_1,c_2 \pmod{\lambda^7}\\ c_1,c_2 \equiv 1 \pmod {\lambda^3}}}(-1)^{C(c_1,c_2,d_1,d_2)} \psi_{c_1d_1,c_2d_2}(-i\lambda)  \sum_{k\in \Z[i]} \mu\Big(\frac{e}{(e,k)}\Big) \varphi((e,k)) g_4(-k,d_1)\overline{g_4(k,d_2)}\\
            & \times \sum_{\substack{\ell \in \Z[i] \\ \ell \equiv 1 \pmod{\lambda^3} \\ (\ell,d)=1 \\ N(\ell) \le Y}} \ec \Big({-\frac{k\ell^2(c_1c_2d_1d_2e)^3}{2\lambda^7}}\Big)\quadrat{\ell}{d_1d_2} \frac{\mu(\ell)}{N(\ell)^2} \\
            & \times \sumtwo_{\substack{m_1,m_2 \in \Z[i]\\ m_i \equiv c_i \pmod{\lambda^7} \\ (m_1m_2,d)=1 \\ b_i\mid dd_i m_i \\(m_1,m_2)=1}} \psi_{m_1,m_2}(e) \xi(m_1) \overline{\xi(m_2)} \quadrat{d_1}{m_1}\quadrat{d_2}{m_2}\quadrat{\ell}{m_1m_2} \frac{g_4(-k,m_1)\overline{g_4(k,m_2)}}{N(m_1m_2)^{3/2}} \\
            &  \times \ddot{F}_{\frac{\lambda^{g_1} d d_1 m_1}{b_1},\frac{\lambda^{g_2} d d_2 m_2}{b_2},\omega }
            \Big( \frac{k \sqrt{X}}{\ell^2 m_1 m_2 d_1 d_2 e} \Big).
        \end{split}
    \end{equation}
    The main term will arise from $k = 0$, so we write
    \begin{equation}\label{eq:M(b1b2)R(b1b2)}
        \mathcal{S}_M \big(\nu_{q,\omega}(\mathfrak{b}_1)\overline{\nu_{q,\omega}(\mathfrak{b}_2})A_\omega(q);F\big) = \mathcal M_{\omega}(b_1,b_2) + \mathcal {R}_{\omega}(b_1,b_2),
    \end{equation}
    where $\mathcal M_{\omega}(b_1,b_2)$ corresponds to all terms with $k=0$ and $\mathcal R_{\omega}(b_1,b_2)$ corresponds to all terms with $0\ne k \in \Z[i]$ in \eqref{eq:pre_split}.
    
    
    \subsection{The main term $\mathcal M_{\omega}(b_1,b_2)$} 
    
    The Gauss sum $g_4(0,c)$ vanishes unless $c$ is a fourth power, in which case $g(0,c) = \varphi(c)$. Hence in \eqref{eq:pre_split} we make the change of variables $d_i\mapsto d_i^4$ and $m_i \mapsto m_i^4$ for $i \in \{1,2\}$. Furthermore, if $c \equiv 1 \pmod{\lambda^3}$ then $c^4 \equiv 1 \pmod{\lambda^7}$ and $(-1)^{C(c^4, v)} = 1$ for any $v\equiv1\pmod{\lambda^3}$. Thus we have
    \begin{equation}\label{eq:main_term_expression}
        \begin{split}
            &\mathcal M_{\omega}(b_1,b_2) \\
            &= \frac{\pi X}{64}N(b_1b_2)^{1/2} \sumtwo_{g_1, g_2\in \Z_{\ge 0}} \frac{\xi(\lambda^{g_1})\overline{\xi(\lambda^{g_2})}}{2^{(g_1+g_2)/2}} \sum_{\substack{d\in \Z[i]\\ d\equiv 1 \pmod{\lambda^3}}}\frac{1}{N(d)} \sumtwo_{\substack{d_1,d_2\in \Z[i] \\ d_1,d_2 \equiv 1 \pmod{\lambda^3} \\ (d_1,d_2)=1 \\ d_1d_2\mid d^{\infty}}} \xi(d_1^4)\overline{\xi(d_2^4)}\ \frac{\varphi(e)}{N(e)} \\
            &\times \frac{\varphi(d_1^4d_2^4)}{N(d_1d_2)^6} \sumtwo_{\substack{m_1,m_2 \in \Z[i] \\ m_1,m_2 \equiv 1 \pmod{\lambda^3} \\ (m_1m_2,d)=1 \\ b_i\mid dd_i^4 m_i^4 \\(m_1,m_2)=1}}\xi(m_1^4)\overline{\xi(m_2^4)}\frac{\varphi(m_1^4m_2^4)}{N(m_1m_2)^6}  \ddot{F}_{\frac{\lambda^{g_1} d d_1^4 m_1^4}{b_1},\frac{\lambda^{g_2} d d_2^4 m_2^4}{b_2},\omega }
            (0) \sum_{\substack{\ell \in \Z[i]\\ \ell \equiv 1 \pmod{\lambda^3} \\ (\ell,dm_1m_2)=1 \\ N(\ell)\le Y}}\frac{\mu(\ell)}{N(\ell)^2}.
        \end{split}
    \end{equation}
    The sum over $\ell$ is
    \begin{equation}\label{eq: ell_sum_eval}
        \sum_{\substack{ \ell \in \Z[i] \\ \ell \equiv 1 \pmod{\lambda^3} \\ N(\ell) \leq Y}} \frac{\mu(\ell) \mathbf{1}_{d m_1 m_2}(\ell)}{N(\ell)^2} = \zeta_\lambda(2)^{-1} \prod_{\substack{\pi \text{ prime} \\ \pi \equiv 1 \pmod{\lambda^3} \\ \pi \mid d m_1 m_2}} \Big(1 - \frac{1}{N(\pi)^2}\Big)^{-1} + O\Big(\frac{1}{Y}\Big).
    \end{equation}

    Let us now evaluate $F_{\mathfrak{n}_1, \mathfrak{n}_2,\omega}(0)$. Recall that $F_{\mathfrak{n}_1, \mathfrak{n}_2,\omega}(t) = F(t) W_{\omega} \big(\frac{N(\mathfrak{n}_1 \mathfrak{n}_2)}{4N(\m_{\omega})Xt} \big)$ and  $F$ has support in $(1, 2)$. Since $J_0(0)=1$, by \eqref{eq:ddotV} and \eqref{eq:W_f} we have
    \begin{align}
        \ddot{F}_{\mathfrak{n}_1, \mathfrak{n}_2,\omega}(0) & = \int_1^{\sqrt{2}}  r F(r^2) W_{\omega} \Big(\frac{N(\mathfrak{n}_1 \mathfrak{n}_2)}{4N(\m_{\omega})Xr^2} \Big) dr \nonumber \\
        & =  \frac{1}{2\pi i} \int_{2 -i\infty}^{2 + i\infty} \Big(\frac{\pi^2 N(\mathfrak{n}_1 \mathfrak{n}_2)}{N(\m_{\omega})X} \Big)^{-w} e^{w^2} \frac{\Gamma(\frac{1}{2}+\frac{|\omega|}{2}+w)^2}{\Gamma(\frac{1}{2}+\frac{|\omega|}{2})^2} \check{F}(w) \frac{dw}{2w}, \label{F_dot_dot_zero}
    \end{align}
    where $\check F(w) : = \int_0^{\infty}t^{w}F(t)dt = 2 \int_1^{\sqrt{2}}r^{2w+1}F(r^2)dr$.  By the decay properties of $W_{\omega}$ described in \eqref{W_f_bound} and the fact that $0 \le F(t) \le 1$, for every $A \in \Z_{\geq 0}$ we have
    \begin{equation}\label{F_dot_dot_zero_bound}
        \ddot{F}_{\mathfrak n_1,\mathfrak n_2,\omega}(0) \ll_{A}\Big(1+\frac{N(\mathfrak n_1 \mathfrak n_2)}{(1+|\omega|)^2X}\Big)^{-A}.
    \end{equation}
    To bound the error term in \eqref{eq: ell_sum_eval}, we use a version of \cite[Lemma 7.1]{DDDS24} for $\Z[i]$.
    \begin{lemma}\label{gcd_sum_lemma}
        For any $C \geq 1/2$, $1\geq \Delta \geq 0$, $\delta \in \R$, and $g \in \Z[i]$ with $g \equiv 1 \pmod{\lambda^3}$,
        \begin{equation*}
            \sum_{\substack{c\in \Z[i] \\ c\equiv 1 \pmod{\lambda^3} \\ N(c) \sim C}} \frac{N((g, c))^\Delta}{N(c)^\delta} \ll_{\delta, \varepsilon} N(g)^\varepsilon C^{1 - \delta}.
        \end{equation*}
    \end{lemma}
    Observe that the condition $b_i \mid d d_i^4 m_i^4$ is equivalent to $\frac{b_i}{(b_i, d)} \mid m_i$, since $\mu^2(b_i) = 1$ and $d_i \mid d^\infty$. Inserting \eqref{eq: ell_sum_eval} into \eqref{eq:main_term_expression}, we can use \eqref{F_dot_dot_zero_bound} to conclude that the error term arising from $O(\frac{1}{Y})$ contributes
    \begin{align}
        &\ll \frac{X N(b_1b_2)^{1/2}}{Y} \sumthree_{\substack{d, d_1, d_2 \in \Z[i] \\ d, d_1, d_2 \equiv 1 \pmod{\lambda^3} \\ (d_1, d_2)=1 \\ d_1 d_2 \mid d^\infty}} \sumtwo_{\substack{m_1, m_2 \in \Z[i] \\ m_1, m_2 \equiv 1 \pmod{\lambda^3} \\ (m_1, m_2) = (m_i, d) = 1 \\ \frac{b_i}{(b_i, d)} \mid m_i}} \frac{\big(1 + \frac{N(d)^2 N(d_1 m_1 \cdot d_2 m_2)^4}{N(b_1b_2)(1+|\omega|)^2 X} \big)^{-100}}{N(d) N(d_1 m_1 \cdot d_2 m_2)^{2}} \nonumber \\
        &\ll \frac{X N(b_1b_2)^{1/2}}{Y} \sum_{\substack{d \in \Z[i] \\ d \equiv 1 \pmod{\lambda^3}}} \frac{\big(1 + \frac{N(d^2)}{N(b_1b_2)(1+|\omega|)^2 X} \big)^{-100}}{N(d)} \frac{N((b_1, d)(b_2, d))^{2}}{N(b_1 b_2)^{2}} \nonumber \\
        &\ll \frac{X}{Y} \sum_{\substack{d \in \Z[i] \\ d \equiv 1 \pmod{\lambda^3}}} \frac{N((b_1, d))+ N((b_2, d))}{N(d)} \Big(1 + \frac{N(d^2)}{N(b_1b_2)(1+|\omega|)^2 X} \Big)^{-100} \ll_\ep \frac{(1+|\omega|)^{\ep}X^{1+\varepsilon}}{Y} \nonumber,
    \end{align}
    where the last step follows from 
    \cref{gcd_sum_lemma}.

    We use $\eqref{F_dot_dot_zero}$ with $n_i = \frac{d\lambda^{g_i}d_i^4m_i^4}{b_i}$ for $i\in \{1,2\}$ in \eqref{eq:main_term_expression} to find
    \begin{equation}\label{eq:M_contour}
        \begin{split}
            \mathcal M_{\omega}(b_1,b_2) &= \frac{\pi X}{64 \cdot \zeta_{\lambda}(2)}\frac{1}{2\pi i} \int_{(2)} \Big(\frac{\pi^2}{N(\m_{\omega})X}\Big)^{-w}\Big(\frac{\Gamma(\frac{1+|\omega|}{2}+w)}{\Gamma(\frac{1+|\omega|}{2})}\Big)^2 e^{w^2} \check F(w) \mathcal G_{b_1,b_2,\omega}(w)\frac{dw}{2w}\\
            & \quad  + O_{\ep}\Big(\frac{(1+|\omega|)^{\ep}X^{1+\ep}}Y\Big),
        \end{split}
    \end{equation}

    where

    \begin{equation}\label{eq:Gb1b2}
        \begin{split}
            &\mathcal{G}_{b_1,b_2,\omega}(w) := N(b_1b_2)^{\tfrac 1 2+w} \Big(1-\frac{\xi(\lambda)}{2^{1/2+w}}\Big)^{-1}\Big(1-\frac{\overline{\xi(\lambda)}}{2^{1/2+w}}\Big)^{-1}\sum_{\substack{d\in \Z[i]\\d \equiv 1 \pmod{\lambda^3}}}\frac{1}{N(d)^{1+2w}}\\
            &\times  \Bigg(\prod_{\substack{\pi \ \text{prime } \\ \pi \equiv 1 \pmod{\lambda^3}\\ \pi \mid d}}\Big(1-\frac{1}{N(\pi)^2}\Big)^{-1} \Bigg)\Bigg(\sumtwo_{\substack{d_1,d_2 \in \Z[i]\\ d_1,d_2 \equiv 1 \pmod{\lambda^3} \\ (d_1,d_2)=1 \\ d_1d_2 \mid d^{\infty}}} \frac{\varphi(e)}{N(e)}\frac{\varphi(d_1^4d_2^4) \xi(d_1^4)\overline{\xi(d_2^4)}}{N(d_1d_2)^{6+4w}}\Bigg) \\
            & \times \sumtwo_{\substack{m_1,m_2 \in \Z[i] \\ m_1,m_2 \equiv 1 \pmod{\lambda^3}\\(m_1m_2,d)=1 \\ \frac{b_i}{(b_i,d)}\mid m_i \\(m_1,m_2)=1}} \frac{\varphi(m_1^4m_2^4) \xi(m_1^4)\overline{\xi(m_2^4)}}{N(m_1m_2)^{6+4w}}  \prod_{\substack{\pi \ \text{prime} \\ \pi \equiv 1 \pmod{\lambda^3} \\\pi |m_1m_2}}\Big(1-\frac{1}{N(\pi)^2}\Big)^{-1}.
        \end{split}
    \end{equation}
    
    Define $b_i':= \frac{b_i}{(b_i,d)}$ for $i \in \{1, 2\}$, so that $(b'_i, d)=1$ since $b_1$ and $b_2$ are squarefree. Note that the sums over $m_1$ and $m_2$ in  \eqref{eq:Gb1b2} are empty unless $(b'_1,b'_2)=1$. Assuming this is the case, the sums over $m_1$ and $m_2$ are then equal to
    \begin{align*}
        &  \sum_{\substack{m_1 \in \Z[i] \\ m_1 \equiv 1 \pmod{\lambda^3} \\ (m_1, d) = 1 \\  b'_1 \mid  m_1}} \frac{\varphi(m_1^4)\xi(m_1^4)}{N(m_1)^{6+4w}} \sum_{\substack{m_2 \in \Z[i] \\ m_2 \equiv 1 \pmod{\lambda^3} \\ (m_2, d) = 1 \\  b'_2 \mid m_2 \\ (m_1, m_2)=1}} \frac{\varphi(m_2^4)\overline{\xi(m_2^4)}}{N(m_2)^{6+4w}} \prod_{\substack{\pi \text{ prime} \\ \pi \equiv 1 \pmod{\lambda^3} \\ \pi \mid m_1 m_2}} \Big(1 - \frac{1}{N(\pi)^2}\Big)^{-1} \\
        & = \prod_{\substack{\pi \text{ prime} \\ \pi \equiv 1 \pmod{\lambda^3} \\ \pi \nmid d b'_1 b'_2}}  \Bigg( 1 + \sum_{k=1}^\infty \frac{\varphi(\pi^{4k}) \big(\xi(\pi^{4k}) + \overline{\xi(\pi^{4k})}\big) \big(1 - \frac{1}{N(\pi)^2}\big)^{-1}}{N(\pi)^{(6+4w)k}} \Bigg) \\
        & \quad \times \prod_{\substack{\pi \ \text{prime}\\ \pi \equiv 1 \pmod{\lambda^3}\\ \pi \mid b'_1}} \Bigg( \sum_{k=1}^{\infty} \frac{\varphi(\pi^{4k}) \xi(\pi^{4k}) \big(1 - \frac{1}{N(\pi)^2}\big)^{-1}}{N(\pi)^{(6+4w)k}} \Bigg) \\
        & \quad \times \prod_{\substack{\pi \ \text{prime}\\ \pi \equiv 1 \pmod{\lambda^3}\\ \pi \mid b'_2}} \Bigg( \sum_{k=1}^{\infty} \frac{\varphi(\pi^{4k}) \overline{\xi(\pi^{4k})} \big(1 - \frac{1}{N(\pi)^2}\big)^{-1}}{N(\pi)^{(6+4w)k}} \Bigg).
    \end{align*}
    We introduce the convenient notations
    \begin{equation}\label{eq:F(pi,w)}
        \begin{split}
            F_{1,\omega}(\pi,w) & := \sum_{k=1}^{\infty} \frac{\varphi(\pi^{4k}) \xi(\pi^{4k}) \big(1 - \frac{1}{N(\pi)^2}\big)^{-1}}{N(\pi)^{(6+4w)k}} = \frac{\big(1 + \frac{1}{N(\pi)}\big)^{-1}}{N(\pi)^{2+4w}\overline{\xi(\pi^{4})}-1}, \\
            F_{2,\omega}(\pi,w) & := \sum_{k=1}^{\infty} \frac{\varphi(\pi^{4k}) \overline{\xi(\pi^{4k})} \big(1 - \frac{1}{N(\pi)^2}\big)^{-1}}{N(\pi)^{(6+4w)k}} = \frac{\big(1 + \frac{1}{N(\pi)}\big)^{-1}}{N(\pi)^{2+4w}\xi(\pi^{4})-1}, \\
            E_{\omega}(\pi,w) & := 1 + F_{1, \omega}(\pi,w) + F_{2, \omega}(\pi,w), \\
            \mathcal{E}_{\omega}(w) & := \prod_{\substack{\pi \ \text{prime}\\ \pi \equiv 1 \pmod{\lambda^3}}}E_\omega(\pi,w).
        \end{split}
    \end{equation}
    Thus the sums over $m_1$ and $m_2$ in \eqref{eq:Gb1b2} are equal to
    \begin{align*}
        \mathcal{E}_{\omega}(w) \prod_{\substack{\pi \: \text{prime}\\ \pi \equiv 1 \pmod{\lambda^3} \\ \pi \mid d b'_1 b'_2}} E^{-1}_{\omega}(\pi,w) & \prod_{\substack{\pi \ \text{prime}\\ \pi \equiv 1 \pmod{\lambda^3} \\ \pi \mid b_1'}} F_{1,\omega}(\pi,w)  \prod_{\substack{\pi \ \text{prime}\\ \pi \equiv 1 \pmod{\lambda^3}\\ \pi \mid b_2'}} F_{2,\omega}(\pi,w).
    \end{align*}
    
    The next step is to write the sums over $d_1$ and $d_2$ in \eqref{eq:Gb1b2} as an Euler product. Recall that $e = \frac{\ra(d)}{\ra(d_1d_2)}$, hence from the identity $\frac{\varphi(\ra(r))}{N(\ra(r))} = \frac{\varphi(r)}{N(r)}$ for $r\equiv 1 \pmod{\lambda^3}$ we obtain
    \begin{align*}
        &\sumtwo_{\substack{d_1,d_2 \in \Z[i]\\ d_1,d_2 \equiv 1 \pmod{\lambda^3} \\ (d_1,d_2)=1 \\ d_1d_2 \mid d^{\infty}}} \frac{\varphi(e)}{N(e)}\frac{\varphi(d_1^4d_2^4) \xi(d_1^4)\overline{\xi(d_2^4)}}{N(d_1d_2)^{6+4w}} = \frac{\varphi(d)}{N(d)} \sumtwo_{\substack{ d_1,d_2 \in \Z[i]\\ d_1,d_2 \equiv 1 \pmod{\lambda^3}\\ (d_1,d_2)=1\\ d_1 d_2 \mid d^{\infty}}} \frac{\xi(d_1^4)\overline{\xi(d_2^4)}}{N(d_1d_2)^{2+4w}} \\
        &= \frac{\varphi(d)}{N(d)} \prod_{\substack{\pi \ \text{prime}\\ \pi \equiv 1 \pmod{\lambda^3} \\ \pi  \mid d}} \Big(1+\sum_{k=1}^{\infty} \frac{\xi(\pi^{4k})+\overline{\xi(\pi^{4k})}}{N(\pi)^{(2+4w)k}}\Big) =: \frac{\varphi(d)}{N(d)} \prod_{\substack{\pi \ \text{prime} \\ \pi \equiv 1 \pmod{\lambda^3}\\\pi  \mid d}}H_{\omega}(\pi,w).
    \end{align*}
    Therefore
    \begin{align}
        \mathcal G_{b_1,b_2,\omega}(w) & = \mathcal{E}_{\omega}(w)N(b_1b_2)^{\frac{1}{2}+w}   \Big(1-\frac{\xi(\lambda)}{2^{1/2+w}}\Big)^{-1}\Big(1-\frac{\overline{\xi(\lambda)}}{2^{1/2+w}}\Big)^{-1} \nonumber \\
        & \quad \times \sum_{\substack{d \in \Z[i]\\ d\equiv 1 \pmod{\lambda^3}\\ (b'_1,b'_2)=1}}\frac{1}{N(d)^{1+2w}} \prod_{\substack{\pi \ \text{prime}\\ \pi \equiv 1 \pmod{\lambda^3} \\ \pi \mid d}} \Big(1+\frac{1}{N(\pi)}\Big)^{-1}(E_{\omega}^{-1}H_{\omega})(\pi,w) \qquad \nonumber \\
        & \quad \times  \prod_{\substack{\pi \ \text{prime}\\ \pi \equiv 1 \pmod{\lambda^3} \\ \pi \mid b'_1}}(E_{\omega}^{-1}F_{1,\omega})(\pi,w)\prod_{\substack{\pi \ \text{prime}\\ \pi \equiv 1 \pmod{\lambda^3} \\ \pi \mid b'_2}}(E_{\omega}^{-1}F_{2,\omega})(\pi,w). \label{eq:Gb1b2_int}
    \end{align}
    
    Set $b:=(b_1,b_2)$, so that $(b'_1, b'_2)=1$ if and only if $b \mid d$. Then we can uniquely decompose $d = bb'd'$ for $b', d' \equiv 1 \pmod{\lambda^3}$ such that $b' \mid b^{\infty}$ and $(d',b)=1$. Thus the sum over $d$ in \eqref{eq:Gb1b2_int} equals 
    \begin{equation}\label{eq:d_sum_interm}
        \sumtwo_{\substack{b', d' \in \Z[i] \\ b', d' \equiv 1 \pmod{\lambda^3} \\ b'\mid b^\infty,\ (d', b)=1}} \frac{1}{N(bb'd')^{1+2w}} \prod_{\substack{\pi \ \text{prime}\\ \pi \equiv 1 \pmod{\lambda^3}\\ \pi  \mid bd'}} \frac{N(\pi) \cdot (E_{\omega}^{-1}H_{\omega})(\pi,w)}{N(\pi)+1}\prod_{i=1}^2\prod_{\substack{ \pi \ \text{prime} \\\pi \equiv 1 \pmod{\lambda^3}\\ \pi  \mid b_i, \ \pi \nmid bd'}} (E_{\omega}^{-1}F_{i,\omega})(\pi,w).
    \end{equation}
    The sum over $b'$ can now be evaluated: since $b$ is squarefree we have
    \begin{align*}
        \sum_{\substack{b' \in \Z[i]\\b' \equiv 1 \pmod{\lambda^3}\\ b' \mid b^{\infty}}} \frac{1}{N(bb')^{1+2w}}= \prod_{\substack{\pi \ \text{prime}\\ \pi \equiv 1 \pmod{\lambda^3}\\ \pi  \mid b}} \frac{1}{N(\pi)^{1+2w}-1},
    \end{align*}
    hence \eqref{eq:d_sum_interm} is equal to
    \begin{align*}
        &\Bigg(\prod_{\substack{\pi \ \text{prime}\\ \pi \equiv 1 \pmod{\lambda^3}\\ \pi  \mid b}} \Big(1+\frac{1}{N(\pi)}\Big)^{-1} \frac{(E_{\omega}^{-1}H_{\omega})(\pi,w)}{N(\pi)^{1+2w}-1}\Bigg) \Bigg( \prod_{i=1}^2\prod_{\substack{ \pi \ \text{prime} \\\pi \equiv 1 \pmod{\lambda^3}\\ \pi  \mid b_i, \ \pi \nmid b}} (E_{\omega}^{-1}F_{i,\omega})(\pi,w) \Bigg) \\
        & \times \sumtwo_{\substack{d' \in \Z[i] \\ d' \equiv 1 \pmod{\lambda^3} \\ (d', b)=1}} \frac{1}{N(d')^{1+2w}} \prod_{\substack{\pi \ \text{prime}\\ \pi \equiv 1 \pmod{\lambda^3}\\ \pi  \mid d'}} \frac{N(\pi) \cdot (E_{\omega}^{-1}H_{\omega})(\pi,w)}{N(\pi)+1}\prod_{i=1}^2\prod_{\substack{ \pi \ \text{prime} \\\pi \equiv 1 \pmod{\lambda^3}\\ \pi  \mid (b_i, d')}} (E_{\omega}F^{-1}_{i,\omega})(\pi,w).
    \end{align*}
    We rewrite the sum over $d'$ above as 
    \begin{align*}
        &\prod_{\substack{\pi \ \text{prime}\\ \pi \equiv 1 \pmod{\lambda^3}\\ \pi \nmid b_1b_2}}\Big(1+\Big(1+\frac{1}{N(\pi)}\Big)^{-1}\frac{(E_{\omega}^{-1}H_{\omega})(\pi,w)}{N(\pi)^{1+2w}-1}\Big)\\
        & \times \prod_{i=1}^{2} \prod_{\substack{\pi \ \text{prime} \\ \pi \equiv 1 \pmod{\lambda^3}\\ \pi  \mid b_i, \ \pi\nmid b}}\Big(1+\Big(1+\frac{1}{N(\pi)}\Big)^{-1}\frac{(F_{i,\omega}^{-1}H_{\omega})(\pi,w)}{N(\pi)^{1+2w}-1}\Big),
    \end{align*}
    which evaluates to
    \begin{align*}
        & \prod_{\substack{\pi \ \text{prime}\\ \pi \equiv 1 \pmod{\lambda^3}}}\Big(1+\Big(1+\frac{1}{N(\pi)}\Big)^{-1}\frac{(E_{\omega}^{-1}H_{\omega})(\pi,w)}{N(\pi)^{1+2w}-1}\Big)\\
        & \times \prod_{\substack{\pi \ \text{prime}\\ \pi \equiv 1 \pmod{\lambda^3}\\ \pi \mid b_1b_2}}\Big(1+\Big(1+\frac{1}{N(\pi)}\Big)^{-1}\frac{(E_{\omega}^{-1}H_{\omega})(\pi,w)}{N(\pi)^{1+2w}-1}\Big)^{-1}\\
        & \times \prod_{i=1}^{2} \prod_{\substack{\pi \ \text{prime} \\ \pi \equiv 1 \pmod{\lambda^3}\\ \pi  \mid  \frac{b_i}{b}}}\Big(1+\Big(1+\frac{1}{N(\pi)}\Big)^{-1}\frac{(F_{i,\omega}^{-1}H_{\omega})(\pi,w)}{N(\pi)^{1+2w}-1}\Big).
    \end{align*}
    
    In conclusion, we have
    \begin{equation}\label{eq:Gb1b2_final}
        \mathcal{G}_{b_1,b_2,\omega}(w)=\frac{2^{1+2w}\mathcal{F}_\omega(w)\mathcal{H}_{b_1,b_2,\omega}(w)}{(2^{1/2+w}-\xi(\lambda))(2^{1/2+w}-\overline{\xi(\lambda)})},
    \end{equation}
    for
    \begin{equation}\label{eq:defF(w)}
        \mathcal{F}_{\omega}(w):= \prod_{\substack{\pi \ \text{prime} \\ \pi \equiv 1 \pmod{\lambda^3} }}E_{0,\omega}(\pi,w)
    \end{equation}
    and
    \begin{equation}\label{eq:Hb1b2}
        \begin{split}
            \mathcal{H}_{b_1,b_2,\omega}(w):=N(b_1b_2)^{1/2+w}& \prod_{\substack{\pi \ \text{prime}\\ \pi \equiv 1 \pmod{\lambda^3}\\ \pi \mid \frac{b_1}b}}E_{1,\omega}(\pi,w)\prod_{\substack{\pi \ \text{prime}\\ \pi \equiv 1 \pmod{\lambda^3} \\ \pi \mid \frac{b_2}b}}E_{2,\omega}(\pi,w) \\
            & \prod_{\substack{\pi \ \text{prime}\\ \pi \equiv 1 \pmod{\lambda^3} \\ \pi \mid b}} E_{3,\omega}(\pi,w) \prod_{\substack{\pi \ \text{prime}\\ \pi \equiv 1 \pmod{\lambda^3} \\ \pi \mid b_1b_2}} E_{4,\omega}(\pi,w) ,
        \end{split}
    \end{equation}
    where assuming from now on that $\Re(w) \ge -\frac{1}{2}+\ep$,
    \begin{equation}\label{eq:E_i}
        \begin{split}
            & E_{0,\omega}(w) := \Big(1+\Big(1+\frac{1}{N(\pi)}\Big)^{-1}\frac{(E_{\omega}^{-1}H_{\omega})(\pi,w)}{N(\pi)^{1+2w}-1}\Big)E_{\omega}(\pi,w),\\
            & E_{1,\omega}(w) := \Big(1+\Big(1+\frac{1}{N(\pi)}\Big)^{-1}\frac{(F_{1,\omega}^{-1}H_{\omega})(\pi,w)}{N(\pi)^{1+2w}-1}\Big)(E_{\omega}^{-1}F_{1,\omega})(\pi,w) \ll_\varepsilon \frac{1}{N(\pi)^{1+2\Re(w)}},\\
            & E_{2,\omega}(w) := \Big(1+\Big(1+\frac{1}{N(\pi)}\Big)^{-1}\frac{(F_{2,\omega}^{-1}H_{\omega})(\pi,w)}{N(\pi)^{1+2w}-1}\Big)(E_{\omega}^{-1}F_{2,\omega})(\pi,w) \ll_\varepsilon \frac{1}{N(\pi)^{1+2\Re(w)}}, \\
            & E_{3,\omega}(w):=\Big(1+\frac{1}{N(\pi)}\Big)^{-1}\frac{(E_{\omega}^{-1}H_{\omega})(\pi,w)}{N(\pi)^{1+2w}-1} \ll_\varepsilon \frac{1}{N(\pi)^{1+2\Re(w)}}, \\
            & E_{4,\omega}(w):= \Big(1+\Big(1+\frac{1}{N(\pi)}\Big)^{-1}\frac{(E_{\omega}^{-1}H_{\omega})(\pi,w)}{N(\pi)^{1+2w}-1}\Big)^{-1} \ll_\varepsilon 1.
        \end{split}
    \end{equation}
    Using the definition of $E_{0,\omega}(w)$, we compute 
    \begin{align}
        & E_{0,\omega}(\pi,w)\Big(1-\frac{1}{N(\pi)^{1+2w}}\Big)  = 1 - \frac{1}{N(\pi)^{1+2w}} \nonumber \\
        & \quad + \Big(1+\frac{1}{N(\pi)}\Big)^{-1}\Big(\frac{1}{N(\pi)^{1+2w}} + \frac{1}{N(\pi)^{2+4w}\overline{\xi(\pi^4)}-1}  + \frac{1}{N(\pi)^{2+4w}\xi(\pi^4)-1}\Big) \qquad \label{eq:Flocfact}\\
        & = 1 + \frac{1}{N(\pi)^{2+4w}\overline{\xi(\pi^4)}-1} + \frac{1}{N(\pi)^{2+4w}\xi(\pi^4)-1} + O_{\ep}\Big(\frac{1}{N(\pi)^{1+\ep}}\Big) \nonumber \\
        & = \Big(1-\frac{1}{N(\pi)^{4+8w}}\Big)\Big(1-\frac{\xi(\pi^4)}{N(\pi)^{2+4w}}\Big)^{-1}\Big(1-\frac{\overline{\xi(\pi^4)}}{N(\pi)^{2+4w}}\Big)^{-1}\Big(1+O_{\ep}\Big(\frac{1}{N(\pi)^{1+\ep}}\Big)\Big). \nonumber
    \end{align}
    Thus using the zeta functions given in \eqref{eq:zeta_def} gives 
    \begin{equation} \label{eq:F(w)_eval}
        \mathcal{F}_{\omega}(w)= \frac{\zeta_{\lambda}(1+2w)\zeta_{\lambda}(2+4w,\xi)\zeta_{\lambda}(2+4w,\overline{\xi})}{\zeta_{\lambda}(4+8w)}\mathcal{J}_{\omega}(w),
    \end{equation}
    for some function $\mathcal J_{\omega} (w)$ that is is holomorphic, non-vanishing, and uniformly bounded (in terms of $w$ and $\f$) for $\Re(w)\ge -\frac{1}{2} +\ep$. Thus the function $\mathcal{F}_{\omega}(w)$ has meromorphic continuation to that half-plane, with a simple pole at $w =0$. When $\omega$ is zero, it also has a double pole at $w = -\frac{1}{4}$. It has no other poles for $\Re(w) \geq -\frac{3}{8}$. 
    
    We move the contour in \eqref{eq:M_contour} to $\Re(w) = -\frac{1}{4}+\ep$, picking up the simple pole at $w =0$. In the remaining integral over $\Re w = -\frac{1}{4}+\ep$, we use the convexity bound for $\zeta_{\lambda}(1+2w)$, Stirling's formula \cite[(5.11.3)]{DLMF}, the exponential decay of $e^{w^2}$, and bound $\mathcal{H}_{b_1, b_2, \omega}(w)$ trivially. Thus
    \begin{align}
        \mathcal M_{\omega}(b_1,b_2)  & =  \frac{\pi X}{128 \cdot \zeta_{\lambda}(2)} \Res_{w =0} \Big(\frac{\pi^2}{N(\m_\omega)X}\Big)^{-w}\Big(\frac{\Gamma(\frac{1+|\omega|}{2} +w)}{\Gamma(\frac{1+|\omega|}{2})}\Big)^2 e^{w^2} \check{F}(w) \frac{\mathcal G_{b_1,b_2,\omega}(w)}{w} \nonumber\\
        & \quad + O_{\ep}\bigg(\frac{(1+|\omega|)^{\ep}X^{1+\ep}}Y + (1+|\omega|)^{-\frac{1}{2}+2\ep}X^{3/4 +\ep}N\Big(\frac{b_1 b_2}{b^2}\Big)^{-1/4}\bigg). \label{eq:Mb1b2res}
    \end{align}
    
    At $w=0$ we have the series expansions
    \begin{align*}
        &\Big(\frac{\pi^2}{N(\m_{\omega})X}\Big)^{-w}  = 1 + \log \Big(\frac{N(\m_{\omega})X}{\pi ^2}\Big) w+ O_X(w^2), \\
        &\Big(\frac{\Gamma(\frac{1+|\omega|} 2+w)}{\Gamma(\frac{1+|\omega|} 2)}\Big)^2 = 1 + 2\frac{\Gamma'}{\Gamma}\Big(\frac{1+|\omega|}{2}\Big)w + O_{\omega}(w^2), \\
        & \frac{2^{1+2w}}{(2^{1/2+w}-\xi(\lambda))(2^{1/2+w}-\overline{\xi(\lambda)})}= \frac{2}{|\sqrt 2 -\xi(\lambda)|^2} - \frac{4\log(2)(\sqrt{2}\Re{\xi(\lambda)} - |{\xi(\lambda)}|^2)}{|\sqrt{2}-\xi(\lambda)|^4} w + O_{\omega}(w^2). \\
    \end{align*}
    By \eqref{eq:zeta_def} and the class number formula we also compute
    \begin{equation}\label{eq:zetalambdares}
        2\Res_{w=0} \zeta_{\lambda}(1+2w) = \Res_{s=1}\zeta_{\lambda}(s) = \frac{1}{2} \Res_{s=1}\zeta_{\Q(i)}(s) = \frac{\pi}{8},
    \end{equation}
    so there is an absolute constant $c$ such that close to $w=0$ we have
    \begin{equation*}
        \frac{\zeta_{\lambda}(1+2w)}{w} =\frac{\pi}{16 w^2} + \frac{c}{w}+O(1).
    \end{equation*}
    Similarly, one computes the series expansions at $w=0$ given by
    \begin{align*}
        & \check F(w) = \check F(0) \Big(1 + \frac{\check F'}{\check F}(0) w + O_F(w^2)\Big), \\
        & \mathcal H_{b_1,b_2,\omega}(w) = \mathcal H_{b_1,b_2}(0)\Big(1 + \frac{\mathcal H_{b_1,b_2,\omega}'}{\mathcal H_{b_1,b_2,\omega}}(0) w +O_{b_1,b_2,\omega}(w^2)\Big),\\
        & \mathcal P_{\omega}(w):= \frac{\zeta_{\lambda}(2+4w,\xi)\zeta_{\lambda}(2+4w,\overline{\xi})}{\zeta_{\lambda}(4 +8w)}\mathcal{J}_{\omega}(w) = \mathcal{P_{\omega}}(0)\Big(1 + \frac{\mathcal P'_{\omega}}{\mathcal P_{\omega}}(0) w + O_{\omega}(w^2)\Big).
    \end{align*}
    Since $e^{w^2}$ is even, by \eqref{eq:Gb1b2_final} the residue at $w=0$ in \eqref{eq:Mb1b2res} is equal to 
    \begin{equation}\label{eq:res}
        \begin{split}
            \frac{\pi \check F(0)\mathcal P_{\omega}(0) \mathcal{H}_{b_1,b_2,\omega}(0)}{8|\sqrt{2} - \xi(\lambda)|^2} \Big(\log(X) + 2\frac{\Gamma'}{\Gamma}\Big(\frac{1+|\omega|}{2}\Big) + \frac{\mathcal P'_{\omega}}{\mathcal P_{\omega}}(0)+\frac{\mathcal{H}_{b_1,b_2,\omega}'}{\mathcal{H}_{b_1,b_2,\omega}}(0) + C_{\omega,F} \Big),
        \end{split}
    \end{equation}
    where $C_{\omega,F} \ll_{F} 1$ uniformly in $\omega$.
    
    Using \eqref{eq:Flocfact} and \eqref{eq:F(w)_eval} we can directly evaluate
    \begin{equation}\label{eq:P(0)}
        \mathcal P_{\omega}(0) = \prod_{\substack{\pi \ \text{prime}\\ \pi \equiv 1 \pmod{\lambda^3}\\ q:= N(\pi)}} \Big( 1 - \frac{1}{q(q+1)} + \frac{q}{(q+1)(q^2\xi(\pi^4)-1)} + \frac{q}{(q+1)(q^2\overline{\xi(\pi^4)}-1)} \Big).
    \end{equation}
    The next step is to evaluate $\mathcal H_{b_1,b_2,\omega}(0)$ and $\frac{\mathcal{H}_{b_1,b_2,\omega}'}{\mathcal{H}_{b_1,b_2,\omega}}(0)$. For $q:= N(\pi)$, we compute 
    \begin{equation}\label{eq:E_1E_3}
        (E_{1,\omega} E_{4,\omega})(\pi,w) = \frac{q^{6 w+4} + q^{4 w+3}(\xi(\pi^4)+1) + q^{2 w+2}}{q^{8 w+4}(q+1) + q^{6 w+4} + q^{4 w+2} (q-2 \Re{\xi(\pi^4)})+q^{2 w+2}+1},
    \end{equation}
    \begin{equation}\label{eq:E_2E_3}
        (E_{2,\omega}E_{4,\omega})(\pi,w) = \frac{q^{6 w+4} + q^{4 w+3}(\overline{\xi(\pi^4)}+1) + q^{2 w+2}}{q^{8 w+4}(q+1) + q^{6 w+4} + q^{4 w+2} (q-2 \Re{\xi(\pi^4)})+q^{2 w+2}+1},
    \end{equation}
    \begin{equation}\label{eq:E_3E_4}
        (E_{3,\omega}E_{4,\omega})(\pi,w)= \frac{q \left(q^{4 w+2}+1\right) \left(q^{2 w+1}+1\right)}{q^{8 w+4}(q+1) + q^{6 w+4} + q^{4 w+2} (q-2 \Re{\xi(\pi^4)}) + q^{2 w+2} + 1}.
    \end{equation}
    For $j\in \{1,2, 3\}$ this implies
    \begin{equation}\label{eq:lderEiE3}
        \frac{(E_{j,\omega}E_{4,\omega})'}{(E_{j,\omega} E_{4,\omega})}(\pi,0) = -2\log q + D_{j, \omega}(\pi) \frac{\log q}{q} \qquad \text{ with } \qquad D_{j, \omega}(\pi) \ll 1
    \end{equation}
    uniformly in $\f$. We deduce from \eqref{eq:E_1E_3}, \eqref{eq:E_2E_3}, \eqref{eq:E_3E_4}, and \eqref{eq:Hb1b2} that 
    \begin{equation}\label{eq:Hb1b2(0)}
        \mathcal H_{b_1,b_2,\omega}(0) = N\Big(\frac{b_1b_2}{b^2}\Big)^{-1/2} g_\f(b)h_{\omega}\Big(\frac{b_1}{b}\Big)\overline{h_{\omega}}\Big(\frac{b_2}{b}\Big),
    \end{equation}
    where $g$ and $h$ are multiplicative functions defined, for each $k \ge 1$, by
    \begin{equation}
        g_{\omega}(\pi^k) :=q(E_{3,\omega}E_{4,\omega})(\pi,0) = \frac{q^2(q^2+1)(q+1)}{q^5+2q^4+q^3+(1-2\Re{\xi(\pi^4)})q^2+1} = 1 + O\Big(\frac{1}{q}\Big)
    \end{equation}
    and 
    \begin{equation}
        h_{\omega}(\pi^k):= q(E_{1,\omega}E_{4,\omega})(\pi,0) = \frac{q^3(q^2+q(\xi(\pi^4)+1) + 1)}{q^5+2 q^4+q^3+(1-2\Re{\xi(\pi^4)})q^2+1}  = 1 + O\Big(\frac{1}{q}\Big).
    \end{equation}
    The implied constants are again independent of $\omega$. Furthermore, \eqref{eq:lderEiE3} and \eqref{eq:Hb1b2} imply
    \begin{equation}\label{eq:lderHb1b2}
        \begin{split}
            \frac{\mathcal H'_{b_1,b_2,\omega}}{H_{b_1,b_2,\omega}}(0) = &-\log N\Big(\frac{b_1b_2}{b^2}\Big) + \sum_{i=1}^2 \sum_{\substack{\mathfrak{p} \text{ prime} \\ \mathfrak{p} \mid \frac{b_i}{b}}} D_{i, \omega}(\mathfrak{p}) \frac{\log{N(\mathfrak{p})}}{N(\mathfrak{p})} + \sum_{\substack{\mathfrak{p} \text{ prime} \\ \mathfrak{p} \mid b}} D_{3, \omega}(\mathfrak{p}) \frac{\log{N(\mathfrak{p})}}{N(\mathfrak{p})}.
        \end{split}
    \end{equation}
    In a similar fashion, using \eqref{eq:Flocfact} one shows the uniform bound 
    \begin{equation}\label{eq:P'/P}
        \frac{\mathcal P'_{\omega}}{\mathcal P_{\omega}}(0) \ll \sum_{\mathfrak{p} \text{ prime}} \frac{\log{N(\mathfrak{p})}}{N(\mathfrak{p})^2} \ll 1.
    \end{equation}
    An application of Stirling's formula \cite[(5.11.2)]{DLMF} gives
    \begin{equation}\label{eq:Psi}
        2\frac{\Gamma'}{\Gamma}\Big(\frac{1+|\omega|}{2}\Big) = 2\log(1+|\omega|) + O(1).
    \end{equation}
    In light of \eqref{eq:P'/P} and \eqref{eq:Psi}, we redefine $C_{\omega,F}$ to include $\frac{\mathcal P'_{\omega}}{\mathcal P_{\omega}}(0)$ and the term $O(1)$ above.
    
    Insert \eqref{eq:P(0)}, \eqref{eq:Hb1b2(0)},  \eqref{eq:lderHb1b2}, and \eqref{eq:Psi} into \eqref{eq:res} and finally \eqref{eq:Mb1b2res} to conclude that 
    \begin{equation*}
        \begin{split}
            \mathcal M_{\omega} (b_1,b_2) & = D_{\omega}\check F (0) X N\Big(\frac{b_1b_2}{b^2}\Big)^{-1/2}g_\f(b) h_{\omega}\Big(\frac{b_1}{b}\Big)\overline{h_{\omega}}\Big(\frac{b_2}{b}\Big) \nonumber \\
            & \quad \times \Big [ \log \Big(\frac{(1+|\omega|)^2 XN(b^2)}{N(b_1b_2)}\Big)+ \mathcal O_{\omega}(b_1,b_2)\Big]\\
            & \quad + O_{\ep}\Big(\frac{(1+|\omega|)^{\ep}X^{1+\ep}}Y+(1+|\omega|)^{-\frac{1}{2}+2\ep} X^{3/4 +\ep}N\Big(\frac{b_1 b_2}{b^2}\Big)^{-1/4}\Big),
        \end{split}
    \end{equation*}
    where
    \begin{align}
        D_{\omega}& := \frac{\pi^2 \mathcal P_{\omega}(0)}{|\sqrt 2 -\xi(\lambda)|^2\cdot 8  \cdot  128   \cdot \zeta_{\lambda}(2)} \\
        & =  \frac{\pi^2}{1024 \cdot \zeta_{\lambda}(2)\cdot |\sqrt 2 -\xi(\lambda)|^2} \prod_{\substack{\pi \equiv 1 \pmod{\lambda^3}\\ q:= N(\pi)}} \Big( 1 -\frac{1}{q(q+1)} +  2 \Re \Big(\frac{q}{(q+1)(q^2\xi(\pi)^4-1)}\Big)\Big)
    \end{align}
    and for $i \in \{1, 2, 3\}$ there exist $C_{\omega,F} \ll 1$, $D_{i, \omega}(\pi) \ll 1$ (uniformly on $\omega$) such that
    \begin{equation*}
        \mathcal{O}_{\omega}(b_1,b_2) :=  C_{\omega,F} + \sum_{i=1}^2 \sum_{\substack{\mathfrak{p} \text{ prime} \\ \mathfrak{p} \mid \frac{b_i}{b}}} D_{i, \omega}(\mathfrak{p}) \frac{\log{N(\mathfrak{p})}}{N(\mathfrak{p})} + \sum_{\substack{\mathfrak{p} \text{ prime} \\ \mathfrak{p} \mid b}} D_{3, \omega}(\mathfrak{p}) \frac{\log{N(\mathfrak{p})}}{N(\mathfrak{p})}.
    \end{equation*}
    
    \subsection{The error term \texorpdfstring{$\mathcal{R}_{\omega}(b_1, b_2)$}{}: initial manipulations}

    Recall that $N(b_1) \sim B_1$,  $N(b_2) \sim B_2$, and $B := \max(B_1, B_2)$, where $1 \leq B_1, B_2 \leq X^{100}$. We will use the notation $\ls$ given in \eqref{specnotation}. In order to separate the variables $m_1$ and $m_2$ in \eqref{eq:pre_split}, we must deal with the condition $(m_1, m_2) = 1$ and the Archimedean transform $\ddot{F}$. This can be done by standard maneuvers via M{\"o}bius inversion and Mellin inversion, respectively. We include the analogous lemma to \cite[Lemma 7.3]{DDDS24}.
    \begin{lemma}\label{F_dot_dot_bound_lemma}
        For any $\mathfrak{n}_1, \mathfrak{n}_2 \unlhd \Z[i]$, $u \in \C$, and $A \in \Z_{\geq 0}$, we have the uniform bound
        \begin{align*}
            \ddot{F}_{\mathfrak{n}_1, \mathfrak{n}_2,\omega}(u) \ll_{F, A} \Big(1 + |u| + \frac{N(\mathfrak{n}_1 \mathfrak{n}_2)}{(1+|\omega|)^2 X}\Big)^{-A}.
        \end{align*}
    \end{lemma}

    \begin{proof}
        From \eqref{eq:ddotV} and integration by parts, antidiferentiating the Bessel function as in \cite[(4.11)]{DR}, for any $j \in \Z_{\geq 0}$ (and assuming $u\neq 0$ if $j \neq 0$) we have
        \begin{equation*}
            \ddot{F}_{\mathfrak{n}_1, \mathfrak{n}_2,\omega}(u) = (-1)^j \Big(\frac{8\sqrt{2}}{\pi}\Big)^j \frac{1}{|u|^j} \int_0^\infty F_{\mathfrak{n}_1, \mathfrak{n}_2,\omega}^{(j)} (r^2) \cdot r^{j+1} J_j\Big(\frac{\pi r |u|}{4\sqrt 2}\Big) dr.
        \end{equation*}
        Recall from \eqref{eq:Fn1n2} that $F_{\mathfrak{n}_1, \mathfrak{n}_2,\omega}(t) = F(t) W_\omega \big(\frac{N(\mathfrak{n}_1 \mathfrak{n}_2)}{4N(\mathfrak{m}_{\omega}) Xt} \big)$. Then \eqref{W_f_bound} and the fact that $F$ has support in $(1, 2)$ give
        \begin{equation*}
            F_{\mathfrak{n}_1, \mathfrak{n}_2,\omega}^{(j)}(t) \ll_{F, j, A} \Big(1 + \frac{N(\mathfrak{n}_1 \mathfrak{n}_2)}{(1+|\omega|)^2 X}\Big)^{-A}.
        \end{equation*}
        Putting those together for $j=0$ if $|u| \leq 1$ or $j=A$ if $|u|\geq 1$ gives the desired result.
    \end{proof}

    \begin{remark}\label{F_bound_remark}
        From \cref{F_dot_dot_bound_lemma} and $\log N(b_1b_2\ell) \ll \log{X}$, if $N(m_1m_2) \gg \frac{X^{1+\varepsilon} (1+|\omega|)^{2+\varepsilon} B_1 B_2}{N(d_1 d_2 d^2)}$ or $N(k) \gg \frac{N(\ell^2 m_1 m_2 d_1 d_2 e)}{(1+|\f|)^{-\varepsilon} X^{1-\varepsilon}}$ then 
        \begin{equation*}
            \ddot{F}_{\frac{\lambda^{g_1} d d_1 m_1}{b_1},\frac{\lambda^{g_2} d d_2 m_2}{b_2},\omega }  \Big( \frac{k \sqrt{X}}{\ell^2 m_1 m_2 d_1 d_2 e} \Big) \ll_{F, A, \varepsilon} X^{-A} (1+|\omega|)^{-A} N(m_1 m_2 d_1 d_2 d k)^{-A}.
        \end{equation*}
    \end{remark}

    We now assume $u \neq 0$. Shifting the line of integration in the definition \eqref{eq:W_f} of $W_\omega(y)$, by Stirling's formula and the exponential decay of $e^{w^2}$ we also obtain
    \begin{align*}
        \ddot{F}_{\mathfrak{n}_1, \mathfrak{n}_2,\omega}(u) = \frac{1}{2\pi i} \int_{\varepsilon -i\infty}^{\varepsilon + i\infty} \Big(\frac{\pi^2 N(\mathfrak{n}_1 \mathfrak{n}_2)}{N(\m_{\omega})X}\Big)^{-w} \frac{\Gamma \big(\frac{1+|\omega|}{2}+w \big)^2}{\Gamma \big(\frac{1+|\omega|}{2}\big)^2} e^{w^2} \int_0^\infty r^{2w} F(r^2)  r J_0\Big(\frac{\pi r |u|}{4\sqrt 2}\Big) dr \frac{dw}{w} \\
        = \frac{1}{2\pi i} \int_{\varepsilon -iX^\varepsilon (1+|\omega|)^\varepsilon}^{\varepsilon + iX^\varepsilon (1+|\omega|)^\varepsilon} \Big(\frac{\pi^2 N(\mathfrak{n}_1 \mathfrak{n}_2)}{N(\m_{\omega})X}\Big)^{-w} \frac{\Gamma \big(\frac{1+|\omega|}{2}+w\big)^2}{\Gamma \big(\frac{1+|\omega|}{2}\big)^2} e^{w^2} \mathcal{I}(w, u) \frac{dw}{w} + O_{A, \varepsilon}(X^{-A}(1+|\omega|)^{-A}),
    \end{align*}
    where $\mathcal{I}(w, u)$ is given in \eqref{Iwudef} and the truncation is justified since $\mathcal{I}(w, u) \ll 1$ uniformly for $\Re(w) = \varepsilon$. Using the Mellin--Barnes integral representation \cite[(10.9.22)]{DLMF} gives
    \begin{align}
        \mathcal{I}(w, u) := &\ \int_0^\infty r^{2w} F(r^2)  r J_0\Big(\frac{\pi r |u|}{4\sqrt 2}\Big) dr \label{Iwudef} \\
        =&\ \frac{(-1)^j}{2\pi i} \int_0^\infty \int_{-\varepsilon -i\infty}^{-\varepsilon + i\infty} G_w^{(j)} (r^2) r^{2j+1}\frac{\Gamma(-s)}{\Gamma(j+s+1)} \Big(\frac{\pi r|u|}{8\sqrt 2}\Big)^{2s} ds\: dr \nonumber
    \end{align}
    for $G_w(y) := y^w F(y)$, $u \neq 0$, and $j\in \Z_{\geq 1}$. Observe that $G_w^{(j)}$ is supported in $(1, 2)$ and satisfies the uniform bound
    \begin{equation*}
        G_w^{(j)}(y) \ll_{F, j} (1+|w|)^j.
    \end{equation*}
    Choosing $j$ sufficiently large in terms of $\varepsilon$ (but fixed), Stirling's formula implies that
    \begin{align}\label{F_transform_bulk}
        \ddot{F}_{\mathfrak{n}_1, \mathfrak{n}_2,\omega}(u) = \int_{\varepsilon -iX^\varepsilon (1+|\omega|)^\varepsilon}^{\varepsilon + iX^\varepsilon (1+|\omega|)^\varepsilon}\int_1^{\sqrt{2}} \int_{-\varepsilon -iX^{\varepsilon} (1+|\omega|)^\varepsilon}^{-\varepsilon + iX^{\varepsilon} (1+|\omega|)^\varepsilon} \mathcal{G}_{2, \omega}(w, r, s)  \frac{|u|^{2s}}{N(\mathfrak{n}_1 \mathfrak{n}_2)^{w}} \, ds\, dr \, dw & \nonumber \\
        +\ O_{F, A, \varepsilon}\big( (1+|u|^{-2\varepsilon}) X^{-A} (1+|\omega|)^{-A}\big), & \qquad
    \end{align}
    where in the display above we used
    \begin{align*}
        \mathcal{G}_{2,\omega}(w, r, s) :=&\ \frac{(-1)^j}{(2\pi i)^2} \Big(\frac{\pi^2}{N(\m_{\omega})X}\Big)^{-w} \frac{\Gamma \big(\frac{1+|\omega|}{2}+w\big)^2}{\Gamma \big(\frac{1+|\omega|}{2}\big)^2} \frac{e^{w^2}}{w} G_w^{(j)}(r^2) r^{2j+1} \frac{\Gamma(-s)}{\Gamma(j+s+1)} \Big(\frac{\pi r}{8\sqrt 2}\Big)^{2s} \nonumber \\
        \ll_{F, j} &\ \frac{X^\varepsilon \Gamma\big(\frac{1+|\omega|}{2}+w\big)^2 e^{w^2} (1+|w|)^j \Gamma(-s)}{\Gamma\big(\frac{1+|\omega|}{2}\big)^2 (j+s)(j-1+s) \cdots (1+s) \Gamma(1+s)} \ll_j \frac{X^\varepsilon e^{\frac{w^2}{2}}\Gamma\big(\frac{1+|\omega|}{2}+w\big)^2}{\Gamma\big(\frac{1+|\omega|}{2}\big)^2 (1+ \left|\Im(s)\right|)^j},
    \end{align*}
    for $\Re(w) = \varepsilon$ and $\Re(s) = -\varepsilon$.
    In particular, Stirling's formula gives the uniform bound
    \begin{align}\label{G_func_uniform_bound}
        \mathcal{G}_{2,\omega}(w, r, s) \ll_{F, \varepsilon} X^\varepsilon(1+|\omega|)^{2\ep}.
    \end{align}

    Let $H$ be a smooth non-negative function with compact support in $(1, 3)$ which gives rise to a partition of unity
    \begin{equation}\label{partition_of_unity}
        1 = \sum_{h \in \Z} H\Big(\frac{x}{2^h}\Big) \qquad \text{for every } x>0.
    \end{equation}
    
    We are ready to manipulate the remaining terms in \eqref{eq:pre_split}. Applying M{\"o}bius inversion to the condition $(m_1, m_2) = 1$ and adding partitions of unity to each $m_i$, we conclude that for $k\neq 0$ and $N(\ell) \sim L \gg 1$, the sums over $m_1$ and $m_2$ in \eqref{eq:pre_split} are equal to
    \begin{align}
        &\sumtwo_{\substack{\alpha_1, \alpha_2 \in \Z \\ M_i = 2^{\alpha_i} \gg 1 \\ M_1 M_2 \ll \frac{X^{1+\varepsilon} (1+|\omega|)^{2+\varepsilon} B_1 B_2}{N(d_1 d_2 d^2)}}} \bbone_{N(k) \ll \frac{ M_1 M_2 L^2 N(d_1 d_2 e)}{(1+|\omega|)^{-\varepsilon} X^{1-\varepsilon}}} \sum_{\substack{f \in \Z[i] \\ f \equiv 1 \pmod{\lambda^3} \\ N(f) \ll M_1+M_2 \\ (f, d)=1}} \mu(f) \nonumber \\
        & \times \int_{\varepsilon -iX^\varepsilon (1+|\omega|)^\varepsilon}^{\varepsilon + iX^\varepsilon (1+|\omega|)^\varepsilon}\int_1^{\sqrt{2}} \int_{-\varepsilon -iX^{\varepsilon} (1+|\omega|)^\varepsilon}^{-\varepsilon + iX^{\varepsilon} (1+|\omega|)^\varepsilon} \frac{\mathcal{G}_{2, \omega}(w, r, s) N(b_1b_2)^w}{N(\lambda^{g_1+g_2} d_1 d_2 d^2)^w} \Big(\frac{N(k)X}{N(\ell^2 d_1 d_2 e)}\Big)^s \nonumber \\
        & \times \Bigg(\sum_{\substack{m_1 \in \Z[i] \\ m_1 \equiv c_1 \pmod{\lambda^7} \\ (m_1, d) = 1 \\ b_1 \mid dd_1m_1 \\ f \mid m_1}} \frac{\xi(m_1)\chi_{m_1}(\ell^2 d_1^2 e) \widetilde{g}_4(-k, m_1)}{N(m_1)^{1+w+s}} H\Big(\frac{N(m_1)}{M_1}\Big) \Bigg) \nonumber \\
        &\times \Bigg( \sum_{\substack{m_2 \in \Z[i] \\ m_2 \equiv c_2 \pmod{\lambda^7} \\ (m_2, d) = 1 \\ b_2 \mid dd_2m_2 \\ f \mid m_2}} \frac{\overline{\xi(m_2)\chi_{m_2}(\ell^2 d_2^2 e)\widetilde{g}_4(k, m_2)}}{N(m_2)^{1+w+s}} H\Big(\frac{N(m_2)}{M_2}\Big) \Bigg) ds\: dr\: dw  \label{m_separation_Mellin}
    \end{align}
    plus the contributions from the error term in \eqref{F_transform_bulk} and from the remaining ranges of $M_1, M_2$, and $N(k)$, which by \cref{F_bound_remark} both contribute to \eqref{eq:pre_split} a negligible error term $\ll_{F, \varepsilon} X^{-1000}(1+|\omega|)^{-1000}$.
    The bound on $f$ and the condition $(f, d) = 1$ were added since otherwise the summands are zero.

    Thus
    \begin{equation}\label{localized_sum_def}
        \mathcal{R}_{\omega}(b_1, b_2) = \sumtwo_{\substack{\alpha_1, \alpha_2 \in \Z \\ M_i = 2^{\alpha_i} \gg 1 \\ M_1 M_2 \ll X^{1+\varepsilon} (1+|\omega|)^{2+\varepsilon} B_1 B_2}} \sum_{\substack{\alpha \in \Z \\ L = 2^{\alpha} \gg 1 \\ L \leq Y}} \mathcal{R}^L_{\f,M_1, M_2}(b_1, b_2) + O_{F, \varepsilon}(X^{-1000}),
    \end{equation}
    where $\mathcal{R}^L_{\f,M_1, M_2}(b_1, b_2)$ corresponds to the terms of $\mathcal{R}_{\omega}(b_1, b_2)$ with $N(m_i)$ localized at $M_i$ (using the partition of unity $H$) and $N(\ell)$ localized at $L$ (using a dyadic decomposition $N(\ell) \sim L$). By the supplements to quartic reciprocity in \eqref{eq:ram&units_1} and \eqref{eq:ram&units_2}, if $m, c\equiv 1\pmod{\lambda^3}$ then
    \begin{equation}\label{congruence_with_chars}
        \bbone_{m\equiv c \pmod {\lambda^7}} = \frac{1}{16} \sum_{\eta \mid \lambda^3} \chi_m(\eta) \overline{\chi_c(\eta)},
    \end{equation}
    so we detect the conditions $m_i \equiv c_i\pmod{\lambda^7}$ with linear combinations of cubic characters using \eqref{congruence_with_chars}.
    With $m_1$ and $m_2$ separated in \eqref{m_separation_Mellin} and the congruence condition removed, we interchange the integrals and the sum over $f$ with the sums over $\ell$ and $k$ in \eqref{eq:pre_split}, and then put absolute values around the sums over $m_1$ and $m_2$.  Recalling \eqref{G_func_uniform_bound}, and writing
    \begin{equation}\label{fidef}
        f_i := \Big[\frac{b_i}{(b_i, dd_i)}, f \Big] = \Big[\frac{b_i}{(b_i, d)}, f \Big],
    \end{equation}
    which is squarefree (since so are $b_i$ and $f$), the resulting bound is
    \begin{align}
        &  \mathcal{R}_{\f,M_1, M_2}^L(b_1, b_2) \ll_{F, \varepsilon} X^{1+\varepsilon} (1+|\omega|)^\varepsilon (B_1 B_2)^{1/2}  \sup_{\substack{t_1, t_2 \in \R \\ |t_1|, |t_2| \leq X^{\varepsilon}(1+|\omega|)^\varepsilon }}  \sup_{\eta_1, \eta_2 \mid \lambda^3} \nonumber \\
        & \times  \mathop{\sum\sum\sum}_{\substack{d, d_1, d_2 \in \mathbb{Z}[i] \\ d, d_1, d_2 \equiv 1 \pmod{\lambda^3} \\ (d_1, d_2)=1,\ d_1d_2 \mid d^\infty \\ N(d_1 d_2 d^2) \ll \frac{X^{1+\varepsilon} (1+|\omega|)^{2+\varepsilon} B_1 B_2}{M_1 M_2} }} \frac{1}{N(d)} \sum_{\substack{f \in \mathbb{Z}[i] \\ f \equiv 1 \pmod{\lambda^3} \\ N(f) \ll M_1+M_2 \\ (f, d) = 1}} \mu^2(f) \sum_{\substack{0 \neq k \in \mathbb{Z}[i] \\ N(k) \ll \frac{M_1M_2 L^2 N(d_1 d_2 e)}{(1+|\omega|)^{-\varepsilon} X^{1-\varepsilon}}}} \frac{|\widetilde{g}_4(-k,d_1)\widetilde{g}_4(k,d_2)|}{N(d_1d_2)}  \nonumber \\
        & \times \frac{N((e, k))}{N(e)} \sum_{\substack{\ell \in \mathbb{Z}[i] \\ \ell \equiv 1 \pmod{\lambda^3} \\ N(\ell) \sim L \\ (\ell, d) = 1 }} \frac{\mu^2(\ell)}{N(\ell)^{2}}  \Bigg |  \sum_{\substack{m_1 \in \mathbb{Z}[i] \\ m_1 \equiv 1 \pmod{\lambda^3} \\ (m_1, d)  = 1 \\ f_1 \mid m_1}} \frac{\xi(m_1) \chi_{m_1}(\eta_1 \ell^2 d_1^2 e) \widetilde{g}_4(-k, m_1)}{N(m_1)^{1+it_1}} H\Big(\frac{N(m_1)}{M_1}\Big) \Bigg| \nonumber \\
        & \times \Bigg| \sum_{\substack{m_2 \in \mathbb{Z}[i] \\ m_2 \equiv 1 \pmod{\lambda^3} \\ (m_2, d)  = 1 \\ f_2 \mid m_2}} \frac{\xi(m_2) \chi_{m_2}(\eta_2 \ell^2 d_2^2 e) \widetilde{g}_4(k, m_2)}{N(m_2)^{1+it_2}} H\Big(\frac{N(m_2)}{M_2}\Big) \Bigg| + X^{-500}(1+|\omega|)^{-500}. \qquad \label{R_after_abs_values}
    \end{align}
    The upper bounds on $N(d_1d_2d^2)$ and $N(k)$ follow from the ranges in \eqref{m_separation_Mellin}.

    \subsection{Preliminary pruning}

    Recall the notation \eqref{specnotation}. Observe in \eqref{R_after_abs_values} that
    \begin{equation*}
        N(\ell) \ls 1.
    \end{equation*}
    The bounds
    \begin{equation} \label{kd}
        N(k) \ll \frac{M_1M_2 L^2 N(d_1d_2e)}{(1+|\omega|)^{-\varepsilon} X^{1-\varepsilon}}, \quad \quad N(d_1d_2d^2) \ll \frac{X^{1+\varepsilon} (1+|\omega|)^{2+\varepsilon} B_1 B_2}{M_1M_2},
    \end{equation}
    and the fact that $e = \frac{\mathrm{rad}(d)}{\mathrm{rad}(d_1d_2)}$ together imply that
    \begin{equation}\label{eq:k_bound}
        N(k) \ls 1.
    \end{equation} 
    Furthermore, the same conditions also imply that
    \begin{equation} \label{M1M2bd}
        \frac{X}{N(d_1d_2 d^2)} \ls M_1M_2 \ls X.
    \end{equation}
    Using the coarse bound \eqref{nicebd} for the Gauss sums (which are normalized here) gives
    \begin{equation} \label{Abd}
        \frac{|\widetilde{g}_4(-k,d_1)\widetilde{g}_4(k,d_2)|}{N(d_1 d_2)}  \frac{N((e, k))}{N(e)}\leq \frac{N(k)^2}{N(d_1d_2 e)} \ls \frac{1}{N(d_1d_2 e)}.
    \end{equation}

    Applying these bounds and taking the supremum over $k$ and $\ell$, we see that \eqref{R_after_abs_values} is 
    \begin{align} \label{R_after_abs_values_2}
        &\ls X \sup_{\substack{t_1, t_2 \in \R \\ |t_1|, |t_2| \ls 1}} \sup_{ \eta_1, \eta_2 \mid \lambda^3} \mathop{\sum\sum\sum}_{\substack{d, d_1, d_2 \in \mathbb{Z}[i] \\ d, d_1, d_2 \equiv 1 \pmod{\lambda^3} \\ (d_1, d_2)=1,\ d_1d_2 \mid d^\infty \\ \frac{X}{M_1M_2} \ls N(d_1 d_2 d^2) \ls X}} \frac{1}{N(de d_1 d_2)} \sum_{\substack{f \in \mathbb{Z}[i] \\ f \equiv 1 \pmod{\lambda^3} \\ N(f) \ls X \\ (f, d) = 1}} \mu^2(f)   \nonumber \\
        & \times \sup_{\substack{0 \neq k, \ell \in \mathbb{Z}[i] \\ N(k) \ls 1 \\ N(\ell) \ls 1}} \Bigg| \sum_{\substack{m_1 \in \mathbb{Z}[i] \\ m_1 \equiv 1 \pmod{\lambda^3} \\ (m_1, d)  = 1 \\ f_1 \mid m_1}} \frac{\xi(m_1) \chi_{m_1}(\eta_1 \ell^2 d_1^2 e) \widetilde{g}_4(-k, m_1)}{N(m_1)^{1+it_1}} H\Big(\frac{N(m_1)}{M_1}\Big) \Bigg| \nonumber \\
        & \times \Bigg| \sum_{\substack{m_2 \in \mathbb{Z}[i] \\ m_2 \equiv 1 \pmod{\lambda^3} \\ (m_2, d)  = 1 \\ f_2 \mid m_2}} \frac{\xi(m_2) \chi_{m_2}(\eta_2 \ell^2 d_2^2 e) \widetilde{g}_4(k, m_2)}{N(m_2)^{1+it_2}} H\Big(\frac{N(m_2)}{M_2}\Big) \Bigg| + X^{-500}. \qquad
    \end{align}

    \subsection{Truncation of outer sums}
    We now truncate the sums over $d, d_1, d_2$, and $f$. Before doing this, recall the definition of $f_i$ in \eqref{fidef}. Making a change of variable $m_i \mapsto m_i f_i$ and applying the triangle inequality, for $i \in \{1, 2\}$ we have
    \begin{align}
        \Bigg| \sum_{\substack{m_i \in \mathbb{Z}[i] \\ m_i \equiv 1 \pmod{\lambda^3} \\ (m_i, d)  = 1 \\ f_i \mid m_i}} \frac{\xi(m_i) \chi_{m_i}(\eta_i \ell^2 d_i^2 e) \widetilde{g}_4( \pm k, m_i)}{N(m_i)^{1+it_i}} H\Big(\frac{N(m_i)}{M_i}\Big) \Bigg| \nonumber \\
        \leq \frac{\sqrt{N(k)}}{N(f_i)} \sum_{\substack{m_i \in \mathbb{Z}[i] \\ m_i \equiv 1 \pmod{\lambda^3} }} \frac{1}{N(m_i)} H\Big(\frac{N(m_i f_i)}{M_i}\Big) \ls \frac{1}{N(f)}, \label{eq:m_sum_ trivial_bound}
    \end{align}
    where the last inequality follows from \eqref{fidef} and $N(k) \ls 1$.

    Let $1 \leq D \leq X$ be a parameter which will be chosen later (we will eventually choose $D =X^{\delta}$ for fixed $0<\delta < 1$).
    Using \eqref{eq:m_sum_ trivial_bound}, we deduce that the contribution of the terms with $N(f) > D$ to \eqref{R_after_abs_values_2} is
    \begin{equation} \label{fcont}
        \ls X \mathop{\sum\sum\sum}_{\substack{d, d_1, d_2 \in \mathbb{Z}[i] \\ d, d_1, d_2 \equiv 1 \pmod{\lambda^3} \\ (d_1, d_2)=1,\ d_1d_2 \mid d^\infty \\ N(d_1 d_2 d^2) \ls X}} \frac{1}{N(de d_1 d_2)} \sum_{\substack{f \in \mathbb{Z}[i] \\ N(f) > D}} \frac{1}{N(f)^2} \ls \frac{X}{D} \mathop{\sum\sum\sum}_{\substack{d, d_1, d_2 \in \mathbb{Z}[i] \\ 1\leq  N(d_1 d_2 d) \ls X}} \frac{1}{N(d d_1 d_2)} \ls \frac{X}{D}. 
    \end{equation}
    Similarly, the contribution of the terms with $N(d_1) > D$ to \eqref{R_after_abs_values_2} is
    \begin{equation} \label{d1cont}
        \ls X \mathop{\sum\sum\sum}_{\substack{d, d_1, d_2 \in \mathbb{Z}[i] \\ d, d_1, d_2 \equiv 1 \pmod{\lambda^3} \\ (d_1, d_2)=1,\ d_1d_2 \mid d^\infty \\ N(d_1 d_2 d^2) \ls X \\ N(d_1)>D }} \frac{1}{N(de d_1 d_2)} \sum_{\substack{0 \neq f \in \mathbb{Z}[i]}} \frac{1}{N(f)^2} \ls \frac{X}{D} \mathop{\sum\sum\sum}_{\substack{d, d_1, d_2 \in \mathbb{Z}[i] \\ 1\leq  N(d_1 d_2 d) \ls X \\ d_1 \mid d^\infty}} \frac{1}{N(d d_2)} \ls \frac{X}{D},
    \end{equation}
    since the number of possible $d_1$ that occur in the sum above is $\ls 1$ by Rankin's trick
    \begin{equation}\label{eq:rankins_trick}
        \sum_{\substack{d_1 \in \mathbb{Z}[i] \\ N(d_1) \leq Z \\ d_1 \mid d^\infty}} 1 \leq Z^\varepsilon \sum_{\substack{d_1 \in \mathbb{Z}[i] \\  d_1 \mid d^\infty}} \frac{1}{N(d_1)^\varepsilon} \ll Z^\varepsilon \prod_{\substack{\mathfrak{p} \text{ prime} \\ \mathfrak{p} \mid (d)}} \Big(1 - \frac{1}{N(\mathfrak{p})^{\varepsilon}} \Big)^{-1} \leq Z^\varepsilon \prod_{\substack{\mathfrak{p} \text{ prime} \\ \mathfrak{p} \mid (d)}} \Big(1 - \frac{1}{2^{\varepsilon}}\Big)^{-1} \ll_\varepsilon (Z N(d))^\varepsilon.
    \end{equation}
    The same argument shows that the contribution of the terms $N(d_2)>D$ to \eqref{R_after_abs_values_2} is $\ls \frac{X}{D}$. Finally, the contribution to \eqref{R_after_abs_values_2} of the terms with $N(d) > D$ is
    \begin{equation} \label{dcont1}
        \ls X \mathop{\sum\sum\sum}_{\substack{d, d_1, d_2 \in \mathbb{Z}[i] \\ d, d_1, d_2 \equiv 1 \pmod{\lambda^3} \\ (d_1, d_2)=1,\ d_1d_2 \mid d^\infty \\ N(d_1 d_2 d^2) \ls X \\  N(d) > D}} \frac{1}{N(de d_1 d_2)} \sum_{\substack{0 \neq f \in \mathbb{Z}[i]}} \frac{1}{N(f)^2} \ls  X \sum_{\substack{d \in \mathbb{Z}[i] \\ D < N(d) \ls X}} \frac{1}{N(d \cdot \mathrm{rad}(d))},
    \end{equation}
    where we used the fact that $\mathrm{rad}(d) \mid ed_1 d_2$, and that there are $\ls 1$ choices for $d_1$ and $d_2$, again by Rankin's trick. Furthermore
    \begin{equation} \label{dcont2}
        X \sum_{\substack{d \in \mathbb{Z}[i] \\ D < N(d) \ls X}} \frac{1}{N(d \cdot \mathrm{rad}(d))}  \leq \frac{X}{D} \sum_{\substack{d \in \mathbb{Z}[i] \\ 1 \leq N(d) \ls X}} \frac{1}{N(\mathrm{rad}(d))} \leq \frac{X}{D} \sum_{\substack{d' \in \mathbb{Z}[i] \\ 1 \leq N(d') \ls X}} \frac{1}{N(d')} \sum_{\substack{a \in \mathbb{Z}[i] \\ N(a) \ls X \\ a \mid d'^\infty}} 1 \ls  \frac{X}{D},
    \end{equation}
    where the penultimate inequality follows from a change variables $d^{\prime} = \mathrm{rad}(d)$ and $a = \frac{d}{d^{\prime}}$.

    We now use \eqref{fcont}, \eqref{d1cont}, and \eqref{dcont2} to deduce that \eqref{R_after_abs_values_2} is
    \begin{align} \label{R_after_abs_values_3}
        & \ls \frac{X}{D} +  X \sup_{\substack{t_1, t_2 \in \R \\ |t_1|, |t_2| \ls 1 \\ \eta_1, \eta_2 \mid \lambda^3}} \mathop{\sum\sum\sum}_{\substack{d, d_1, d_2 \in \mathbb{Z}[i] \\ d, d_1, d_2 \equiv 1 \pmod{\lambda^3} \\ (d_1, d_2)=1,\ d_1d_2 \mid d^\infty \\ N(d_1),N(d_2), N(d) \leq D \\ \frac{X}{M_1M_2} \ls N(d_1d_2d^2)}} \frac{1}{N(de d_1 d_2)} \sum_{\substack{f \in \mathbb{Z}[i] \\ f \equiv 1 \pmod{\lambda^3} \\ N(f) \leq D \\ (f, d) = 1}} \mu^2(f)  \nonumber \\
        & \times \sup_{\substack{0 \neq k, \ell \in \mathbb{Z}[i] \\ N(k) \ls 1 \\ N(\ell) \ls 1}}  \Bigg| \sum_{\substack{m_1 \in \mathbb{Z}[i] \\ m_1 \equiv 1 \pmod{\lambda^3} \\ (m_1, d)  = 1 \\ f_1 \mid m_1}} \frac{\xi(m_1) \chi_{m_1}(\eta_1 \ell^2 d_1^2 e) \widetilde{g}_4(-k, m_1)}{N(m_1)^{1+it_1}} H\Big(\frac{N(m_1)}{M_1}\Big) \Bigg| \nonumber \\
        & \times \Bigg| \sum_{\substack{m_2 \in \mathbb{Z}[i] \\ m_2 \equiv 1 \pmod{\lambda^3} \\ (m_2, d)  = 1 \\ f_2 \mid m_2}} \frac{\xi(m_2) \chi_{m_2}(\eta_2 \ell^2 d_2^2 e) \widetilde{g}_4(k, m_2)}{N(m_2)^{1+it_2}} H\Big(\frac{N(m_2)}{M_2}\Big) \Bigg| .
    \end{align}
    In particular, the inequalities in the display above imply that the sum over $d, d_1, d_2$ is only present if
    \begin{equation}\label{eq:M_i_bound}
        \frac{X}{D^4} \ls M_1M_2.
    \end{equation}
    Thus we assume that \eqref{eq:M_i_bound} holds from now on.

    \subsection{Removal of outer sums}

    For $i \in \{1, 2\}$, by \eqref{fidef} have $N(f_i) \leq N(b_i f) \ls D$, and the condition $(m_i, d) = 1$ can be removed by adding $\chi_{m_i}(d^4)$. Writing $r_i := \eta_i^3 \ell^2 d^4 d_i^2 e^3$, so that $N(r_i) \ls D^9$, we observe that the sum over $m_i$ in \eqref{R_after_abs_values_3} is equal to 
    \begin{align*}
        &\sum_{\substack{m_i \in \Z[i] \\ m_i \equiv 1 \pmod{\lambda^3} \\ f_i \mid m_i}} \frac{\xi(m_i) \chi_{m_i}(\eta_i \ell^2 d^4 d_i^2 e) \widetilde{g}_4((-1)^i k, m_i)}{N(m_i)^{1+it_i}} H\Big(\frac{N(m_i)}{M_i}\Big) \\
        \ls\ & \sup_{\substack{f_i \in \Z[i] \\ f_i \equiv 1 \pmod{\lambda^3}\\ \mu^2(f_i) = 1 \\ N(f_i) \ls D}} \sup_{\substack{0 \neq r_i \in \Z[i] \\ N(r_i) \ls D^9}} \sum_{\substack{m_i \in \Z[i] \\ m_i \equiv 1 \pmod{3} \\ f_i \mid m_i}} \frac{\xi(m_i) \overline{\chi_{m_i}(r_i)} \widetilde{g}_4((-1)^i k, m_i)}{N(m_i)^{1+it_i}} H\Big(\frac{N(m_i)}{M_i}\Big).
    \end{align*}
    Applying this to the sums over $m_1$ and $m_2$, the sums over $d, d_1, d_2, f$ in \eqref{R_after_abs_values_3} can then be directly evaluated and give $\ls D$. In conclusion, the final bound for \eqref{R_after_abs_values_3} becomes
    \begin{align}
        & \ls \frac{X}{D} + X D \prod_{i=1}^2 \Bigg(\sup_{\substack{t \in \R \\ |t| \ls 1}} \sup_{\substack{h \in \Z[i] \\ h \equiv 1 \pmod{\lambda^3}\\ \mu^2(h) = 1 \\ N(h) \ls D}} \sup_{\substack{0 \neq k, r \in \Z[i] \\ N(k) \ls 1 \\ N(r) \ls D^9}} \Bigg| \sum_{\substack{m \in \Z[i] \\ m \equiv 1 \pmod{\lambda^3} \\ h \mid m}} \frac{\xi(m)\overline{\chi_{m}(r)} \widetilde{g}_4(k, m)}{N(m)^{1+it}} H\Big(\frac{N(m)}{M_i}\Big) \Bigg| \Bigg). \label{eq:second_mom_prod_prelim}
    \end{align}
    Note that we changed variables $k \mapsto (-1)^i k$ above.

    \subsection{Removal of gcd and divisibility conditions}

    Write (uniquely) $m = w v$ for $w, v \equiv 1\pmod{\lambda^3}$ with $w \mid k^\infty$ and $(v, k)=1$. By \cref{quartic-GS-lemma1}, 
    \begin{equation*}
        \widetilde{g}_4(k, w v) = (-1)^{C(w,v)} \chi_{v}(w^2) \widetilde{g}_4(k, w) \widetilde{g}_4(k, v) =
        \widetilde{g}_4(k, w)  (-1)^{C(w,v)} \overline{\chi_{v}(k w^2)} \widetilde{g}_4(v).
    \end{equation*}
    Thus the sum over $m$ in \eqref{eq:second_mom_prod_prelim} is equal to
    \begin{equation} \label{wivi}
        \sum_{\substack{w \in \mathbb{Z}[i] \\ w \equiv 1 \pmod{\lambda^3} \\ w \mid k^\infty}} \frac{\xi(w) \overline{\chi_{w}(r)}\widetilde{g}_4(k, w)}{N(w)^{1+it}} \sum_{\substack{v \in \mathbb{Z}[i] \\ v \equiv 1  \pmod{\lambda^3} \\ h \mid w v}} (-1)^{C(w,v)} \frac{\xi(v) \overline{\chi_{v}(r k w^2)} \widetilde{g}_4(v)}{N(v)^{1+it}} H\Big(\frac{N(w v)}{M_i}\Big). 
    \end{equation}
    Note that we removed the condition $(v, k)=1$ since it is enforced by $\overline{\chi_{v}(k)}$. Let
    \begin{equation}\label{q_i_def}
        q := \frac{h}{(h, w)},
    \end{equation}
    which is squarefree (since so is $h$), so the divisibility condition for the sum over $v$ becomes $q \mid v$. Setting $v = q n$ for $n\equiv 1\pmod{\lambda^3}$, observe that $v$ is squarefree, otherwise $g_4(v) = 0$ by \eqref{sqrootcancel}, so $(q, n)=1$ and we once again have
    \begin{equation} \label{qini}
        \widetilde{g}_4(q n) = (-1)^{C(q,n)} \chi_{n}(q^2) \widetilde{g}_4(q) \widetilde{g}_4(n).
    \end{equation}
    Apply \eqref{qini} in \eqref{wivi}, where we can drop the condition $(q, n) = 1$ due to the presence of $\chi_{n}(q^2)$, and substitute the result into \eqref{eq:second_mom_prod_prelim}. Then move the sum over $w$ to the outside of the absolute values using the triangle inequality and $|\widetilde{g}_4(q)| \leq 1$, by \eqref{sqrootcancel}.
    We also stratify the $n$ into congruence classes $\rho \pmod{4}$ to deduce that \eqref{eq:second_mom_prod_prelim} is
    \begin{align}
        & \ls \frac{X}{D} + X D \prod_{i=1}^2 \Bigg(\sup_{\substack{t \in \R \\ |t| \ls 1}} \sup_{\substack{\rho \in \Z[i] \\ \rho \equiv 1 \pmod{\lambda^3} }} \sup_{\substack{h \in \Z[i] \\ h \equiv 1 \pmod{\lambda^3}\\ \mu^2(h) = 1 \\ N(h) \ls D}} \sup_{\substack{0 \neq k, r \in \Z[i] \\ N(k) \ls 1 \\ N(r) \ls D^9}}  \\
        & \times \sum_{\substack{w \in \mathbb{Z}[i] \\ w \equiv 1 \pmod{\lambda^3} \\ w \mid k^\infty}} \frac{|\widetilde{g}_4(k, w)|}{N(w q)} \cdot \Bigg| \sum_{\substack{n \in \mathbb{Z}[i] \\ n \equiv \rho  \pmod{4} }} \frac{\xi(n)  \overline{\chi_{n}(r k w^2 q^2)} \widetilde{g}_4(n)}{N(n)^{1+it}} H\Big(\frac{N(w q n )}{M_i}\Big) \Bigg| \Bigg).  \label{almost_ready_for_poles_eq}
    \end{align}
    Note that a factor of $(-1)^{C(w,qn) + C(q, n)}$ was removed since it is fixed for $n \equiv \rho \pmod{4}$.

    \subsection{Truncation of sum over $w$}
    For $i \in \{1, 2\}$, a trivial estimation yields
    \begin{equation} \label{trivnisum}
        \Bigg | \sum_{\substack{n \in \mathbb{Z}[i] \\ n \equiv \rho  \pmod{4}}} \frac{\xi(n) \overline{\chi_{n}(r k w^2 q^2)} \widetilde{g}_4(n)}{N(n)^{1+it}} H\Big(\frac{N(w q n)}{M_i}\Big) \Bigg| \ll \log{M_i} \ls 1.
    \end{equation}
    Therefore, for any $Z \geq 1$, the contribution of the terms with $N(w) > Z$ to each factor in the product over $i \in \{1, 2\}$ in \eqref{almost_ready_for_poles_eq} is
    \begin{align}
        & \ls \sup_{\substack{0 \neq k\in \Z[i] \\ N(k) \ls 1}}  \sum_{\substack{w \in \mathbb{Z}[i] \\ w \equiv 1 \pmod{\lambda^3} \\ N(w) > Z,\ w \mid k^\infty}} \frac{|\widetilde{g}_4(k, w)|}{N(w)} \leq \sup_{\substack{0 \neq k\in \Z[i] \\ N(k) \ls 1}}  N(k)^{1/2} \sum_{\substack{w \in \mathbb{Z}[i] \\ w \equiv 1 \pmod{\lambda^3} \\ N(w) > Z,\ w \mid k^\infty}} \frac{1}{N(w)} \nonumber \\
        & \ls \sup_{\substack{0 \neq k\in \Z[i] \\ N(k) \ls 1}}  \frac{N(k)^{1/2}}{Z^{1/2}} \prod_{\substack{\mathfrak{p} \text{ prime} \\ \mathfrak{p} \mid (k)}} \Big(1 - \frac{1}{N(\mathfrak{p})^{1/2}} \Big)^{-1} \ll_\varepsilon \sup_{\substack{0 \neq k\in \Z[i] \\ N(k) \ls 1}}  \frac{N(k)^{1/2+\varepsilon}}{Z^{1/2}} \ls \frac{1}{Z^{1/2}},
    \end{align}
    where we used \eqref{nicebd} to bound the Gauss sums (mind the normalization) and a coarse version of Rankin's trick. Applying the bound above for $Z = 1$ and $Z = D^4$, we conclude that the contribution to \eqref{almost_ready_for_poles_eq} of the terms (in either factor of the product over $i$) with $N(w)>D^4$ is
    \begin{align}
        & \ls XD \cdot \frac{1}{(D^4)^{1/2}} = \frac{X}{D}.\label{wtrunc1}
    \end{align}
    Thus we may add the condition $N(w) \leq D$ to \eqref{almost_ready_for_poles_eq} without any other changes.

    \subsection{Removal of sum over $w$}
    Since $N(h) \ls D$, we deduce from \eqref{q_i_def} that $N(q) \leq N(h) \ls D$. Together with the condition $N(w) \leq D^4$, this implies that
    $N(r kw^2 q^2) \ls D^{19}$ and $N(w q) \ls D^5$. Therefore \eqref{almost_ready_for_poles_eq} is 
    \begin{align} \label{almost_ready_for_poles_eq_2}
        &\ls \frac{X}{D} + X D \prod_{i=1}^2 \Bigg(\sup_{\substack{t \in \R \\ |t| \ls 1}} \sup_{\substack{\rho \in \Z[i] \\ \rho \equiv 1 \pmod{\lambda^3} }} \sup_{\substack{0 \neq u \in \Z[i] \\ N(u) \ls D^5}} \sup_{\substack{0 \neq k, s \in \Z[i] \\ N(k) \ls 1 \\ N(s) \ls D^{19}}} \nonumber \\
        &\times \sum_{\substack{w \in \mathbb{Z}[i] \\ w \equiv 1 \pmod{\lambda^3} \\ N(w) \leq D^4,\ w \mid k^\infty}} \frac{1}{N(w)} \cdot \Bigg| \sum_{\substack{n \in \mathbb{Z}[i] \\ n \equiv \rho  \pmod{4} }} \frac{\xi(n)  \overline{\chi_{n}(s)} \widetilde{g}_4(n)}{N(n)^{1+it}} H\Big(\frac{N(u n )}{M_i}\Big) \Bigg| \Bigg), \qquad
    \end{align}
    where we once again used $|\widetilde{g}_4(k, w)| \leq N(k)^{1/2} \ls 1$. The sum over $w$ in the display above can now be freely evaluated and is $\ls 1$. 

    \subsection{Sums of quartic Gauss sums}

    For each $0 \neq s \in \mathbb{Z}[i]$ with $N(s) \ls D^{19}$, we can uniquely factorize it as $s = \eta \lambda^{4j} \alpha \beta^2 \gamma^3 \delta^4 \epsilon^4$ with $\eta \mid \lambda^3$, $j \in \Z_{\geq 0}$, $\alpha, \beta, \gamma, \delta, \epsilon \equiv 1 \pmod{\lambda^3}$, $\mu^2(\alpha \beta \gamma \delta)=1$, and $\epsilon \mid (\alpha \beta \gamma \delta)^{\infty}$. Note that $\chi_{n}(s) = \chi_{n}(\eta \alpha \beta^2 \gamma^3 \delta^4)$ in \eqref{almost_ready_for_poles_eq_2}, and \cref{truncated_Gauss_sums_lemma} gives
    \begin{align}\label{decomposition_Gauss_sums_eq}
        \sum_{\substack{n \in \mathbb{Z}[i] \\ n \equiv \rho  \pmod{4}}} \frac{\xi(n) \overline{\chi_{n}(\eta \alpha  \beta^2 \gamma^3 \delta^4)} \widetilde{g}_4(n)}{N(n)^{1+it}} H\Big(\frac{N(u n)}{M_i}\Big) = \mathcal{P}_i + \mathcal{I}_i,
    \end{align}
    where 
    \begin{align*}
        & \mathcal{P}_i := \bbone_{\omega=0} \cdot \bbone_{\gamma=1} \cdot \widetilde{H}(-\tfrac{1}{4}-it) \Big( \frac{M_i}{N(u)} \Big)^{-\frac{1}{4}-it} \cdot \Delta^{*}_{\alpha} (\alpha \eta, \tfrac{5}{4})^{-1} \Delta_{\beta \delta}(\tfrac{5}{4})^{-1} N(\beta)^{-1/4} \nonumber \\
        & \times \mathop{\sum \sum}_{\substack{e, f \in \Z[i] \\ e, f \equiv 1\pmod{\lambda^3} \\ e \mid \beta,\ f \mid \delta}} (-1)^{C(ef\rho,\beta)} \chi_{f}(-1) \frac{\mu(ef)}{N(ef)} \overline{\widetilde{g}_4 \Big(\frac{\alpha \beta^2 \eta}{e^2}, e \Big)}  \widetilde{g}_4 \Big( \frac{\alpha \beta^2 \eta}{e^2}, f \Big) \overline{\widetilde{g}_4 \Big(\alpha \eta,\frac{\beta f}{e} \Big)} \tau_4(\alpha \eta; \beta \rho) \nonumber
    \end{align*}
    and
    \begin{align*}
        \mathcal{I}_i \ll_{A, \varepsilon} (1+|\omega|)^{\varepsilon}X^{\varepsilon} \Big(\frac{M_i}{N(u)}\Big)^{-\frac{1}{2}} \sup_{\substack{ 0 \neq z \in \mathbb{Z}[i] \\ N(z) \ls D^{19} }} \sup_{\substack{y \in \mathbb{R}}} \sup_{\substack{\kappa \in \Z[i] \\ \kappa \equiv 1 \pmod{\lambda^3}}} \frac{\big|\psi\big(z, 1+\varepsilon + i(t + y),\xi; \kappa\big)\big|}{(1+|y|)^A}.
    \end{align*}
    Thus by \cref{psi_convexity_lemma} and \cref{thetaconvexity} we have
    \begin{equation} \label{supest}
        \sup_{\substack{t \in \R \\ |t| \ls 1}} \sup_{\substack{\rho \in \Z[i] \\ \rho \equiv 1 \pmod{\lambda^3} }} \sup_{\substack{0 \neq u \in \Z[i] \\ N(u) \ls D^5}} \sup_{\substack{0 \neq s \in \Z[i] \\ N(s) \ls D^{19}}}
        (\mathcal{P}_i+\mathcal{I}_i) \ls \Big( \Big( \frac{D^5}{M_i}\Big)^{1/4} D^{19/8} +\Big(\frac{D^5}{M_i} \Big)^{1/2} D^{19/4} \Big).
    \end{equation}

    \subsection{Endgame for second moment error term}

    Using \eqref{decomposition_Gauss_sums_eq} and \eqref{supest} in \eqref{almost_ready_for_poles_eq_2} for $i\in \{1,2\}$, we conclude that it is
    \begin{align} \label{secondest}
        \ls \frac{X}{D} + XD \prod_{i=1}^2 \Big( \frac{D^{29/8}}{M_i^{1/4}} + \frac{D^{29/4}}{M_i^{1/2}} \Big) \ls \frac{X}{D}+ \frac{X D^{31/2}}{(M_1M_2)^{1/4}} \ls \frac{X}{D} + X^{3/4} D^{33/2},
    \end{align}
    where we used restriction \eqref{eq:M_i_bound} to obtain the last inequality. Since the display above is a bound for \eqref{R_after_abs_values}, there is a constant $C\geq 1$ such that
    \begin{equation} \label{RM1M2est}
        \mathcal{R}^L_{\f,M_1,M_2}(b_1,b_2) \ll_{F, \varepsilon} X^{\varepsilon} \cdot \big(B Y(1+|\omega|)\big)^C \cdot \Big( \frac{X}{D} + X^{3/4} D^{33/2} \Big)
    \end{equation}
    for all $1 \leq L \leq Y$ and $M_1, M_2 \gg 1$ such that $M_1 M_2 \ll X^{1+\varepsilon} (1+|\omega|)^{2+\varepsilon} B_1 B_2$. Summing over all $L, M_1, M_2$ as in \eqref{localized_sum_def} gives the same bound for $\mathcal{R}_\omega(b_1, b_2)$. Choose for instance $D=X^{1/70}$ to obtain
    \begin{equation} \label{RM1M2est2}
        \mathcal{R}_{\omega}(b_1,b_2) \ll_{F, \varepsilon} \big(B Y (1+|\omega|)\big)^C \cdot X^{1-\frac{1}{70}+\varepsilon},
    \end{equation}
    which proves \eqref{Rbound}, as desired.
\end{proof}


\section{First moment asymptotics}\label{sec:first_moment}

\begin{proof}[Proof of \cref{prop:first_moment}]
    Recall the definitions of $B_{\f,U}(q)$ and $\widetilde{B}_{\f,U}(q)$ in \eqref{BUdef} and \eqref{BUtildedef}. Since we denote $V_{\omega}(y) := V_{\frac{1}{2},\omega}(y)$, \cref{decaylem} gives the bound
    \begin{equation}\label{eq:V_f_bound}
        V_{\omega}(y)\ll_{A} \Big(1+\frac{y}{1+|\omega|}\Big)^{-A}
    \end{equation}
    uniformly in $\omega$. Note from \eqref{SMdef} that
    \begin{align*}
        &\mathcal{S}_M \Big( \nu_{q,\omega}(\mathfrak{b}) \Big[B_{\f,U}(q) + \ep(\f)\xi(q)\widetilde{g}_4(q) \cdot \widetilde{B}_{\f,U^{-1}}(q) \Big]; F \Big) \\
        & = \sum_{\substack{ q \in \mathbb{Z}[i] \\ q \equiv 1 \pmod{\lambda^7}}}  M_Y(q)  \nu_{q,\omega}(\mathfrak{b}) \Big[B_{\f,U}(q) + \ep(\f)\xi(q)\widetilde{g}_4(q) \cdot \widetilde{B}_{\f,U^{-1}}(q) \Big] F \Big( \frac{N(q)}{X} \Big).
    \end{align*}
    The expression above is equal to 
    \begin{equation} \label{SMopen_first_mom}
        \sum_{\substack{0 \neq \mathfrak{n} \unlhd \Z[i]}}
        N(\mathfrak{n})^{-1/2}  \Big[ \mathcal{S}_M \big( \nu_{q,\omega}(\mathfrak{b} \mathfrak{n}); F_{\mathfrak{n},\f,U} \big) + \mathcal{S}_M \big( \ep(\f)\xi(q)\widetilde{g}_4(q) \overline{\nu_{q,\omega}(\mathfrak{b}^{-1} \mathfrak{n})}; F_{\mathfrak{n},\f,U^{-1}} \big) \Big],
    \end{equation}
    where we used the fact that $V_{\omega}$ is real-valued, and denoted
    \begin{equation}\label{eq:F_first_mom}
        F_{\mathfrak{n},\f,U}(t):= F(t) V_{\omega} \Big( \frac{N(\mathfrak{n})}{2U\sqrt{N(\mathfrak m_{\omega})X t}} \Big).
    \end{equation}
    Opening up the two terms of \eqref{SMopen_first_mom} and using \eqref{MYRYdef}, we have
    \begin{align*}
        \mathcal{S}_M \big( \nu_{q,\omega}(\mathfrak{b} \mathfrak{n}); F_{\mathfrak{n},\omega, U} \big) = \xi(\mathfrak b \mathfrak n ) \sum_{\substack{\ell \in \Z[i] \\ \ell \equiv 1 \pmod{\lambda^3} \\ N(\ell) \leq Y}} \mu(\ell) \sum_{\substack{ m \in \Z[i]  \\ \ell^2 m \equiv 1 \pmod{\lambda^7}}}  \chi_{\ell^2m}(\mathfrak{b} \mathfrak{n}) F_{\mathfrak{n},\f,U} \Big( \frac{N(\ell^2 m)}{X} \Big)
    \end{align*}
    and
    \begin{align*}
        & \mathcal{S}_M \big( \ep(\f)\xi(q)\widetilde{g}_4(q) \overline{\nu_{q,\omega}(\mathfrak{b}^{-1} \mathfrak{n})}; F_{\mathfrak{n},\f,U^{-1}} \big) \\
        &= \ep(\f)\xi(\mathfrak b)\overline{\xi(\mathfrak n)}\sum_{\substack{ m \in \Z[i]  \\ m \equiv 1 \pmod{\lambda^7}}}  \xi(m)\widetilde{g}_4(m) \overline{\chi_m(\mathfrak{b}^3 \mathfrak{n})} F_{\mathfrak{n},\omega,U^{-1}} \Big( \frac{N(m)}{X} \Big),
    \end{align*}
    where we used the fact \eqref{sqrootcancel} that $\widetilde{g}_4(\ell^2 m) = 0$ for $\ell \equiv 1 \pmod{\lambda^3}$, unless $\ell = 1$.

    Write $\mathfrak{b} = b \mathbb{Z}[i]$ and $\mathfrak{n} =\lambda^{g} n \mathbb{Z}[i]$ for $g \in \mathbb{Z}_{\geq 0}$ and $b, n \equiv 1 \pmod{\lambda^3}$, where $\mu^2(b)=1$ since we are assuming $\mathfrak{b}$ is squarefree. By \eqref{eq:ram&units_1} and quartic reciprocity, if $\ell^2 m \equiv 1 \pmod{\lambda^7}$ then
    \begin{equation*}
        \chi_{\ell^2 m}(\mathfrak{b} \mathfrak{n}) = \chi_{\ell^2 m}(b \lambda^{g} n)=\chi_{\ell^2 m}(b n)=\chi_{b n}(\ell^2 m) = \quadrat{\ell}{b n}\chi_{b n}(m),
    \end{equation*}
    and similarly for $m \equiv 1 \pmod{\lambda^7}$ we get
    \begin{equation*}
        \overline{\chi_{m}(\mathfrak{b}^3 \mathfrak{n})} = \overline{\chi_{m}(b^3 \lambda^{g} n)} = \overline{\chi_{m}(b^3 n)}.
    \end{equation*}
    Thus
    \begin{equation}\label{first_mom_sum}
        \begin{split}
            &  \mathcal{S}_M \big( \nu_{q,\omega}(\mathfrak{b} \mathfrak{n}); F_{\mathfrak{n},\omega, U} \big)\\
            & = \xi(b\lambda^g n)\sum_{\substack{\ell \in \Z[i] \\ \ell \equiv 1 \pmod {\lambda^3} \\ N(\ell) \leq Y}} \mu(\ell)\quadrat{\ell}{bn}\sum_{\substack{ m \in \Z[i] \\  \ell^2 m \equiv 1 \pmod{\lambda^7}}} \chi_{b n}(m) F_{\lambda^{g} n,\f,U} \Big( \frac{N(\ell^2 m)}{X}  \Big)
        \end{split}
    \end{equation}
    and
    \begin{equation}\label{first_mom_dual}
        \begin{split}
            & \mathcal{S}_M \big( \ep(\f)\xi(q)\widetilde{g}_4(q) \overline{\nu_{q,\omega}(\mathfrak{b}^{-1} \mathfrak{n})}; F_{\mathfrak{n},\f,U^{-1}} \big)\\
            &  = \ep(\f)\xi(b) \overline{\xi(\lambda^g n)} \sum_{\substack{ m \in \Z[i] \\ m \equiv 1 \pmod{\lambda^7}}} \xi(m)\widetilde{g}_4(m) \overline{\chi_m(b^3 n)} F_{\lambda^g n,\omega,U^{-1}} \Big( \frac{N(m)}{X} \Big).
        \end{split}
    \end{equation}
    
    We apply Poisson summation (\cref{lem:pois}) to \eqref{first_mom_sum} and find that the sum over $m$ is equal to
    \begin{equation}\label{eq:first_after_Poisson}
        \frac{\pi X}{64 N(bn\ell^2)}\sum_{k \in \Z[i]} \ddot{\chi}_{bn}(k)\ec\Big({-\frac{k\ell^2(bn)^3}{2\lambda^7}} \Big)\ddot{F}_{\lambda^gn,\f,U}\Big(\frac{k\sqrt X}{bn\ell^2}\Big),
    \end{equation}
    where we use the fact that $\overline{\ell^2}\equiv \ell^2 \pmod{\lambda^7}$ and have
    \begin{equation}\label{eq:first_chi_dot_dot}
        \ddot{\chi}_{bn}(k):=  \sum_{a \pmod {bn}} \chi_{bn}(2\lambda^7 a)\ec\Big({-\frac{ka}{bn}}\Big) = \chi_{bn}(-i\lambda)g_{4}(-k,bn).
    \end{equation}
    Denote the terms corresponding to the first $\mathcal{S}_M$ in \eqref{SMopen_first_mom} by $\mathcal{S}_{\f,U}(b)$, and those corresponding to the second $\mathcal{S}_M$ by $\widetilde{\mathcal{S}}_{\f,U^{-1}}(b)$. Applying \eqref{eq:first_after_Poisson} and \eqref{eq:first_chi_dot_dot} in \eqref{first_mom_sum}, we obtain
    \begin{align}
        \mathcal{S}_{\f,U}(b) &= \frac{\pi X \xi(b)}{64 N(b)} \sum_{g=0}^\infty \frac{\xi(\lambda^g)}{2^{g/2}} \sum_{\substack{c \pmod{\lambda^7} \\ c \equiv 1 \pmod{\lambda^3}}}  \chi_{bc}(-i\lambda) \sum_{k\in \Z[i]} \sum_{\substack{\ell \in \Z[i] \\ \ell \equiv 1 \pmod{\lambda^3} \\ N(\ell) \leq Y}} \frac{\mu(\ell)}{N(\ell)^2}\quadrat{\ell}{b} \check{e}\Big({ -\frac{k \ell^2 b^3 c^3}{2\lambda^7}} \Big) \nonumber \\
        & \quad \times \sum_{\substack{n \in \Z[i] \\ n\equiv c \pmod{\lambda^7}}} \frac{\xi(n)g_4(-k, bn)}{N(n)^{3/2}} \quadrat{\ell}{n} \ddot{F}_{\lambda^g n,\f,U} \Big( \frac{k \sqrt{X}}{bn \ell^2} \Big). \label{S_first_mom_exp}
    \end{align}
    Let $\mathcal{M}_{\f,U}(b)$ denote the terms corresponding to $k=0$ in \eqref{S_first_mom_exp}, which turn out to be the main terms, and let $\mathcal{S}'_{\omega, U}(b)$ denote the terms corresponding to $k \neq 0$. Finally, denote $\mathcal{R}_{\f,U}(b) := \mathcal{S}'_{\f,U}(b) + \widetilde{\mathcal{S}}_{\f,U^{-1}}(b)$.
    
    \subsection{The main term \texorpdfstring{$\mathcal{M}_{\f,U}(b)$}{}}
    
    From the definition, $\mathcal{M}_{\f,U}(b)$ is equal to
    \begin{equation*}
        \frac{\pi X}{64 N(b)} \sum_{g =0}^{\infty}\frac{\xi(\lambda^g)}{2^{g/2}} \sum_{\substack{\ell \in \Z[i] \\ \ell \equiv 1 \pmod{\lambda^3} \\ N(\ell) \leq Y}} \frac{\mu(\ell)}{N(\ell^2)}  \sum_{\substack{n\in \Z[i]\\n \equiv 1 \pmod{\lambda^3}}} \frac{\chi_{bn}(-i\lambda\ell^2)\xi(bn)g_4(0, bn)}{N(n)^{3/2}}\ddot{F}_{\lambda^g n,\f,U} (0).
    \end{equation*}
    Observe that $g_4(0,bn) =0$ unless $bn$ is a fourth power. Since $b\equiv 1 \pmod{\lambda^3}$ is squarefree, this forces $n = b^3m^4$ for $m\equiv 1\pmod{\lambda^3}$. Thus $g_4(0,bn)= g(0,b^4m^4) = \varphi(b^4m^4)= \varphi(bm)N(bm)^3$. We uniquely decompose $m$ as $m = b'c$ for $b', c \equiv 1 \pmod{\lambda^3}$ with $b'\mid b^{\infty}$ and $(b,c)=1$. Then $\varphi(bm)=\varphi(bb')\varphi(c) =N(b')\varphi(b) \varphi(c)$, so 
    \begin{equation}\label{eq:first_moment_main}
        \begin{split}
            & \mathcal{M}_{\omega,U}(b) = \frac{\pi X}{64} \sum_{g =0}^{\infty}\frac{\xi(\lambda^g)}{2^{g/2}} \sum_{\substack{m\in \mathbb{Z}[i]\\ m \equiv 1 \pmod{\lambda^3}}}\frac{\xi(b^4m^4)\varphi(bm)}{N(b)^{5/2}N(m)^3}\ddot{F}_{\lambda^gb^3m^4,\f,U}(0) \sum_{\substack{\ell \in \Z[i]\\ \ell \equiv 1 \pmod{\lambda^3} \\ N(\ell) \leq Y}}
            \frac{\mu(\ell)\mathbf{1}_{bm}(\ell)}{N(\ell^2)}\\
            & = \frac{\pi X \varphi(b) \xi(b^4)}{64 N(b)^{5/2}} \sum_{g =0}^{\infty}\frac{\xi(\lambda^g)}{2^{g/2}} \sumtwo_{\substack{b',c\in \Z[i]\\b',c\equiv 1 \pmod{\lambda^3}\\ b' \mid b^{\infty},\ (b,c)=1}}\frac{\varphi(c)\xi(c^4b'^4)\ddot{F}_{\lambda^gb^3b'^4c^4,\f,U}(0)}{N(b')^{2}N(c)^{3}} \sum_{\substack{\ell \in \Z[i]\\ \ell \equiv 1 \pmod{\lambda^3} \\ N(\ell) \leq Y}}  \frac{\mu(\ell)\mathbf{1}_{bc}(\ell)}{N(\ell^2)}.
        \end{split}
    \end{equation}

    Recall the definition of $F_{\mathfrak n,\f,U}(t)$ in \eqref{eq:F_first_mom}, and that $F$ is supported on $(1,2)$ and satisfies $0 \leq F(t) \leq 1$ for all $t\in\R$. Since $J_0(0) = 1$, by \eqref{eq:ddotV} and \eqref{eq:V_s_f} we have 
    \begin{align} \label{eq:F_n_f(0)}
        \ddot{F}_{\mathfrak{n},\f,U}(0) &=  \int_{1}^{\sqrt 2}rF(r^2)V_{\omega} \Big(\frac{N(\mathfrak n)}{2U r\sqrt{N(\m_{\omega})X}}\Big) dr \nonumber \\
        &= \frac{1}{2\pi i}\int_{2-i\infty}^{2+i\infty} \Big(\frac{\pi N(\mathfrak n)}{U\sqrt{N(\m_{\omega})X}}\Big)^{-w}
        \frac{\Gamma(\tfrac{1+|\omega|}{2}+w)}{\Gamma(\frac{1+|\omega|}{2})} e^{w^2} \check F \Big(\frac{w}{2}\Big)\frac{dw}{2w},
    \end{align}
    where $\check F(w) := \int_0^{\infty} t^w F(t) dt = 2 \int_1^{\sqrt 2}r^{2w+1}F(r^2)dr$. From \eqref{eq:V_f_bound}, we have the coarse bound
    \begin{equation}\label{F_dot_dot_zero_bound_first_mom}
        \ddot F_{\mathfrak n,\f,U}(0) \ll_{A} \Big(1+\frac{N(\mathfrak n)}{(1+|\omega|)U\sqrt X }\Big)^{-A}.
    \end{equation}
    The sum over $\ell$ in \eqref{eq:first_moment_main} is equal to 
    \begin{equation}
        \zeta_{\lambda}^{-1}(2)\prod_{\substack{\pi \ \text{prime}\\ \pi \equiv 1 \pmod{\lambda^3}\\ \pi \mid bc}}\Big(1-\frac{1}{N(\pi)^2}\Big)^{-1} +O\Big(\frac{1}{Y}\Big).
    \end{equation}
    The contribution of the error term $O(1/Y)$ to \eqref{eq:first_moment_main} is
    \begin{equation}
        \ll_{A} \frac{X}{YN(b)^{3/2}}\sum_{\substack{m \in \Z[i]\\ m\equiv 1 \pmod{\lambda^3}}}\frac{1}{N(m)^2}\ll \frac{X}{YN(b)^{3/2}} \ll \frac{X}{Y B^{3/2}}.
    \end{equation}

    Using \eqref{eq:F_n_f(0)}, we see that 
    \begin{align}
        M_{\f,U}(b) = \frac{\pi X}{128 \cdot \zeta_{\lambda}(2)}\frac{1}{2\pi i}\int_{2-i\infty}^{2+i\infty}\Big(\frac{\pi}{U\sqrt{N(\m_{\omega}) X}}\Big)^{-w} \frac{\Gamma(\tfrac{1+|\f|}{2}+w)}{\Gamma(\frac{1+|\f|}{2})} e^{w^2} \check F \Big(\frac{w}{2}\Big) & \mathcal{G}_{b,\omega}(w)\frac{dw}{w} \nonumber \\
        +\ O\Big(\frac{X}{Y B^{3/2}}\Big), & \label{eq:M_f_int_mom_1}
    \end{align}
    where
    \begin{align}
        &\mathcal G_{b,\omega}(w) := \frac{\varphi(b)\xi(b^4)}{N(b)^{5/2+3w}} \bigg(\sum_{g \in \Z_{\ge 0}}\frac{\xi(\lambda^g)}{2^{g(1/2+w)}}\bigg) \bigg(\prod_{\substack{\pi \ \text{prime}\\ \pi \equiv 1 \pmod{\lambda^3}\\ \pi  \mid b}}\Big(1-\frac{1}{N(\pi)^2}\Big)^{-1}\bigg) \nonumber  \\
        & \times \bigg(\sum_{\substack{b'\in \Z[i] \\ b'\equiv 1 \pmod{\lambda^3}\\b'\mid b^{\infty}}}\frac{\xi(b'^4)}{N(b')^{2+4w}}\bigg) \sum_{\substack{c\in \Z[i]\\ c\equiv 1 \pmod{\lambda^3} \\ (b,c)=1}} \frac{\xi(c^4)\varphi(c)}{N(c)^{3+4w}} \prod_{\substack{\pi \ \text{prime}\\ \pi \equiv 1\pmod{\lambda^3}\\ \pi  \mid c}}\Big(1-\frac{1}{N(\pi)^2}\Big)^{-1}. \qquad \label{eq:G_b_initial}
    \end{align} 
    The sum over $c$ in \eqref{eq:G_b_initial} is
    \begin{equation}
        \sum_{\substack{c\in \Z[i]\\ c\equiv 1 \pmod{\lambda^3} \\ (b,c)=1}}\frac{\xi(c^4)\varphi(c)}{N(c)^{3+4w}}
        \prod_{\substack{\pi \ \text{prime}\\ \pi \equiv 1\pmod{\lambda^3}\\ \pi \mid c}}\Big(1-\frac{1}{N(\pi)^2}\Big)^{-1} =\mathcal{K}_{\omega}(w) \prod_{\substack{\pi \ \text{prime}\\ \pi \equiv 1 \pmod{\lambda^3}\\ \pi \mid b}}K_{\omega}(\pi,w)^{-1}  ,
    \end{equation}
    where
    \begin{align}
        K_{\omega}(\pi,w):= & \ 1 + \Big(1-\frac{1}{N(\pi)^2}\Big)^{-1} \sum_{k=1}^{\infty} \frac{\xi(\pi^{4k}) \varphi(\pi^k)}{N(\pi)^{k(2+4w)} N(\pi)^k}= 1 + \frac{\big(1+\frac{1}{N(\pi)}\big)^{-1}}{N(\pi)^{2+4w} \overline{\xi(\pi^4)}-1} \qquad \nonumber \\
        = & \ \Big(1-\frac{\xi(\pi^4)}{N(\pi)^{2+4w}}\Big)^{-1} \Big(1-\frac{\xi(\pi^4)}{N(\pi)^{2+4w}(N(\pi)+1)}\Big) \label{eq:K(pi,w)}
    \end{align}
    and
    \begin{equation}\label{eq:K}
        \mathcal K_{\omega}(w) := \prod_{\substack{\pi \ \text{prime}\\ \pi \equiv 1 \pmod{\lambda^3}}}K_{\omega}(\pi,w) =:\zeta_{\lambda}(2+4w,\xi) \cdot \mathcal Q _{\omega}(w).
    \end{equation}
    Computing the sums over $g$ and $b'$ in \eqref{eq:G_b_initial} gives
    \begin{equation}\label{eq:G_b_f}
        \mathcal{G}_{b,\omega}(w) = \Big(\frac{2^{1/2+w}}{2^{1/2+w}-\xi(\lambda)}\Big)\cdot \zeta_{\lambda}(2+4w,\xi) \cdot \mathcal Q _{\omega}(w) \cdot \mathcal{H}_{b,\omega}(w),
    \end{equation}
    where from now on we assume $\Re(w)\ge -\frac{1}{2}+\ep$ and have 
\begin{align}
        \mathcal{H}_{b,\omega}(w) & := \frac{\xi(b^4)}{N(b)^{3/2+3w}}\frac{\varphi(b)}{N(b)} \prod_{\substack{\pi \ \text{prime}\\ \pi \equiv 1 \pmod{\lambda^3}\\\pi \mid b}} \Big(1-\frac{1}{N(\pi)^2}\Big)^{-1}
        \Big(1 - \frac{\xi(\pi^4)}{N(\pi)^{2+4w}}\Big)^{-1} K_{\omega}(\pi,w)^{-1} \nonumber \\
        & = \frac{\xi(b^4)}{N(b)^{3/2+3w}} \prod_{\substack{\pi \ \text{prime}\\ \pi \equiv 1 \pmod{\lambda^3}\\ \pi  \mid b}}  \Big(1+\frac{1}{N(\pi)}\Big)^{-1} \Big(1-\frac{\xi(\pi^4)}{N(\pi)^{2+4w}(N(\pi)+1)}\Big)^{-1} \label{eq:H_b_f(w)}\\
        & \ll_{\ep} N(b)^{-3/2-3\Re(w) + \varepsilon}. \nonumber
    \end{align}
    It follows from \eqref{eq:K(pi,w)} and \eqref{eq:K} that $\mathcal{Q}_{\omega}(w)$ is holomorphic and uniformly bounded for $\Re(w)\ge -\frac{1}{2} +\ep$. Note that $\mathcal G_{b,\omega}(w)$ is holomorphic for $\Re(w) \ge - \frac{1}{2}+\ep$ if $\omega \neq 0$. If $\omega = 0$, note that $G_{b,0}(0)$ has a simple pole at $w  = -\frac{1}{4}$, since neither  $Q_{0}(- \frac{1}{4})$ nor $\mathcal{H}_{b,0}(-\frac{1}{4})$ are zero, and no other poles for $\Re(w)\ge -\frac{1}{2} +\ep$.

    Therefore we may shift the contour in \eqref{eq:M_f_int_mom_1} to $\Re w = - \frac{1}{4} +\ep$, picking up only a simple pole at $w =0$. Since $\zeta_{\lambda}(2+4w,\xi)$ and $\mathcal Q_{\omega}(w)$ are both uniformly bounded in that vertical line, the remaining integral is
    \begin{align*}
        & \ll \frac{X^{7/8+\ep/2}}{U^{1/4-\ep}N(b)^{3/4}}\int_{-1/4+\ep -i\infty}^{-1/4+\ep +i\infty} \Bigg| \frac{\Gamma(\tfrac{1+|\f|}{2}+w)}{\Gamma(\frac{1+|\f|}{2})}\Bigg| \cdot |e^{w^2}| \cdot |w|^{100}|dw| \ll \frac{X^{7/8+\ep}}{(1+|\omega|)^{1/4} U^{1/4}B^{3/4}}
    \end{align*}
    by Stirling's formula and \eqref{firstmomentassump}.

    Denote $q :=N(\pi)$. From \eqref{eq:M_f_int_mom_1} we have
    \begin{equation} \label{maintermintermed}
        M_{\omega,U}(b) =  \frac{\pi X}{128\cdot\zeta_{\lambda}(2)} \check F(0) \mathcal{G}_{b,\omega}(0)+ O\Big(\frac{X}{Y B^{3/2}}\Big)+O_{\ep}\Big(\frac{X^{7/8+\ep}}{(1+|\omega|)^{1/4} U^{1/4} B^{3/4}}\Big).
    \end{equation}
    Observe by \eqref{eq:K(pi,w)}, \eqref{eq:K}, \eqref{eq:G_b_f}, and \eqref{eq:H_b_f(w)} that
    \begin{equation*}
        \mathcal{G}_{b,\omega}(0) = \frac{\sqrt{2}}{\sqrt 2 - \xi(\lambda)} \zeta_{\lambda}(2,\xi)\mathcal Q_{\omega}(0)\frac{\xi(b^4) r_{\omega}(b)}{N(b)^{3/2}},
    \end{equation*}
    where
    \begin{equation*}
        \zeta_{\lambda}(2,\xi)\mathcal Q_{\omega}(0) =  \mathcal{K}_\omega(0) = \prod_{\substack{\pi \ \text{prime}\\ \pi \equiv 1 \pmod{\lambda^3} \\ q:=N(\pi)}} \Big(1+\frac{q}{(q+1)(q^{2} \overline{\xi(\pi^4)}-1)}\Big)
    \end{equation*}
    and $r$ is the multiplicative function given, for $k\geq 1$, by
    \begin{equation*}
        r(\pi^k):= \frac{q^{3}}{q^3+q^2-\xi(\pi^4)} = 1+ O\Big(\frac{1}{q}\Big),
    \end{equation*}
    for $\pi$ prime and $k\in \Z_{>0}$.
    
     \subsection{The standard sum $\mathcal{S}^{\prime}_{\omega,U}(b)$}
    
    Recall the notation $\ls$ in \eqref{specnotation} and the definition of $F_{\mathfrak{n},\f,U}(t)$ in \eqref{eq:F_first_mom}. We observe for future reference that the proof of \cref{F_dot_dot_bound_lemma} applies directly to give, for any $\mathfrak{n} \unlhd \Z[i]$, $u\in \C$, and $A \in \Z_{\geq 0}$, that
    \begin{equation}\label{f_dot_dot_first_mom_bound}
        \ddot{F}_{\mathfrak{n},\f,U}(u) \ll_{F, A} \Big(1 + |u| + \frac{N(\mathfrak{n})}{(1+|\omega|)U\sqrt{X}}\Big)^{-A}.
    \end{equation}
    
    We now use the hypothesis \eqref{firstmomentassump} to quickly verify that $|\mathcal{S}^{\prime}_{\omega,U}(b)|$ is negligible. Applying the triangle inequality to \eqref{S_first_mom_exp} gives
    \begin{equation}\label{S_prime_bound}
        \mathcal{S}^{\prime}_{\f,U}(b) \ll \frac{X}{B} \sum_{g=0}^\infty \frac{1}{2^{g/2}} \sum_{\substack{ k\in \Z[i] \\ k \neq 0}} \sum_{\substack{\ell \in \Z[i] \\ \ell \equiv 1 \pmod{\lambda^3} \\ N(\ell) \leq Y}} \frac{\mu^2(\ell)}{N(\ell)^2} \sum_{\substack{n \in \Z[i] \\ n \equiv 1 \pmod{\lambda^3}}} \frac{|g_4(-k, bn)|}{N(n)^{3/2}} \Big| \ddot{F}_{\lambda^g n,\f,U} \Big( \frac{k \sqrt{X}}{bn \ell^2} \Big) \Big|.
    \end{equation}
    By \eqref{f_dot_dot_first_mom_bound} we can restrict to $N(n) \leq (1+|\omega|)^{1+\varepsilon} U X^{1/2+\varepsilon}$ up to a negligible error term $O_{F, A, \varepsilon}(X^{-A}(1+|\omega|)^{-A})$. Up to the same error term, we may also restrict to
    \begin{equation*}
        N(k) \leq \frac{N(bn\ell^2)}{X^{1-\varepsilon} (1+|\omega|)^{-\varepsilon}} \ll \frac{BY^2 N(n)}{X^{1-\varepsilon} (1+|\omega|)^{-\varepsilon}} \leq \frac{(1+|\omega|)^{1+2\varepsilon} BY^2 U}{X^{1/2-2\varepsilon}}
    \end{equation*}
    From the hypothesis \eqref{firstmomentassump}, we deduce that
    \begin{equation*}
        N(k) \ll \frac{(1+|\omega|)^{1+2\varepsilon} BY^2 U}{X^{1/2-2\varepsilon}} \ll \frac{X^{(1/2-\nu)(1+2\varepsilon)}}{X^{1/2-2\varepsilon}} = X^{-\nu + \varepsilon(3-2\nu)}.
    \end{equation*}
    For $\varepsilon > 0$ sufficiently small in terms of $\nu$, the sum over $k \neq 0$ becomes empty. Therefore
    \begin{equation} \label{SprimeU}
        \mathcal{S}^{\prime}_{\omega,U} (b) \ll_{F, \nu, A} X^{-A}.
    \end{equation}

    \subsection{The dual sum $\widetilde{\mathcal{S}}_{\omega,U^{-1}}(b)$}

    Recalling \eqref{first_mom_dual}, and detecting the congruence on $m$ using multiplicative characters as in \eqref{congruence_with_chars}, we have
    \begin{align}
        \widetilde{\mathcal{S}}_{\omega,U^{-1}}(b) = \frac{\ep(\f) \xi(b)}{16} \sum_{\eta \mid \lambda^3} \sum_{g=0}^{\infty} \frac{\overline{\xi(\lambda^g)}}{2^{g/2}} \sum_{\substack{n \in \Z[i] \\ n \equiv 1 \pmod{\lambda^3}}} \frac{\overline{\xi(n)}}{N(n)^{1/2}} \nonumber \\
        \times \sum_{\substack{ m \in \Z[i]  \\ m \equiv 1 \pmod{\lambda^3}}}  \xi(m) \overline{\chi_m(\eta b^3 n )}
        \widetilde{g}_4(m)  F_{\lambda^g n,\omega,U^{-1}} \Big( \frac{N(m)}{X} \Big).
    \end{align}
    Separate $m$ into congruence classes $v \pmod{4}$, and then partition the range of $n$ into dyadic intervals $N(n) \sim R$. The triangle inequality gives
    \begin{align}
        & \widetilde{\mathcal{S}}_{\omega,U^{-1}}(b)  \ll \sup_{\substack{v \in \Z[i] \\ v \equiv 1 \pmod{\lambda^3}}} \sup_{\eta \mid \lambda^3}  \sum_{\substack{i \in \Z \\ R = 2^i \gg 1}} \sum_{g=0}^\infty \frac{1}{2^{g/2}}  \nonumber \\
        & \times \Bigg|\sum_{\substack{n \in \Z[i] \\ n \equiv 1 \pmod{\lambda^3} \\ N(n) \sim R}} \frac{\overline{\xi(n)}}{N(n)^{1/2}} \sum_{\substack{ m \in \Z[i]  \\ m \equiv v \pmod{4}}}
        \xi(m) \overline{\chi_m(\eta b^3 n)} \widetilde{g}_4(m)  F_{\lambda^g n,\omega,U^{-1}} \Big( \frac{N(m)}{X} \Big) \Bigg|. \qquad \label{dual_sum_dyadic_decomp}
    \end{align}

    Recalling the definition of $F_{\lambda^g n,\omega,U}$ in \eqref{eq:F_first_mom} and the bound \eqref{eq:V_f_bound}, we can restrict to
    \begin{equation*}
        R \ll \frac{(1+|\omega|)^{1+\varepsilon} X^{1/2+\varepsilon}}{U}
    \end{equation*}
    in \eqref{dual_sum_dyadic_decomp}, up to error $O_{A, \varepsilon}((1+|\omega|)^{-A} X^{-A})$. For the remaining dyadic ranges of $R$, open $V_{\omega}$ using \eqref{eq:V_s_f} and then shift the line of integration to $\Re(w) = \varepsilon$. We deduce that the sum over $n$ in \eqref{dual_sum_dyadic_decomp} is equal to 
    \begin{align}
        &  \frac{1}{2\pi i} \int_{\varepsilon - i \infty}^{\varepsilon + i \infty}  \sum_{\substack{n \in \Z[i] \\ n \equiv 1 \pmod{\lambda^3} \\ N(n) \sim R}} \frac{\overline{\xi(n)}}{2^{gw} N(n)^{1/2+w}} \sum_{\substack{ m \in \Z[i]  \\ m \equiv v \pmod{4}}}  \frac{ \xi(m) \overline{\chi_m(\eta b^3 n)} \widetilde{g}_4(m)
        }{N(m)^{-w/2}} F \Big( \frac{N(m)}{X} \Big) \nonumber \\
        &  \qquad \qquad \qquad \qquad \qquad \qquad \qquad \qquad \qquad  \times \Big(\frac{\pi U}{\sqrt{N(\mathfrak{m}_{\omega})}} \Big)^{-w}
        e^{w^2} \frac{\Gamma\big(\tfrac{1+|\omega|}{2}+w\big)}{\Gamma\big(\tfrac{1+|\omega|}{2} \big)}  \frac{dw}{w} \nonumber \\
        & \ll_{A, \varepsilon} \int_{-(1+|\omega|)^\varepsilon X^\varepsilon}^{(1+|\omega|)^\varepsilon X^\varepsilon}  \sum_{\substack{n \in \Z[i] \\ n \equiv 1 \pmod{\lambda^3} \\ N(n) \sim R}} \frac{1}{N(n)^{1/2+\varepsilon }} \cdot \Bigg| \sum_{\substack{ m \in \Z[i]  \\ m \equiv v \pmod{4}}}  \frac{\xi(m) \overline{\chi_m(\eta b^3 n)}  g_4(m)}{N(m)^{1/2-\varepsilon/2 - it/2}} F \Big( \frac{N(m)}{X} \Big) \Bigg| \, |dt| \nonumber \\
        & \qquad \qquad \qquad \qquad \qquad \qquad \qquad \qquad \qquad \times (1+|\omega|)^\varepsilon  + (1+|\omega|)^{-A} X^{-A}      \label{dual_ready_for_eval},
    \end{align}
    where the second step is justified by Stirling's formula and the decay of $e^{w^2}$.

    Now uniquely factorize $b^3 n = \alpha \beta^2 \gamma^3 \delta^4 \epsilon^4$ with $\alpha, \beta, \gamma, \delta, \epsilon \equiv 1 \pmod{\lambda^3}$, $\mu^2(\alpha \beta \gamma \delta)=1$, and $\epsilon \mid (\alpha \beta \gamma \delta)^\infty$. Then $\chi_{m}(\eta b^3 n) = \chi_m(\eta \alpha \beta^2 \gamma^3 \delta^4)$, and \cref{truncated_Gauss_sums_lemma} gives
    \begin{equation}
        \sum_{\substack{ m \in \Z[i]  \\ m \equiv v \pmod{4}}}  \frac{\xi(m) \overline{\chi_m(\eta \alpha \beta^2 \gamma^3 \delta^4)} g_4(m) }{N(m)^{1/2-\varepsilon/2 - it/2}} F \Big( \frac{N(m)}{X} \Big) \nonumber \\
        = \widetilde{\mathcal{P}}_\omega + \widetilde{\mathcal{I}}_\omega,  \nonumber
    \end{equation}
    where
    \begin{align} \label{Pequal}
        & \widetilde{\mathcal{P}}_\omega := \bbone_{\omega=0} \cdot \bbone_{\gamma=1} \cdot \widetilde{F}(\tfrac{3}{4}+\tfrac{\varepsilon+it}{2}) X^{\frac{3}{4}+\frac{\varepsilon+it}{2}} \cdot \Delta^{*}_{\alpha} (\alpha \eta, \tfrac{5}{4})^{-1} \Delta_{\beta \delta}(\tfrac{5}{4})^{-1} N(\beta)^{-1/4} \nonumber \\
        & \times \mathop{\sum \sum}_{\substack{e, f \in \Z[i] \\ e, f \equiv 1\pmod{\lambda^3} \\ e \mid \beta,\ f \mid \delta}} (-1)^{C(efv, \beta)} \chi_{f}(-1) \frac{\mu(ef)}{N(ef)}
        \overline{\widetilde{g}_4 \Big(\frac{\alpha \beta^2 \eta}{e^2}, e \Big)}  \widetilde{g}_4 \Big( \frac{\alpha \beta^2 \eta}{e^2}, f \Big) \overline{\widetilde{g}_4 \Big(\alpha \eta,\frac{\beta f}{e} \Big)} \tau_4(\alpha \eta; \beta v)
    \end{align}
    and
    \begin{align} \label{Ibound}
        \widetilde{\mathcal{I}}_\omega & \ll_{F, A, \varepsilon} (BR)^{3\varepsilon} X^{1/2 + \varepsilon} \sup_{\substack{\rho \in \Z[i] \\ \rho \equiv 1 \pmod{\lambda^3}}} \sup_{\substack{0 \neq r \in \Z[i] \\ N(r) \ll B^3R}} \int_{-\infty}^\infty \frac{\big|\psi\big(r, 1+\frac{\varepsilon}{2} + i(y - \frac{t}{2}),\xi;\rho\big)\big|}{(1+|y|)^A} \, dy. \qquad
    \end{align}
    Applying the triangle inequality to $\widetilde{\mathcal{P}}_\omega$ in \eqref{Pequal} gives
    \begin{equation} \label{Pineq}
        \widetilde{\mathcal{P}}_\omega \ll_{F, \varepsilon} \bbone_{\omega=0} \cdot (BR)^\varepsilon X^{3/4+\varepsilon} \sup_{\substack{\rho \in \Z[i] \\ \rho \equiv 1 \pmod{\lambda^3} }} \sup_{\substack{0 \neq r \in \Z[i] \\ N(r) \ll B^3R}} |\tau_4(r;\rho)| \ls \bbone_{\omega=0} \cdot X^{3/4} R^{1/8},
    \end{equation}
    where we used \cref{thetaconvexity} in the last step. Similarly, by \cref{psi_convexity_lemma} we have
    \begin{equation*}
        \widetilde{\mathcal{I}}_\omega \ll_{F, \varepsilon} (BR)^{4\varepsilon} X^{1/2+2\varepsilon} \sup_{\substack{0 \neq r \in \Z[i] \\ N(r) \ll B^3R}} N(r)^{1/4+\varepsilon} (1+|\omega|)^{3/2+\varepsilon} \ls X^{1/2} R^{1/4}.
    \end{equation*}
    
    Inserting the bounds above for $\widetilde{\mathcal{P}}_\omega$ and $\widetilde{\mathcal{I}}_\omega$ into \eqref{dual_ready_for_eval} and then \eqref{dual_sum_dyadic_decomp}, we conclude that
    \begin{equation} \label{SfUestimate}
        \widetilde{\mathcal{S}}_{\omega,U^{-1}}(b) \ls \sup_{1 \ls R \ls \frac{X^{1/2}}{U}} \Big(\bbone_{\omega=0} \cdot X^{3/4} R^{5/8} +  X^{1/2} R^{3/4} \Big) \ls \bbone_{\omega=0} \cdot \frac{X^{17/16}}{U^{5/8}}
        + \frac{X^{7/8}}{U^{3/4}}.
    \end{equation}

    \subsection{Conclusion}
    
    Combining \eqref{maintermintermed}, \eqref{SprimeU}, and \eqref{SfUestimate}, we obtain \cref{prop:first_moment}.
\end{proof}


\section{Choosing the mollifier} \label{mollifier_section}

With asymptotics for twisted moments in place, in this section we use them to make an explicit choice of mollifier and prove a positive proportion of non-vanishing in \cref{mainthm}.

Recall the mollifier given in \eqref{mollifier}.
We suppose throughout that its coefficients $\lambda_{\omega}(\mathfrak{b})$ are supported on squarefree $0 \neq \mathfrak{b} \unlhd \Z[i]$ coprime with $2$ such that $N(\mathfrak{b}) \leq M$.
We also assume that $\lambda_{\omega}(\mathfrak{b}) \ll_\varepsilon N(\mathfrak{b})^{-1+\varepsilon}$ for every $\varepsilon > 0$. This gives the trivial bound 
\begin{equation}\label{eq:M_w(q)_bound}
    |\mathcal M_{\omega}(q)| \ll_\ep \sum_{\substack{0 \ne \mathfrak b \unlhd \Z[i]\\ N(\mathfrak b)\le M}} N(\mathfrak b)^{-1/2+\ep} \ll_{\ep} M^{1/2+\ep}.
\end{equation}
Fix a Schwartz function $F$ supported in $(1, 2)$ such that $0 \leq F(t) \leq 1$ for every $t \in \R$. For simplicity, also assume that $\omega \in \Z$ is fixed. Thus we allow implied constants in this section to depend on $F$ and $\omega$ without further comment.

Let $H \geq 10^{10}$ be an absolute constant such that both bounds \eqref{Rfirstbd} and \eqref{Rbound} for the error terms of the first and second moments hold. We set 
\begin{equation} \label{eq:MYU_def}
    M:=X^{\vartheta^3}, \qquad Y:= X^{\vartheta^2}, \qquad \text{and} \qquad U:= X^{1/2-\vartheta},
\end{equation}
where $0 < \vartheta < 10^{-10}$ will later be chosen sufficiently small (depending only on $H$). For $X>1$ sufficiently large (depending only on $\omega$ and $\vartheta$) and (say) $\nu = \frac{\vartheta}{2}$, these parameters clearly satisfy the constraint
\begin{equation} \label{eq:MYUassume2}
    (1+|\omega|) MY^2U \le X^{1/2-\nu}.
\end{equation}
From now on we allow implied constants to depend on $\vartheta$.

It is convenient to make the change of variables
\begin{equation}\label{eq:kappa_def}
    \kappa_{\omega}(\mathfrak{l}) := \sum_{\substack{0 \neq \mathfrak{a} \unlhd \Z[i]}} \lambda_{\omega}(\mathfrak{l a}) h_{\omega}(\mathfrak{a}),
\end{equation}
where $h_{\omega}$ is the multiplicative function defined in \eqref{eq:h_xi}. Note that
\begin{equation}\label{lambda_chi_mobius_inv}
    \lambda_{\omega}(\mathfrak{l}) = \sum_{\substack{0 \neq \mathfrak{a} \unlhd \Z[i]}} \mu(\mathfrak{a}) h_{\omega}(\mathfrak{a}) \kappa_{\omega}(\mathfrak{l a}).
\end{equation}
From now on we assume that $\kappa_{\omega}(\mathfrak{l})$ is supported on squarefree $\mathfrak l$ coprime to $2$ with $N(\mathfrak{l}) \leq M$, which by \eqref{lambda_chi_mobius_inv} implies the same property for $\lambda_\omega(\mathfrak{l})$.
The sequence $\kappa_{\omega}$ will be chosen to satisfy 
\begin{align}\label{kap_bound}
    |\kappa_{\omega}(\mathfrak{d})| \ll \frac{1}{N(\mathfrak{d}) \log{M}} \prod_{\substack{\mathfrak{p} \text{ prime} \\ \mathfrak{p} \mid \mathfrak{d}}} \Big(1 + O\Big(\frac{1}{N(\mathfrak{p})}\Big) \Big).
\end{align}
By \eqref{lambda_chi_mobius_inv} and $h_{\omega}(\mathfrak{p}) = 1 + O(N(\mathfrak{p})^{-1})$, this ensures the desired bound
\begin{equation} \label{lambdabd}
    \lambda_{\omega}(\mathfrak{d}) \ll_{\varepsilon} N(\mathfrak{d})^{-1+\varepsilon}.
\end{equation}


\subsection{The first mollified moment}

Express $L(1/2, \nu_{q,\omega})$ for $q \in \mathcal{F}'_4$ using \cref{lem:afe_2}, and then apply \eqref{nonnegproperty} and \eqref{SMexpand} to the resulting sums in the first mollified moment. For the parameters $M$, $Y$, $U$ in \eqref{eq:MYU_def}, which satisfy \eqref{eq:MYUassume2}, we observe that \cref{prop:first_moment} and \cref{SRestimate1} hold. Combining them with the trivial bound \eqref{eq:M_w(q)_bound} for the mollifier yields 
\begin{align}
    \mathcal{S}(L(1/2, \nu_{q,\omega}) \mathcal{M}_{\omega}(q) ; F) &= C_{\omega} X \check{F}(0) Q_{1,\omega}(M) \nonumber \\
    & + O\Big(\sum_{0 \neq \mathfrak{b} \unlhd \Z[i]} |\lambda_{\omega}(\mathfrak b)| \sqrt{N(\mathfrak{b})} \cdot|\mathcal{R}_{\omega,U}(\mathfrak{b})|\Big) \label{SLintermed} \\
    & +O_\varepsilon \Big( X^{\ep} M^{1/2} \Big [  \frac{X^{7/8}}{U^{1/4}}+ U^{1/2} X^{1/2}+ \frac{X^{3/4}}{Y^{1/4}}+ \frac{X}{Y^{5/6}} \Big] \Big)\label{first_moment_Qv1},
\end{align}
where the constant $C_{\omega}$ is given in \eqref{eq:C_f},
\begin{align}
    Q_{1,\omega}(M) &:= \sum_{\substack{0 \neq \mathfrak{b} \unlhd \Z[i] \\ N(\mathfrak{b}) \leq M}} \frac{\lambda_{\omega}(\mathfrak{b}) \xi(\mathfrak{b}^4) r_{\omega}(\mathfrak{b})}{N(\mathfrak{b})} = \sum_{\substack{0 \neq \mathfrak{b} \unlhd \Z[i] \\ N(\mathfrak{b}) \leq M}} \frac{\xi(\mathfrak{b}^4) r_{\omega}(\mathfrak{b})}{N(\mathfrak{b})} \sum_{\substack{0 \neq \mathfrak{a} \unlhd \Z[i]}} \mu(\mathfrak{a}) h_{\omega}(\mathfrak{a}) \kappa_{\omega}(\mathfrak{a b}) \qquad \nonumber \\
    &\ =  \sum_{\substack{0 \neq \mathfrak{d} \unlhd \Z[i] \\ N(\mathfrak{d}) \leq M}} \kappa_{\omega}(\mathfrak{d}) G_{\omega}(\mathfrak{d}), \label{eq:Q1}
\end{align}
$r_{\omega}$ is given in \eqref{eq:def r},
and $G_{\omega}$ is the Dirichlet convolution of $\frac{\xi^4 r_{\omega}}{N}$ with $\mu  h_{\omega}$. Hence $G_\omega$ is multiplicative and on primes $\mathfrak p$ we have
\begin{equation}
    G_{\omega}(\mathfrak{p}) := \frac{\xi(\mathfrak{p}^4) r_{\omega}(\mathfrak{p})}{N(\mathfrak{p})} - h_{\omega}(\mathfrak{p}) = -1 + O\Big(\frac{1}{N(\mathfrak{p})}\Big), \label{G_primes}
\end{equation}
which follows from the asymptotics for $r_{\omega}$ and $h_{\omega}$ in \eqref{eq:def r} and \eqref{eq:h_xi}, respectively. 

For the sum over $\mathfrak{b}$ in \eqref{SLintermed}, we apply \eqref{Rfirstbd}, \eqref{eq:M_w(q)_bound}, and \eqref{eq:MYU_def} to deduce that it is
\begin{equation*}
    \ll_\varepsilon X^{\ep} M^{1/2} (MY)^H \Big(\frac{X^{17/16}}{U^{5/8}}
    + \frac{X^{7/8}}{U^{3/4}}\Big) \ll X^{3/4 + \vartheta H} = o(X),
\end{equation*}
provided that $0< \vartheta < 10^{-10}$ is sufficiently small in terms of $H \geq 10^{10}$, which we assume from now on.
By \eqref{eq:MYU_def}, the other error terms in \eqref{first_moment_Qv1} are also 
\begin{equation*}
    \ll_\varepsilon X^{\ep} M^{1/2} \Big(  \frac{X^{7/8}}{U^{1/4}}+ U^{1/2} X^{1/2}+ \frac{X^{3/4}}{Y^{1/4}}+ \frac{X}{Y^{5/6}} \Big) \ll X^{3/4 + \vartheta + \varepsilon} + X^{1 - \frac{\vartheta^2}{2} +\varepsilon} = o(X).
\end{equation*}
Hence
\begin{equation}\label{first_moment_Q}
    \mathcal{S}(L(1/2, \nu_{q,\omega}) \mathcal{M}_{\omega}(q) ; F)   =   C_{\omega} X \check{F}(0) Q_{1,\omega}(M) +o(X).
\end{equation}


\subsection{The second mollified moment}

Write $\mathfrak{b} = (\mathfrak{b}_1,\mathfrak{b}_2)$ and $\mathfrak{b}_i = \mathfrak{a}_i \mathfrak{b}$ for $i \in \{1, 2\}$. Since $|L(1/2, \nu_{q, \omega})|^2 = 2 A_\omega(q)$ for $q \in \mathcal{F}'_4$ by \cref{lem:afe_2}, we may apply \eqref{nonnegproperty} and \eqref{SMexpand2}, then treat the resulting terms using \cref{SRestimate2} and \cref{prop:second_moment}. We obtain
\begin{align}
    & \mathcal{S}(|L(1/2,\nu_{q,\omega}) \mathcal{M}_{\omega}(q)|^2; F) =  2 \sumtwo_{\substack{0 \neq \mathfrak{b}_1, \mathfrak{b}_2 \unlhd \Z[i]}} \lambda_{\omega}(\mathfrak{b}_1) \overline{\lambda_{\omega}(\mathfrak{b}_2)} \sqrt{N(\mathfrak{b}_1\mathfrak{b}_2)} \label{SL21} \\
    &\times \Big( D_{\omega} \check{F}(0) X \frac{g_{\omega}(\mathfrak{b}) h_{\omega}(\mathfrak{a}_1) \overline{h_{\omega}(\mathfrak{a}_2)}}{\sqrt{N(\mathfrak{a}_1 \mathfrak{a}_2)}} \Big[\log \Big( \frac{(1+|\omega|)^2 X}{N(\mathfrak{a}_1 \mathfrak{a}_2)} \Big) +  \mathcal{O}_{\omega}(\mathfrak{b}_1, \mathfrak{b}_2)\Big] + \mathcal{R}_{\omega}(\mathfrak{b}_1, \mathfrak{b}_2) \Big) \qquad \label{SL22} \\
    & + O_\varepsilon \Big(X^{\ep} M \Big( X^{3/4}+ \frac{X}{Y^{2/3}} \Big) \Big),  \label{SL2intermed}
\end{align}
where $D_{\omega}$, $g_{\omega}$, $h_{\omega}$, and $\mathcal{O}_{\omega}(\mathfrak{b}_1,\mathfrak{b}_2)$ are given respectively in \eqref{eq:D_omega}, \eqref{eq:g_xi_1}, \eqref{eq:h_xi},
and \eqref{eq:O(b1,b2)}. Note that to obtain the error terms in \eqref{SL2intermed} we used \eqref{eq:M_w(q)_bound}.

We again bound the mollifier trivially as in \eqref{eq:M_w(q)_bound} and use \eqref{Rbound} to deduce that the contribution from the terms $\mathcal R_{\omega}(\mathfrak{b}_1,\mathfrak{b}_2)$ in \eqref{SL22} is
\begin{align} \label{Rerrorterm}
    & \ll_\varepsilon M^{1+\varepsilon} \sup_{\substack{0 \neq \mathfrak{b}_1, \mathfrak{b}_2 \unlhd \Z[i] \\ N(\mathfrak{b}_1),N(\mathfrak{b}_2)\leq M\\ (\mathfrak{b}_1\mathfrak{b}_2,2)=1}} \mu^2(\mathfrak{b}_1)\mu^2(\mathfrak{b}_2)\left|\mathcal{R}_{\omega}(\mathfrak{b}_1,\mathfrak{b}_2)\right| \nonumber \\
    & \ll_\varepsilon X^{\ep} M \cdot (MY)^H \cdot X^{1-\frac{1}{70}} \ll X^{1-\frac{1}{70} + H \vartheta + \varepsilon} = o(x)
\end{align}
due to \eqref{eq:MYU_def}, provided that $0 < \vartheta < 10^{-10}$ is sufficiently small in terms of $H \geq 10^{10}$, which we do assume. Similarly, the error term in \eqref{SL2intermed} is $\ll_\varepsilon X^{3/4 + \vartheta + \varepsilon} + X^{1-\vartheta^2/2 + \varepsilon} = o(x)$. Thus
\begin{equation}
    \mathcal{S}(|L(1/2,\chi_q) \mathcal{M}_{\omega}(q)|^2; F) = 2 D_{\omega} \check{F}(0) X Q_{2_,\omega}(M) + o(X), \label{second_moment_Q}
\end{equation}
where
\begin{equation*}
    \begin{split}
        Q_{2,\omega}(M) := \sumtwo_{\substack{0 \neq \mathfrak{b}_1, \mathfrak{b}_2 \unlhd \Z[i] \\ N(\mathfrak{b}_1), N(\mathfrak{b}_1) \leq M}} \lambda_{\omega}(\mathfrak{b}_1) \overline{\lambda_{\omega}(\mathfrak{b}_2)} N(\mathfrak{b}) g_{\omega}(\mathfrak{b}) h_{\omega}(\mathfrak{a}_1) \overline{h_{\omega}(\mathfrak{a}_2)} & \\
        \times \Big(\log \Big( \frac{(1+|\omega|)^2 X}{N(\mathfrak{a}_1 \mathfrak{a}_2)} \Big) +  \mathcal{O}_{\omega}(\mathfrak{b}_1, \mathfrak{b}_2)\Big). &
    \end{split}
\end{equation*}

\subsection{Evaluation of $Q_{2,\omega}(M)$}

In this section we closely follow \cite[p.~85--86]{DDDS24}.
In order to evaluate $Q_{2,\omega}(M)$, recall that $\mathfrak{b} = (\mathfrak{b}_1, \mathfrak{b}_2)$ and $\mathfrak{b}_i = \mathfrak{a}_i \mathfrak{b}$ to rewrite it as
\begin{align*}
    Q_{2,\omega}(M) = \sum_{\substack{0 \ne \mathfrak{b} \unlhd \Z[i]\\ N(\mathfrak{b}) \le M}}g_{\omega}(\mathfrak b) N(\mathfrak b) \sumtwo_{\substack{0 \neq \mathfrak{a}_1,\mathfrak{a}_2 \unlhd \Z[i] \\ N(\mathfrak{a}_1), N(\mathfrak{a}_2)\le \frac{M}{N(\mathfrak{b})} \\ (\mathfrak{a}_1,\mathfrak{a}_2)=1}} \lambda_{\omega}(\mathfrak{a}_1 \mathfrak{b}) \overline{\lambda_{\omega}(\mathfrak{a}_2 \mathfrak{b})} h_{\omega}(\mathfrak{a}_1) \overline{h_{\omega}(\mathfrak{a}_2)} & \\
    \times \Big(\log\pfrac{(1+|\omega|)^2 X}{N(\mathfrak{a}_1 \mathfrak{a}_2)} + \mathcal{O}_{\omega}(\mathfrak{a}_1\mathfrak{b},\mathfrak{a}_2\mathfrak{b})\Big). &
\end{align*}
Apply M\"{o}bius inversion to remove the condition  $(\mathfrak{a}_1,\mathfrak{a}_2)=1$, writing $\mathfrak{a}_i = \mathfrak{c}_i \mathfrak{c}$, to get
\begin{align}
    Q_{2,\omega}(M) & = \sum_{\substack{ 0\neq \mathfrak{b} \unlhd \Z[i] \\ N(\mathfrak{b}) \leq M}} g_{\omega}(\mathfrak{b})N(\mathfrak{b}) \sum_{\substack{ 0\neq \mathfrak{c} \unlhd \Z[i] \\ N(\mathfrak{c}) \leq \frac{M}{N(\mathfrak{b})} }} \mu(\mathfrak{c}) |h_{\omega}(\mathfrak{c})|^2 \sumtwo_{\substack{0 \neq \mathfrak{c}_1, \mathfrak{c}_2 \unlhd \Z[i] \\ N(\mathfrak{c}_1), N(\mathfrak{c}_2)\leq \frac{M}{N(\mathfrak{bc})}}} \lambda_{\omega}(\mathfrak{c}_1\mathfrak{c}\mathfrak{b}) \overline{\lambda_{\omega}(\mathfrak{c}_2\mathfrak{c}\mathfrak{b})} \nonumber \\
    & \qquad \qquad \qquad \qquad \times  h_{\omega}(\mathfrak{c}_1) \overline{h_{\omega}(\mathfrak{c}_2)} \Big(\log\Big(\frac{(1+|\omega|)^2 X}{N(\mathfrak{c}_1\mathfrak{c}_2\mathfrak{c}^2)}\Big) + \mathcal{O}_{\omega}(\mathfrak{c}_1\mathfrak{c}\mathfrak{b}, \mathfrak{c}_2\mathfrak{c}\mathfrak{b}) \Big) \nonumber \\
    & = \sum_{\substack{ 0\neq \mathfrak{d} \unlhd \Z[i] \\ N(\mathfrak{d}) \leq M}} N(\mathfrak{d}) H_{\omega}(\mathfrak{d}) \sumtwo_{\substack{0 \neq \mathfrak{c}_1, \mathfrak{c}_2 \unlhd \Z[i] \\ N(\mathfrak{c}_1),N(\mathfrak{c}_2)\leq \frac{M}{N(\mathfrak{d})}}} \lambda_{\omega}(\mathfrak{c}_1\mathfrak{d}) \overline{\lambda_{\omega}(\mathfrak{c}_2\mathfrak{d})} h_{\omega}(\mathfrak{c}_1) \overline{h_{\omega}(\mathfrak{c}_2)} \nonumber \\
    & \qquad \qquad \qquad \qquad \times \left(\log\pfrac{(1+|\omega|)^2X}{N(\mathfrak{c}_1\mathfrak{c}_2)} + 2 \eta_{\omega}(\mathfrak{d}) + \mathcal{O}_{\omega}(\mathfrak{c}_1\mathfrak{d},\mathfrak{c}_2\mathfrak{d}) \right), \label{Q_2_with_error_terms}
\end{align}
where for squarefree $\mathfrak{d} \unlhd \Z[i]$ we define
\begin{equation} \label{Hd}
    H_{\omega}(\mathfrak{d}) := \sumtwo_{\substack{0 \neq \mathfrak{b}, \mathfrak{c} \unlhd \Z[i] \\ \mathfrak{b}\mathfrak{c} = \mathfrak{d}}} g_{\omega}(\mathfrak{b}) \frac{\mu(\mathfrak{c})|h_{\omega}(\mathfrak{c})|^2}{N(\mathfrak{c})} = \prod_{\substack{\mathfrak{p} \text{ prime} \\ \mathfrak{p} \mid \mathfrak{d}}} \Big(g_{\omega}(\mathfrak{p}) - \frac{|h_{\omega}(\mathfrak{p})|^2}{N(\mathfrak{p})}\Big),
\end{equation}
and
\begin{equation*}
    \eta_{\omega}(\mathfrak{d}) := -\frac{1}{H_{\omega}(\mathfrak{d})} \sumtwo_{\substack{0 \neq \mathfrak{b}, \mathfrak{c} \unlhd \Z[i] \\ \mathfrak{b}\mathfrak{c} = \mathfrak{d}}} g_{\omega}(\mathfrak{b}) \frac{\mu(\mathfrak{c})|h_{\omega}(\mathfrak{c})|^2}{N(\mathfrak{c})} \log N(\mathfrak{c}) = \sum_{\substack{\mathfrak{p} \text{ prime} \\ \mathfrak{p} \mid \mathfrak{d}}} \frac{|h_{\omega}(\mathfrak{p})|^2 \log N(\mathfrak{p})}{N(\mathfrak{p}) H_{\omega}(\mathfrak{p})}
\end{equation*}
by a standard calculation. From the definition of $g_{\omega}$ (which is real and positive) in \eqref{eq:g_xi_1} and the definition of $h_{\omega}$ in
\eqref{eq:h_xi}, it follows by explicit computation that
\begin{equation}\label{eq:H_w_bounds}
    0< H_{\omega}(\mathfrak p)= 1 + O\Big(\frac{1}{N(\mathfrak{p})}\Big).
\end{equation}

Splitting the sums below based on when  $N(\mathfrak{p}) \ll \log{N(\mathfrak{d})}$, we see for $N(\mathfrak{d})$ large that
\begin{equation*} 
    \sum_{\substack{\mathfrak{p} \text{ prime} \\ \mathfrak{p} \mid \mathfrak{d} }} \frac{1}{N(\mathfrak{p})} \leq \log\log\log{N(\mathfrak{d})} + O(1) \quad \text{and} \quad \sum_{\substack{\mathfrak{p} \text{ prime} \\ \mathfrak{p} \mid \mathfrak{d} }} \frac{\log{N(\mathfrak{p})}}{N(\mathfrak{p})} \leq \log\log{N(\mathfrak{d})} + O(1). \qquad
\end{equation*}
Combining this with \eqref{kap_bound} and \eqref{eq:H_w_bounds} gives
\begin{equation*}
    H_{\omega}(\mathfrak{d}) \ll (\log\log{N(\mathfrak{d})})^{O(1)}, \quad  \eta_{\omega}(\mathfrak{d}) \ll \log\log N(\mathfrak{d}), \quad \text{and} \quad \kappa_{\omega}(\mathfrak d)\ll \frac{(\log \log M)^{O(1)}}{N(\mathfrak d)\log M}.
\end{equation*}
Thus recalling \eqref{eq:kappa_def}, the contribution of the term $2\eta_{\omega}(\mathfrak{d})$ to \eqref{Q_2_with_error_terms} is
\begin{align*}
    = 2 \sum_{\substack{ 0\neq \mathfrak{d} \unlhd \Z[i] \\ N(\mathfrak{d}) \leq M}} N(\mathfrak{d}) H_{\omega}(\mathfrak{d}) \eta_{\omega}(\mathfrak{d}) |\kappa_{\omega}(\mathfrak{d})|^2 \ll \sum_{\substack{ 0\neq \mathfrak{d} \unlhd \Z[i] \\ N(\mathfrak{d}) \leq M}} \frac{(\log\log{N(\mathfrak{d})})^{O(1)}}{(\log{M})^2 N(\mathfrak{d})} \ll \frac{(\log\log{X})^{O(1)}}{\log{X}} = o(1).
\end{align*}
The same argument shows that the term $2\log(1+|\omega|)$ contributes $o(1)$ to \eqref{Q_2_with_error_terms}. 

Similarly, by \eqref{eq:O(b1,b2)} and \eqref{eq:kappa_def}, the contribution of the term $\mathcal{O}_{\omega}(\mathfrak{c}_1\mathfrak{d}, \mathfrak{c}_2\mathfrak{d})$ to \eqref{Q_2_with_error_terms} is 
\begin{align*}
    &\sum_{\substack{ 0\neq \mathfrak{d} \unlhd \Z[i] \\ N(\mathfrak{d}) \leq M}} N(\mathfrak{d}) H_{\omega}(\mathfrak{d}) \bigg( C_{\omega, F} |\kappa_{\omega}(\mathfrak{d})|^2 + \sum_{\substack{\mathfrak{p} \text{ prime}}} D_{1, \omega}(\mathfrak{p}) \frac{\log{N(\mathfrak{p})}}{N(\mathfrak{p})} \big( \kappa_{\omega}(\mathfrak{pd}) \overline{\kappa_{\omega}(\mathfrak{d})}  - |\kappa_{\omega}(\mathfrak{pd})|^2 \big)   \\
    & + \sum_{\substack{\mathfrak{p} \text{ prime}}} D_{2, \omega}(\mathfrak{p}) \frac{\log{N(\mathfrak{p})}}{N(\mathfrak{p})} \big( \kappa_{\omega}(\mathfrak{d}) \overline{\kappa_{\omega}(\mathfrak{pd})}  - |\kappa_{\omega}(\mathfrak{pd})|^2 \big) + \sum_{\substack{\mathfrak{p} \text{ prime} \\ \mathfrak{p} \mid \mathfrak{d}}} D_{3, \omega}(\mathfrak{p}) \frac{\log{N(\mathfrak{p})}}{N(\mathfrak{p})} |\kappa_{\omega}(\mathfrak{d})|^2    \\
    & + \sum_{\substack{\mathfrak{p} \text{ prime}\\  \mathfrak{p} \nmid \mathfrak{d}}} D_{3, \omega}(\mathfrak{p}) \frac{\log{N(\mathfrak{p})}}{N(\mathfrak{p})} |\kappa_{\omega}(\mathfrak{pd})|^2 \bigg),
\end{align*}
which we directly bound by
\begin{align}
    & \ll \sum_{\substack{ 0\neq \mathfrak{d} \unlhd \Z[i] \\ N(\mathfrak{d}) \leq M}} \frac{(\log \log{M})^{O(1)}}{N(\mathfrak{d}) (\log{M})^2} \Big( 1 + \sum_{\substack{\mathfrak{p} \text{ prime} \\  N(\mathfrak{p}) \leq M}} \frac{\log{N(\mathfrak{p})}}{N(\mathfrak{p})^2} + \sum_{\substack{\mathfrak{p} \text{ prime} \\  \mathfrak{p} \mid \mathfrak{d}}} \frac{\log{N(\mathfrak{p})}}{N(\mathfrak{p})} \Big) \ll \frac{(\log \log{X})^{O(1)}}{\log{X}} = o(1). \nonumber
\end{align}
It follows that
\begin{equation}\label{eq:Q_2_w_formula}
    \begin{split}
        & Q_{2,\omega}(M) = \sum_{\substack{0 \neq \mathfrak d \unlhd \Z[i]}} N(\mathfrak{d})H_{\omega}(\mathfrak d) \sumtwo_{\substack{0 \neq \mathfrak{b}_1,\mathfrak{b}_2 \unlhd \Z[i]}} \lambda_\omega(\mathfrak{b}_1 \mathfrak{d})\overline{\lambda_\omega(\mathfrak{b}_2 \mathfrak{d})}h_{\omega}(\mathfrak b_1)\overline{h_{\omega}(\mathfrak b_2)}\log\Big(\frac{X}{N(\mathfrak b_1 \mathfrak b_2)}\Big)+o(1).
    \end{split}
\end{equation}


\subsection{The optimal mollifier}

Observe that $Q_{2,\omega}(M)$ is essentially equal to the diagonal quadratic form
\begin{align*}
    \log{X} \sum_{\substack{ 0\neq \mathfrak{d} \unlhd \Z[i] \\ N(\mathfrak{d}) \leq M}} N(\mathfrak{d}) H_{\omega}(\mathfrak{d}) |\kappa_{\omega}(\mathfrak{d})|^2.
\end{align*}
By Cauchy--Schwarz, this quadratic form is minimized, subject to fixing $Q_{1, \omega}(M)$ given by the linear constraint \eqref{eq:Q1}, when $\kappa_{\omega}$ is proportional to $\frac{\overline{G_{\omega}(\mathfrak d)}}{N(\mathfrak d)H_{\omega}(\mathfrak d)}$. Thus, for squarefree $\mathfrak{d} \unlhd \Z[i]$ coprime to $2$ such that $N(\mathfrak{d}) \leq M$ we choose
\begin{equation}\label{eq:kappachoice}
    \kappa_{\omega}(\mathfrak d) = \frac{\overline{C_{\omega}}}{D_{\omega}\log M} \cdot \frac{\overline{G_{\omega}(\mathfrak d)}}{N(\mathfrak d)H_{\omega}(\mathfrak d)}.
\end{equation}
Note that this choice of $\kappa_{\omega}$ satisfies \eqref{kap_bound}, by \eqref{G_primes} and \eqref{eq:H_w_bounds}.


\subsubsection{Endgame for first moment}

By \eqref{first_moment_Q} and \eqref{eq:Q1}, our choice of $\kappa_{\omega}(\mathfrak d)$ gives
\begin{align} \label{firstset}
    \mathcal{S}(L(1/2, \nu_{q,\omega}) \mathcal{M}_{\omega}(q); F) = \check{F}(0) X \frac{|C_{\omega}|^2}{D_{\omega} \log{M}} \sum_{\substack{0 \neq \mathfrak{d} \unlhd \Z[i] \\  N(\mathfrak{d}) \leq M \\ (\mathfrak{d}, 2) = 1}} \mu^2(\mathfrak{d}) \frac{|G_{\omega}(\mathfrak{d})|^2}{N(\mathfrak{d}) H_{\omega}(\mathfrak{d})} + o(X).
\end{align}
A standard argument using Perron's formula,
$\frac{|G_{\omega}(\mathfrak{p})|^2}{H_{\omega}(\mathfrak{p})} = 1 + O\big(\frac{1}{N(\mathfrak{p})}\big)$,
and \eqref{eq:zetalambdares} yields
\begin{align}
    &\sum_{\substack{0 \neq \mathfrak{d} \unlhd \Z[i] \\  N(\mathfrak{d}) \leq M \\ (\mathfrak{d}, 2) = 1}} \frac{\mu^2(\mathfrak{d}) |G_{\omega}(\mathfrak{d})|^2}{N(\mathfrak{d}) H_{\omega}(\mathfrak{d})} = \Res_{s=0} \frac{\zeta_\lambda(1+s) M^s}{s} \prod_{\substack{\mathfrak{p} \text{ prime} \\ (\mathfrak{p}, 2) = 1}} \Big(1 - \frac{1}{N(\mathfrak{p})} \Big) \Big(1 + \frac{|G_{\omega}(\mathfrak{p})|^2}{N(\mathfrak{p}) H_{\omega}(\mathfrak{p})} \Big) + O(1) \nonumber \\
    & = \Big(\frac{\pi}{8}\log{M} + O(1)\Big) \prod_{\substack{\mathfrak{p} \text{ prime} \\ (\mathfrak{p}, 2) = 1}} \Big(1 - \frac{1}{N(\mathfrak{p})} \Big) \Big(1 + \frac{|G_{\omega}(\mathfrak{p})|^2}{N(\mathfrak{p}) H_{\omega}(\mathfrak{p})} \Big) = \frac{\pi}{8} \mathcal{Q}_{\omega} \log{M} +O(1),  \qquad \label{first_mom_asymptotic_comp}
\end{align}
where
\begin{equation*}
    \mathcal{Q}_{\omega} := \prod_{\substack{\mathfrak{p} \text{ prime} \\ (\mathfrak{p}, 2) = 1 \\ q := N(\mathfrak{p})}}\frac{(q^2-1) ( q^5+2 q^4+q^3+(1-2\Re{\xi(\mathfrak{p}^4)})q^2+1)}{q (q^6 + 2 q^5 + q^4 - 2 q^2(q+1) \Re{\xi(\mathfrak{p}^4)} + 1)}
\end{equation*}
is computed from the definitions of $G_{\omega}$ and $H_{\omega}$ in \eqref{G_primes} and \eqref{Hd}, using \eqref{eq:def r}, \eqref{eq:g_xi_1}, and \eqref{eq:h_xi}. From the definitions of $C_{\omega}$ and $D_{\omega}$ in \eqref{eq:C_f} and \eqref{eq:D_omega},
we conclude that the first mollified moment is equal to $\mathcal{C}_{1,\omega} \mathcal{P}_{1,\omega} \check{F}(0) X + o(X)$, where
\begin{align*}
    \mathcal{C}_{1,\omega} := \frac{\pi}{8} \Big|\frac{\pi}{48\sqrt{2} \cdot \zeta_{\Q(i)}(2)\cdot (\sqrt 2 -\xi(\lambda))} \Big|^2 \Big( \frac{\pi^2}{768 \cdot \zeta_{\Q(i)}(2) \cdot |\sqrt{2}-\xi(\lambda)|^2 } \Big)^{-1} = \frac{\pi }{48 \cdot \zeta_{\Q(i)}(2)},
\end{align*}
and the Euler factors completely cancel due to the remarkable identity
\begin{align*}
    \mathcal{P}_{1,\omega} := \mathcal{Q}_{\omega} \cdot \prod_{\substack{\mathfrak{p} \text{ prime} \\ (\mathfrak{p}, 2) = 1 \\ q := N(\mathfrak{p})}}  \frac{\big|1+ \frac{q}{(q+1)(q^2 \overline{\xi(\mathfrak p^4)}-1)}\big|^2}{\big( 1 -\frac{1}{q(q+1)} +  2 \Re \big(\frac{q}{(q+1)(q^2\xi(\mathfrak{p}^4)-1)}\big)\big)} = 1.
\end{align*}
In conclusion,
\begin{equation}
    \mathcal{S}(L(1/2, \nu_{q,\omega}) \mathcal{M}_{\omega}(q); F) = \frac{\pi}{48\cdot \zeta_{\Q(i)}(2)} \check{F}(0) X + o(X). \label{first_mom_asymp}
\end{equation}


\subsubsection{Endgame for second moment}

Note from \eqref{eq:Q_2_w_formula} that
\begin{align*}
    Q_{2,\omega}(M)  = \sum_{\substack{ 0\neq \mathfrak{d} \unlhd \Z[i]}} N(\mathfrak{d}) H_{\omega}(\mathfrak{d}) &\Big( |\kappa_{\omega}(\mathfrak{d})|^2 \log(X) \\
    & - 2 \Re \Big[ \overline{\kappa_{\omega}}(\mathfrak{d}) \sum_{\substack{0 \neq \mathfrak{b} \unlhd \Z[i]}} \lambda_{\omega}(\mathfrak{b} \mathfrak{d}) h_{\omega}(\mathfrak{b}) \log{N(\mathfrak{b})} \Big] \Big) +o(1).
\end{align*}
Using \eqref{eq:kappa_def} followed by \eqref{eq:kappachoice}, the asymptotic estimates for $G_{\omega}$ and $H_{\omega}$ in \eqref{G_primes} and \eqref{eq:H_w_bounds}, and the fact that $h_{\omega}(\mathfrak{p}) = 1 + O(N(\mathfrak{p})^{-1})$, we observe that
\begin{align*}
    \sum_{\substack{0 \neq \mathfrak{b} \unlhd \Z[i]}} \lambda_{\omega}(\mathfrak{b} \mathfrak{d}) h_{\omega}(\mathfrak{b}) \log{N(\mathfrak{b})} &= \sum_{\substack{\mathfrak{p} \text{ prime}}} \log{N(\mathfrak{p})} h_{\omega}(\mathfrak{p}) \kappa_{\omega}(\mathfrak{pd}) \\
    = \kappa_{\omega}(\mathfrak{d})\sum_{\substack{\mathfrak{p} \text{ prime} \\ N(\mathfrak{p}) \leq \frac{M}{N(\mathfrak{d})} \\ \mathfrak{p} \nmid 2 \mathfrak{d}}} \log{N(\mathfrak{p})} h_{\omega}(\mathfrak{p}) \frac{\overline{G_{\omega}(\mathfrak{p})}}{N(\mathfrak{p}) H_{\omega}(\mathfrak{p})} &=  - \kappa_{\omega}(\mathfrak{d}) \sum_{\substack{\mathfrak{p} \text{ prime} \\ N(\mathfrak{p}) \leq \frac{M}{N(\mathfrak{d})} \\ \mathfrak{p} \nmid 2\mathfrak{d}}} \frac{\log{N(\mathfrak{p})}}{N(\mathfrak{p})} \Big(1 + O\Big(\frac{1}{N(\mathfrak{p})}\Big)\Big). 
\end{align*}
A short computation gives
\begin{align*}
    \sum_{\substack{\mathfrak{p} \text{ prime} \\ N(\mathfrak{p}) \leq \frac{M}{N(\mathfrak{d})} \\ \mathfrak{p} \nmid 2 \mathfrak{d}}} \frac{\log{N(\mathfrak{p})}}{N(\mathfrak{p})} \Big(1 + O\Big(\frac{1}{N(\mathfrak{p})}\Big)\Big) & = \sum_{\substack{\mathfrak{p} \text{ prime} \\ N(\mathfrak{p}) \leq \frac{M}{N(\mathfrak{d})}}} \frac{\log{N(\mathfrak{p})}}{N(\mathfrak{p})}  + O\big(\log\log{N(\mathfrak{d})} + 1\big) \\
    & = \log{\Big( \frac{M}{N(\mathfrak{d})} \Big)} + O(\log\log{X}).
\end{align*}
Therefore
\begin{equation} \label{Q2intermed}
    Q_{2,\omega}(M) = \sum_{\substack{ 0\neq \mathfrak{d} \unlhd \Z[i]}} N(\mathfrak{d}) H_{\omega}(\mathfrak{d}) |\kappa_{\omega}(\mathfrak{d})|^2 \Big(\log\Big({\frac{XM^2}{N(\mathfrak{d})^2}}\Big)  + O(\log\log{X}) \Big)+o(1).
\end{equation}
By the bounds \eqref{kap_bound} and \eqref{eq:H_w_bounds}, the term $O(\log\log{X})$ above contributes to $Q_{2, \omega}(M)$ at most $\frac{(\log\log{X})^{O(1)}}{\log{X}} = o(1)$. Inserting this into \eqref{second_moment_Q}, we deduce that
\begin{equation}
    \mathcal{S}(|L(1/2,\nu_{q,\omega}) \mathcal{M}_{\omega}(q)|^2; F) = \frac{2 \check{F}(0) X |C_{\omega}|^2}{D_{\omega} (\log{M})^2} \sum_{\substack{ 0\neq \mathfrak{d} \unlhd \Z[i] \\ N(\mathfrak{d}) \leq M \\ (\mathfrak{d}, 2) = 1}} \frac{\mu^2(\mathfrak{d}) |G_{\omega}(\mathfrak{d})|^2}{N(\mathfrak{d}) H_{\omega}(\mathfrak{d})} \log\Big({\frac{XM^2}{N(\mathfrak d)^2}}\Big) + o(X). \label{second_mom_prelim_asymp}
\end{equation}

Partial summation with the asymptotic \eqref{first_mom_asymptotic_comp} gives
\begin{align*}
    \sum_{\substack{ 0\neq \mathfrak{d} \unlhd \Z[i] \\ N(\mathfrak{d}) \leq M \\ (\mathfrak{d}, 2) = 1}} \frac{\mu^2(\mathfrak{d}) |G_{\omega}(\mathfrak{d})|^2}{N(\mathfrak{d}) H_{\omega}(\mathfrak{d})} \log{N(\mathfrak{d})} = \frac{ \pi}{16} \mathcal{Q}_{\omega} (\log{M})^2 + O(\log{M}).
\end{align*}
Combining this with \eqref{first_mom_asymptotic_comp} in \eqref{second_mom_prelim_asymp}, and recalling from \eqref{eq:MYU_def} that $M = X^{\vartheta^3}$, we obtain a term almost identical to the first moment, and conclude that
\begin{equation}\label{second_mom_asymp}
    \mathcal{S}(|L(1/2,\nu_{q,\omega}) \mathcal{M}_{\omega}(q)|^2; F) = \frac{\pi}{24 \cdot \zeta_{\Q(i)}(2)} \Big(1 + \frac{1}{\vartheta^3}\Big) \check{F}(0) X  + o(X). 
\end{equation}


\subsection{Final density computation}

We are ready to prove the main non-vanishing result.

\begin{proof}[Proof of \cref{mainthm}]

    Recall that $F$ is supported in $(1, 2)$ and satisfies $0 \leq F(t) \leq 1$ for all $t \in \R$. By Cauchy--Schwarz and the asymptotic expressions \eqref{first_mom_asymp} and \eqref{second_mom_asymp},
    \begin{align*}
        \sum_{\substack{q \in \mathbb{Z}[i] \\ q \equiv 1 \pmod{\lambda^7} \\ X < N(q) \leq 2X  \\ L(1/2,\nu_{q,\omega}) \neq 0}} \mu^2(q) & \geq \sum_{\substack{q \in \mathbb{Z}[i] \\ q \equiv 1 \pmod{\lambda^7} \\ L(1/2,\nu_{q,\omega}) \neq 0}} \mu^2(q) F \Big( \frac{N(q)}{X} \Big) \geq \frac{|\mathcal{S}(L(1/2,\nu_{q,\omega}) \mathcal{M}_{\omega}(q); F)|^2}{\mathcal{S}(|L(1/2,\nu_{q,\omega}) \mathcal{M}_{\omega}(q)|^2; F)} \\
        &= \frac{\pi}{ 96 \cdot \zeta_{\Q (i)}(2)} \frac{\vartheta^3}{\vartheta^3 + 1} \check{F}(0) X  + o(X).
    \end{align*}

    We take $F$ arbitrarily close to the indicator function of the interval $(1, 2)$, so that $\check{F}(0)$ approaches $1$. Note that
    \begin{align*}
        \sum_{\substack{q \in \mathbb{Z}[i] \\ q \equiv 1 \pmod{\lambda^7} \\ X < N(q) \leq 2X}} \mu^2(q) = \frac{1}{16} \sum_{\substack{q \in \mathbb{Z}[i] \\ q \equiv 1 \pmod{\lambda^3} \\ X < N(q) \leq 2X}} \mu^2(q) + o(X) = \frac{X}{16} \mathop{\mathrm{Res}}_{s=1} \frac{\zeta_\lambda(s)}{\zeta_\lambda(2s)} + o(X).
    \end{align*}
    Moreover, \eqref{eq:zetalambdares} and \eqref{eq:zeta_def} imply
    \begin{align*}
        \frac{X}{16} \mathop{\mathrm{Res}}_{s=1} \frac{\zeta_\lambda(s)}{\zeta_\lambda(2s)} = \frac{X}{12 \cdot \zeta_{\Q(i)}(2)  } \Res_{s=1}\zeta_{\lambda}(s) = \frac{\pi X}{96 \cdot \zeta_{\Q(i)}(2)}.
    \end{align*}
    Therefore for any $\varepsilon > 0$ we have
    \begin{equation*}
        \sum_{\substack{q \in \mathbb{Z}[i] \\ q \equiv 1 \pmod{\lambda^7} \\ X < N(q) \leq 2X  \\ L(1/2,\nu_{q,\omega}) \neq 0}} \mu^2(q) \geq \Big(\frac{\vartheta^3}{\vartheta^3 + 1} - \varepsilon \Big) \sum_{\substack{q \in \mathbb{Z}[i] \\ q \equiv 1 \pmod{\lambda^7} \\ X < N(q) \leq 2X}} \mu^2(q)
    \end{equation*}
    for all $X$ sufficiently large in terms of $\varepsilon$. Then sum dyadically over $X$ to finish the proof.

\end{proof}


\bibliography{references}

@book {IR90,
    AUTHOR = {Ireland, K.~ and Rosen, M.~},
     TITLE = {A classical introduction to modern number theory},
    SERIES = {Graduate Texts in Mathematics},
    VOLUME = {84},
   EDITION = {Second edition},
 PUBLISHER = {Springer-Verlag, New York},
      YEAR = {1990},
     PAGES = {xiv+389},
      ISBN = {0-387-97329-X},
   MRCLASS = {11-01 (11-02)},
  MRNUMBER = {1070716},
MRREVIEWER = {Glenn Stevens},
       DOI = {10.1007/978-1-4757-2103-4},
       URL = {https://doi.org/10.1007/978-1-4757-2103-4},
}

@article {BGL,
    AUTHOR = {Blomer, V. and Goldmakher, L. and Louvel, B.},
     TITLE = {{$L$}-functions with {$n$}-th-order twists},
   JOURNAL = {Int. Math. Res. Not. IMRN},
  FJOURNAL = {International Mathematics Research Notices. IMRN},
      YEAR = {2014},
    NUMBER = {7},
     PAGES = {1925--1955},
      ISSN = {1073-7928,1687-0247},
   MRCLASS = {11M41 (11F66 11R42)},
  MRNUMBER = {3190355},
MRREVIEWER = {Timothy\ Lee\ Gillespie},
       DOI = {10.1093/imrn/rns257},
       URL = {https://doi.org/10.1093/imrn/rns257},
}

@incollection {Pat2,
    AUTHOR = {Patterson, S. J.},
     TITLE = {Whittaker models of generalized theta series},
 BOOKTITLE = {Seminar on number theory, {P}aris 1982--83 ({P}aris,
              1982/1983)},
    SERIES = {Progr. Math.},
    VOLUME = {51},
     PAGES = {199--232},
 PUBLISHER = {Birkh\"{a}user Boston, Boston, MA},
      YEAR = {1984},
   MRCLASS = {11F70 (11F27 22E55)},
  MRNUMBER = {791596},
}

@article {EckPat,
    AUTHOR = {Eckhardt, C. and Patterson, S. J.},
     TITLE = {On the {F}ourier coefficients of biquadratic theta series},
   JOURNAL = {Proc. London Math. Soc. (3)},
  FJOURNAL = {Proceedings of the London Mathematical Society. Third Series},
    VOLUME = {64},
      YEAR = {1992},
    NUMBER = {2},
     PAGES = {225--264},
      ISSN = {0024-6115},
   MRCLASS = {11F30 (11F27)},
  MRNUMBER = {1143226},
MRREVIEWER = {Solomon Friedberg},
       DOI = {10.1112/plms/s3-64.2.225},
       URL = {https://doi.org/10.1112/plms/s3-64.2.225},
}

@article {DR,
    AUTHOR = {Dunn, A.~ and Radziwi\l\l, M.~},
     TITLE = {Bias in cubic {G}auss sums: {P}atterson's conjecture},
   JOURNAL = {Ann. of Math. (2)},
  FJOURNAL = {Annals of Mathematics. Second Series},
    VOLUME = {200},
      YEAR = {2024},
    NUMBER = {3},
     PAGES = {967--1057},
      ISSN = {0003-486X,1939-8980},
   MRCLASS = {11L20 (11F27 11F30 11L05 11L15 11N36)},
  MRNUMBER = {4816436},
MRREVIEWER = {Olivier\ Bordell\`es},
       DOI = {10.4007/annals.2024.200.3.3},
       URL = {https://doi.org/10.4007/annals.2024.200.3.3},
}

@article{DDHL,
      title={Quartic {G}auss sums over primes and metaplectic theta functions}, 
      author={David, C.~ and Dunn, A.~ and  Hamieh, H.~ and Lin, H.~},
      year={2025},
      journal={to appear in Algebra \& Number Theory}
}

@article {Suz1,
    AUTHOR = {Suzuki, T.~},
     TITLE = {Some results on the coefficients of the biquadratic theta
              series},
   JOURNAL = {J. Reine Angew. Math.},
  FJOURNAL = {Journal f\"{u}r die Reine und Angewandte Mathematik. [Crelle's Journal]},
    VOLUME = {340},
      YEAR = {1983},
     PAGES = {70--117},
      ISSN = {0075-4102},
   MRCLASS = {10D24 (10A15 10D12)},
  MRNUMBER = {691962},
MRREVIEWER = {A. I. Vinogradov},
       DOI = {10.1515/crll.1983.340.70},
       URL = {https://doi.org/10.1515/crll.1983.340.70},
}

@article {FHL,
    AUTHOR = {Friedberg, S.~ and Hoffstein, J.~ and Lieman, D.~},
     TITLE = {Double {D}irichlet series and the {$n$}-th order twists of {H}ecke {$L$}-series},
   JOURNAL = {Math. Ann.},
  FJOURNAL = {Mathematische Annalen},
    VOLUME = {327},
      YEAR = {2003},
    NUMBER = {2},
     PAGES = {315--338},
      ISSN = {0025-5831,1432-1807},
   MRCLASS = {11R42 (11F66 11M41 11R47)},
  MRNUMBER = {2015073},
MRREVIEWER = {M.\ Ram\ Murty},
       DOI = {10.1007/s00208-003-0455-4},
       URL = {https://doi.org/10.1007/s00208-003-0455-4},
}

@book {IK,
    AUTHOR = {Iwaniec, H.~ and Kowalski, E.~},
     TITLE = {Analytic number theory},
    SERIES = {American Mathematical Society Colloquium Publications},
    VOLUME = {53},
 PUBLISHER = {American Mathematical Society, Providence, RI},
      YEAR = {2004},
     PAGES = {xii+615},
      ISBN = {0-8218-3633-1},
   MRCLASS = {11-02 (11Fxx 11Lxx 11Mxx 11Nxx)},
  MRNUMBER = {2061214},
MRREVIEWER = {K.\ Soundararajan},
       DOI = {10.1090/coll/053},
       URL = {https://doi.org/10.1090/coll/053},
}

@article {Dia,
    AUTHOR = {Diaconu, A.~},
     TITLE = {Mean square values of {H}ecke {$L$}-series formed with {$r$}-th order characters},
   JOURNAL = {Invent. Math.},
  FJOURNAL = {Inventiones Mathematicae},
    VOLUME = {157},
      YEAR = {2004},
    NUMBER = {3},
     PAGES = {635--684},
      ISSN = {0020-9910,1432-1297},
   MRCLASS = {11F66 (11F67 11R42)},
  MRNUMBER = {2092772},
MRREVIEWER = {Emmanuel\ P.\ Royer},
       DOI = {10.1007/s00222-004-0363-6},
       URL = {https://doi.org/10.1007/s00222-004-0363-6},
}

@book{Kub,
  author    = {Kubota, T.~},
  title     = {On Automorphic Functions and the Reciprocity Law in a Number Field},
  series    = {Lectures in Mathematics, Department of Mathematics, Kyoto University},
  volume    = {2},
  publisher = {Kinokuniya Book-Store},
  year      = {1969},
  address   = {Tokyo},
  pages     = {65 pp.},
}

@misc{DLMF,
         key = "{\relax DLMF}",
       title = "{\it NIST Digital Library of Mathematical Functions}",
howpublished = "\url{https://dlmf.nist.gov/}, Release 1.2.5 of 2025-12-15",
         url = "https://dlmf.nist.gov/",
        note = "F.~W.~J. Olver, A.~B. {Olde Daalhuis}, D.~W. Lozier, B.~I. Schneider,
                R.~F. Boisvert, C.~W. Clark, B.~R. Miller, B.~V. Saunders,
                H.~S. Cohl, and M.~A. McClain, eds."
}

@book {H,
    AUTHOR = {Hasse, H.~},
     TITLE = {Vorlesungen \"{u}ber {Z}ahlentheorie},
    SERIES = {Die Grundlehren der mathematischen Wissenschaften in Einzeldarstellungen mit besonderer Ber\"{u}cksightigung der Anwendungsgebiete. Band LIX},
 PUBLISHER = {Springer-Verlag, Berlin-G\"{o}ttingen-Heidelberg},
      YEAR = {1950},
     PAGES = {xii+474},
   MRCLASS = {10.0X},
  MRNUMBER = {0051844},
MRREVIEWER = {P. T. Bateman},
}

@article {HB,
    AUTHOR = {Heath-Brown, D. R.},
     TITLE = {Kummer's conjecture for cubic {G}auss sums},
   JOURNAL = {Israel J. Math.},
  FJOURNAL = {Israel Journal of Mathematics},
    VOLUME = {120},
      YEAR = {2000},
    NUMBER = {part A},
     PAGES = {97--124},
      ISSN = {0021-2172},
   MRCLASS = {11L05 (11L40)},
  MRNUMBER = {1815372},
MRREVIEWER = {Matti Jutila},
       DOI = {10.1007/s11856-000-1273-y},
       URL = {https://doi-org.proxy2.library.illinois.edu/10.1007/s11856-000-1273-y},
}

@article {BY10,
    AUTHOR = {Baier, S.~ and Young, M.~P.~},
     TITLE = {Mean values with cubic characters},
   JOURNAL = {J. Number Theory},
  FJOURNAL = {Journal of Number Theory},
    VOLUME = {130},
      YEAR = {2010},
    NUMBER = {4},
     PAGES = {879--903},
      ISSN = {0022-314X,1096-1658},
   MRCLASS = {11M06 (11A15 11L05 11N37)},
  MRNUMBER = {2600408},
MRREVIEWER = {Olivier\ Bordell\`es},
       DOI = {10.1016/j.jnt.2009.11.007},
       URL = {https://doi-org.proxy2.library.illinois.edu/10.1016/j.jnt.2009.11.007},
}

@article {KatzSarnak,
    AUTHOR = {Katz, N.~M. and Sarnak, P.},
     TITLE = {Zeroes of zeta functions and symmetry},
   JOURNAL = {Bull. Amer. Math. Soc. (N.S.)},
  FJOURNAL = {American Mathematical Society. Bulletin. New Series},
    VOLUME = {36},
      YEAR = {1999},
    NUMBER = {1},
     PAGES = {1--26},
      ISSN = {0273-0979,1088-9485},
   MRCLASS = {11M41 (11F66 11G40)},
  MRNUMBER = {1640151},
MRREVIEWER = {Daniel\ A.\ Goldston},
       DOI = {10.1090/S0273-0979-99-00766-1},
       URL = {https://doi-org.proxy2.library.illinois.edu/10.1090/S0273-0979-99-00766-1},
}

@article {G25,
    AUTHOR = {G\"{u}lo{\u g}lu, A.~M.~},
     TITLE = {Non-vanishing of cubic {D}irichlet {$L$}-functions over the {E}isenstein field},
   JOURNAL = {Proc. Amer. Math. Soc.},
  FJOURNAL = {Proceedings of the American Mathematical Society},
    VOLUME = {153},
      YEAR = {2025},
    NUMBER = {5},
     PAGES = {1947--1961},
      ISSN = {0002-9939,1088-6826},
   MRCLASS = {11R42 (11R45)},
  MRNUMBER = {4881386},
MRREVIEWER = {Sami\ Omar},
       DOI = {10.1090/proc/17155},
       URL = {https://doi-org.proxy2.library.illinois.edu/10.1090/proc/17155},
}

@article {GY24,
    AUTHOR = {G\"{u}lo{\u g}lu, A.~M. and Yesilyurt, H.},
     TITLE = {Mollified moments of cubic {D}irichlet {$L$}-functions over the {E}isenstein field},
   JOURNAL = {J. Math. Anal. Appl.},
  FJOURNAL = {Journal of Mathematical Analysis and Applications},
    VOLUME = {533},
      YEAR = {2024},
    NUMBER = {2},
     PAGES = {Paper No. 128014, 49 pp.},
      ISSN = {0022-247X,1096-0813},
   MRCLASS = {11M20 (11R42)},
  MRNUMBER = {4677712},
MRREVIEWER = {Timothy\ Lee\ Gillespie},
       DOI = {10.1016/j.jmaa.2023.128014},
       URL = {https://doi-org.proxy2.library.illinois.edu/10.1016/j.jmaa.2023.128014},
}

@article {GZ22,
    AUTHOR = {Gao, P. and Zhao, L.},
     TITLE = {Bounds for moments of cubic and quartic {D}irichlet
              {$L$}-functions},
   JOURNAL = {Indag. Math. (N.S.)},
  FJOURNAL = {Koninklijke Nederlandse Akademie van Wetenschappen. Indagationes Mathematicae. New Series},
    VOLUME = {33},
      YEAR = {2022},
    NUMBER = {6},
     PAGES = {1263--1296},
      ISSN = {0019-3577,1872-6100},
   MRCLASS = {11M06},
  MRNUMBER = {4498233},
MRREVIEWER = {Olivier\ Bordell\`es},
       DOI = {10.1016/j.indag.2022.08.003},
       URL = {https://doi-org.proxy2.library.illinois.edu/10.1016/j.indag.2022.08.003},
}

@article {DG22,
    AUTHOR = {David, C. and G\"{u}lo{\u g}lu, A. M.},
     TITLE = {One-level density and non-vanishing for cubic {$L$}-functions over the {E}isenstein field},
   JOURNAL = {Int. Math. Res. Not. IMRN},
  FJOURNAL = {International Mathematics Research Notices. IMRN},
      YEAR = {2022},
    NUMBER = {23},
     PAGES = {18833--18873},
      ISSN = {1073-7928,1687-0247},
   MRCLASS = {11M50 (11R42)},
  MRNUMBER = {4519156},
MRREVIEWER = {Timothy\ Lee\ Gillespie},
       DOI = {10.1093/imrn/rnab240},
       URL = {https://doi-org.proxy2.library.illinois.edu/10.1093/imrn/rnab240},
}

@article {Pat1,
    AUTHOR = {Patterson, S. J.},
     TITLE = {A cubic analogue of the theta series},
   JOURNAL = {J. Reine Angew. Math.},
  FJOURNAL = {Journal f\"{u}r die Reine und Angewandte Mathematik. [Crelle's Journal]},
    VOLUME = {296},
      YEAR = {1977},
     PAGES = {125--161},
      ISSN = {0075-4102},
   MRCLASS = {10D20},
  MRNUMBER = {563068},
       DOI = {10.1515/crll.1977.296.125},
       URL = {https://doi-org.proxy2.library.illinois.edu/10.1515/crll.1977.296.125},
}

@article {sound00,
    AUTHOR = {Soundararajan, K.},
     TITLE = {Nonvanishing of quadratic {D}irichlet {$L$}-functions at {$s=\frac12$}},
   JOURNAL = {Ann. of Math. (2)},
  FJOURNAL = {Annals of Mathematics. Second Series},
    VOLUME = {152},
      YEAR = {2000},
    NUMBER = {2},
     PAGES = {447--488},
      ISSN = {0003-486X,1939-8980},
   MRCLASS = {11M20 (11R42)},
  MRNUMBER = {1804529},
MRREVIEWER = {J.\ B.\ Conrey},
       DOI = {10.2307/2661390},
       URL = {https://doi-org.proxy2.library.illinois.edu/10.2307/2661390},
}

@article {OS99,
    AUTHOR = {\"Ozl\"uk, A. E. and Snyder, C.},
     TITLE = {On the distribution of the nontrivial zeros of quadratic {$L$}-functions close to the real axis},
   JOURNAL = {Acta Arith.},
  FJOURNAL = {Acta Arithmetica},
    VOLUME = {91},
      YEAR = {1999},
    NUMBER = {3},
     PAGES = {209--228},
      ISSN = {0065-1036,1730-6264},
   MRCLASS = {11M26 (11M20)},
  MRNUMBER = {1735673},
MRREVIEWER = {Daniel\ A.\ Goldston},
       DOI = {10.4064/aa-91-3-209-228},
       URL = {https://doi-org.proxy2.library.illinois.edu/10.4064/aa-91-3-209-228},
}

@article {GZ20,
    AUTHOR = {Gao, P. and Zhao, L.},
     TITLE = {One-level density of low-lying zeros of quadratic and quartic {H}ecke {$L$}-functions},
   JOURNAL = {Canad. J. Math.},
  FJOURNAL = {Canadian Journal of Mathematics. Journal Canadien de
              Math\'ematiques},
    VOLUME = {72},
      YEAR = {2020},
    NUMBER = {2},
     PAGES = {427--454},
      ISSN = {0008-414X,1496-4279},
   MRCLASS = {11R42 (11L40 11M06 11M26 11M50 11R16)},
  MRNUMBER = {4081698},
MRREVIEWER = {Sandro\ Bettin},
       DOI = {10.4153/s0008414x1900021x},
       URL = {https://doi-org.proxy2.library.illinois.edu/10.4153/s0008414x1900021x},
}

@incollection {Gold79,
    AUTHOR = {Goldfeld, D.},
     TITLE = {Conjectures on elliptic curves over quadratic fields},
 BOOKTITLE = {Number theory, {C}arbondale 1979 ({P}roc. {S}outhern {I}llinois {C}onf., {S}outhern {I}llinois {U}niv., {C}arbondale, {I}ll., 1979)},
    SERIES = {Lecture Notes in Math.},
    VOLUME = {751},
     PAGES = {108--118},
 PUBLISHER = {Springer, Berlin},
      YEAR = {1979},
      ISBN = {3-540-09559-4},
   MRCLASS = {12A70 (14K07)},
  MRNUMBER = {564926},
MRREVIEWER = {Kenneth\ Kramer},
}

@article {BurTia26,
    AUTHOR = {Burungale, A. and Tian, Y.},
     TITLE = {A rank zero {$p$}-converse to a theorem of {G}ross--{Z}agier, {K}olyvagin and {R}ubin},
   JOURNAL = {Ann. of Math. (2)},
  FJOURNAL = {Annals of Mathematics. Second Series},
    VOLUME = {203},
      YEAR = {2026},
    NUMBER = {1},
     PAGES = {1--13},
      ISSN = {0003-486X,1939-8980},
   MRCLASS = {11G40 (11G15 11R23)},
  MRNUMBER = {5008551},
       DOI = {10.4007/annals.2026.203.1.1},
       URL = {https://doi-org.proxy2.library.illinois.edu/10.4007/annals.2026.203.1.1},
}

@article {Smith26II,
    AUTHOR = {Smith, A.},
     TITLE = {The distribution of {$\ell^\infty$}-{S}elmer groups in degree {$\ell$} twist families {II}},
   JOURNAL = {J. Amer. Math. Soc.},
  FJOURNAL = {Journal of the American Mathematical Society},
    VOLUME = {39},
      YEAR = {2026},
    NUMBER = {2},
     PAGES = {453--514},
      ISSN = {0894-0347,1088-6834},
   MRCLASS = {11R34 (11G05 11G10 11R29)},
  MRNUMBER = {4999533},
       DOI = {10.1090/jams/1063},
       URL = {https://doi-org.proxy2.library.illinois.edu/10.1090/jams/1063},
}

@article {Smith26I,
    AUTHOR = {Smith, A.},
     TITLE = {The distribution of {$\ell^\infty$}-{S}elmer groups in degree {$\ell$} twist families {I}},
   JOURNAL = {J. Amer. Math. Soc.},
  FJOURNAL = {Journal of the American Mathematical Society},
    VOLUME = {39},
      YEAR = {2026},
    NUMBER = {1},
     PAGES = {1--72},
      ISSN = {0894-0347,1088-6834},
   MRCLASS = {11R34 (11G05 11L40 11R29 11R45)},
  MRNUMBER = {4969355},
       DOI = {10.1090/jams/1062},
       URL = {https://doi-org.proxy2.library.illinois.edu/10.1090/jams/1062},
}

@article{Smith25,
      title={The {B}irch and {S}winnerton-{D}yer conjecture implies {G}oldfeld's conjecture}, 
      author={A. Smith},
      year={2025},
      journal={arXiv preprint},
}

@article{ABS24,
      title={Integers expressible as the sum of two rational cubes}, 
      author={L. Alp{\"o}ge and M. Bhargava and A. Shnidman},
      year={2024},
      journal={arXiv preprint},
}

@article {Rubin91,
    AUTHOR = {Rubin, K.},
     TITLE = {The ``main conjectures'' of {I}wasawa theory for imaginary quadratic fields},
   JOURNAL = {Invent. Math.},
  FJOURNAL = {Inventiones Mathematicae},
    VOLUME = {103},
      YEAR = {1991},
    NUMBER = {1},
     PAGES = {25--68},
      ISSN = {0020-9910,1432-1297},
   MRCLASS = {11R23 (11G05 11G40)},
  MRNUMBER = {1079839},
MRREVIEWER = {M.\ A.\ Kenku},
       DOI = {10.1007/BF01239508},
       URL = {https://doi-org.proxy2.library.illinois.edu/10.1007/BF01239508},
}

@article {CW76,
    AUTHOR = {Coates, J. and Wiles, A.},
     TITLE = {On the conjecture of {B}irch and {S}winnerton-{D}yer},
   JOURNAL = {Invent. Math.},
  FJOURNAL = {Inventiones Mathematicae},
    VOLUME = {39},
      YEAR = {1977},
    NUMBER = {3},
     PAGES = {223--251},
      ISSN = {0020-9910,1432-1297},
   MRCLASS = {14G10 (10D15 14G25)},
  MRNUMBER = {463176},
MRREVIEWER = {Kenneth\ A.\ Ribet},
       DOI = {10.1007/BF01402975},
       URL = {https://doi-org.proxy2.library.illinois.edu/10.1007/BF01402975},
}

@incollection {Sound23,
    AUTHOR = {Soundararajan, K.},
     TITLE = {The distribution of values of zeta and {$L$}-functions},
 BOOKTITLE = {I{CM}---{I}nternational {C}ongress of {M}athematicians. {V}ol. 2. {P}lenary lectures},
     PAGES = {1260--1310},
 PUBLISHER = {EMS Press, Berlin},
      YEAR = {2023},
      ISBN = {978-3-98547-060-0; 978-3-98547-560-5; 978-3-98547-058-7},
   MRCLASS = {11M06 (11-02 11M26 11M50 60F05 60F10)},
  MRNUMBER = {4680281},
MRREVIEWER = {Anurag\ Sahay},
}

@article {DFL25,
      title={Nonvanishing of ${L}$-functions associated to fixed order characters over function fields}, 
      author={C. David and A. Florea and M. Lalin},
      date={2025},
      journal = {arXiv preprint},
}

@article{DDDS24,
      title={Non-vanishing for cubic {H}ecke ${L}$-functions}, 
      author={David, C.~ and de Faveri, A.~ and Dunn, A.~ and Stucky, J.~},
      year={2024},
      journal={arXiv preprint}, 
}

@article {BSD,
    AUTHOR = {Birch, B. J. and Swinnerton-Dyer, H. P. F.},
     TITLE = {Notes on elliptic curves. {II}},
   JOURNAL = {J. Reine Angew. Math.},
  FJOURNAL = {Journal f\"ur die Reine und Angewandte Mathematik. [Crelle's Journal]},
    VOLUME = {218},
      YEAR = {1965},
     PAGES = {79--108},
      ISSN = {0075-4102,1435-5345},
   MRCLASS = {14.48 (10.12)},
  MRNUMBER = {179168},
MRREVIEWER = {D.\ J.\ Lewis},
       DOI = {10.1515/crll.1965.218.79},
       URL = {https://doi-org.proxy2.library.illinois.edu/10.1515/crll.1965.218.79},
}

@article {DFL21,
    AUTHOR = {David, C. and Florea, A. and Lalin, M.},
     TITLE = {Nonvanishing for cubic {$L$}-functions},
   JOURNAL = {Forum Math. Sigma},
  FJOURNAL = {Forum of Mathematics. Sigma},
    VOLUME = {9},
      YEAR = {2021},
     PAGES = {Paper No. e69, 58 pp.},
      ISSN = {2050-5094},
   MRCLASS = {11R59 (11M06 11M38 11R16 11R58)},
  MRNUMBER = {4323990},
MRREVIEWER = {Jos\'e\ Alejandro\ Lara Rodr\'iguez},
       DOI = {10.1017/fms.2021.62},
       URL = {https://doi-org.proxy2.library.illinois.edu/10.1017/fms.2021.62},
}

@article {BF18,
    AUTHOR = {Bui, H. M. and Florea, A.},
     TITLE = {Zeros of quadratic {D}irichlet {$L$}-functions in the hyperelliptic ensemble},
   JOURNAL = {Trans. Amer. Math. Soc.},
  FJOURNAL = {Transactions of the American Mathematical Society},
    VOLUME = {370},
      YEAR = {2018},
    NUMBER = {11},
     PAGES = {8013--8045},
      ISSN = {0002-9947,1088-6850},
   MRCLASS = {11M38 (11M06 11M50)},
  MRNUMBER = {3852456},
MRREVIEWER = {Kohji\ Matsumoto},
       DOI = {10.1090/tran/7317},
       URL = {https://doi-org.proxy2.library.illinois.edu/10.1090/tran/7317},
}

@article {Suz93,
    AUTHOR = {Suzuki, T.},
     TITLE = {On the biquadratic theta series},
   JOURNAL = {J. Reine Angew. Math.},
  FJOURNAL = {Journal f\"ur die Reine und Angewandte Mathematik. [Crelle's Journal]},
    VOLUME = {438},
      YEAR = {1993},
     PAGES = {31--85}
}

@article {Art78,
    AUTHOR = {Arthaud, N.},
     TITLE = {On {B}irch and {S}winnerton-{D}yer's conjecture for elliptic curves with complex multiplication. {I}},
   JOURNAL = {Compositio Math.},
  FJOURNAL = {Compositio Mathematica},
    VOLUME = {37},
      YEAR = {1978},
    NUMBER = {2},
     PAGES = {209--232}
}

@article {GZ86,
    AUTHOR = {Gross, B. and Zagier, D.},
     TITLE = {Heegner points and derivatives of {$L$}-series},
   JOURNAL = {Invent. Math.},
  FJOURNAL = {Inventiones Mathematicae},
    VOLUME = {84},
      YEAR = {1986},
    NUMBER = {2},
     PAGES = {225--320}
}

@incollection {Kol90,
    AUTHOR = {Kolyvagin, V.},
     TITLE = {Euler systems},
 BOOKTITLE = {The {G}rothendieck {F}estschrift, {V}ol.\ {II}},
    SERIES = {Progr. Math.},
    VOLUME = {87},
     PAGES = {435--483},
 PUBLISHER = {Birkh\"auser Boston, Boston, MA},
      YEAR = {1990}
}

@article {Rub87,
    AUTHOR = {Rubin, K.},
     TITLE = {Tate-{S}hafarevich groups and {$L$}-functions of elliptic curves with complex multiplication},
   JOURNAL = {Invent. Math.},
  FJOURNAL = {Inventiones Mathematicae},
    VOLUME = {89},
      YEAR = {1987},
    NUMBER = {3},
     PAGES = {527--559}
}

@article {Luo04,
    AUTHOR = {Luo, Wenzhi},
     TITLE = {On {H}ecke {$L$}-series associated with cubic characters},
   JOURNAL = {Compos. Math.},
  FJOURNAL = {Compositio Mathematica},
    VOLUME = {140},
      YEAR = {2004},
    NUMBER = {5},
     PAGES = {1191--1196}
}

@incollection {Del80,
    AUTHOR = {Deligne, P.},
     TITLE = {Sommes de {G}auss cubiques et rev\^etements de {${\rm SL}(2)$}\ [d'apr\`es {S}. {J}. {P}atterson]},
 BOOKTITLE = {S\'eminaire {B}ourbaki (1978/79)},
    SERIES = {Lecture Notes in Math.},
    VOLUME = {770, Exp. No. 539},
     PAGES = {244--277},
 PUBLISHER = {Springer, Berlin},
      YEAR = {1980}
}

@article{KS24,
      title={Sums of rational cubes and the {$3$}-{S}elmer group}, 
      author={P. Koymans and A. Smith},
      year={2024},
      journal={arXiv preprint},
}

@article {KL19,
    AUTHOR = {Kriz, D. and Li, C.},
     TITLE = {Goldfeld's conjecture and congruences between {H}eegner
              points},
   JOURNAL = {Forum Math. Sigma},
  FJOURNAL = {Forum of Mathematics. Sigma},
    VOLUME = {7},
      YEAR = {2019},
     PAGES = {Paper No. e15, 80 pp.}
}

@article {BES20,
    AUTHOR = {Bhargava, M. and Elkies, N. and Shnidman, A.},
     TITLE = {The average size of the 3-isogeny {S}elmer groups of elliptic curves {$y^2=x^3+k$}},
   JOURNAL = {J. Lond. Math. Soc. (2)},
  FJOURNAL = {Journal of the London Mathematical Society. Second Series},
    VOLUME = {101},
      YEAR = {2020},
    NUMBER = {1},
     PAGES = {299--327}
}

@article {BKLS19,
    AUTHOR = {Bhargava, M. and Klagsbrun, Z. and Lemke Oliver, R. and Shnidman, A.},
     TITLE = {3-isogeny {S}elmer groups and ranks of abelian varieties in quadratic twist families over a number field},
   JOURNAL = {Duke Math. J.},
  FJOURNAL = {Duke Mathematical Journal},
    VOLUME = {168},
      YEAR = {2019},
    NUMBER = {15},
     PAGES = {2951--2989}
}

@article {Lie94,
    AUTHOR = {Lieman, D.},
     TITLE = {Nonvanishing of {$L$}-series associated to cubic twists of elliptic curves},
   JOURNAL = {Ann. of Math. (2)},
  FJOURNAL = {Annals of Mathematics. Second Series},
    VOLUME = {140},
      YEAR = {1994},
    NUMBER = {1},
     PAGES = {81--108}
}

@article {MM91,
    AUTHOR = {Murty, M. R. and Murty, V. K.},
     TITLE = {Mean values of derivatives of modular {$L$}-series},
   JOURNAL = {Ann. of Math. (2)},
  FJOURNAL = {Annals of Mathematics. Second Series},
    VOLUME = {133},
      YEAR = {1991},
    NUMBER = {3},
     PAGES = {447--475}
}

@article {BFH90,
    AUTHOR = {Bump, D. and Friedberg, S. and Hoffstein, J.},
     TITLE = {Nonvanishing theorems for {$L$}-functions of modular forms and their derivatives},
   JOURNAL = {Invent. Math.},
  FJOURNAL = {Inventiones Mathematicae},
    VOLUME = {102},
      YEAR = {1990},
    NUMBER = {3},
     PAGES = {543--618}
}

@article {Kol88,
    AUTHOR = {Kolyvagin, V.},
     TITLE = {Finiteness of {$E({\bf Q})$} and {$\Sha(E,{\bf Q})$} for a subclass of {W}eil curves},
   JOURNAL = {Izv. Akad. Nauk SSSR Ser. Mat.},
  FJOURNAL = {Izvestiya Akademii Nauk SSSR. Seriya Matematicheskaya},
    VOLUME = {52},
      YEAR = {1988},
    NUMBER = {3},
     PAGES = {522--540, 670--671}
}

@article {DT05,
    AUTHOR = {Diaconu, A. and Tian, Y.},
     TITLE = {Twisted {F}ermat curves over totally real fields},
   JOURNAL = {Ann. of Math. (2)},
  FJOURNAL = {Annals of Mathematics. Second Series},
    VOLUME = {162},
      YEAR = {2005},
    NUMBER = {3},
     PAGES = {1353--1376}
}

@article {Iwa90,
    AUTHOR = {Iwaniec, H.},
     TITLE = {On the order of vanishing of modular {$L$}-functions at the critical point},
   JOURNAL = {S\'em. Th\'eor. Nombres Bordeaux (2)},
  FJOURNAL = {S\'eminaire de Th\'eorie des Nombres de Bordeaux. S\'erie 2},
    VOLUME = {2},
      YEAR = {1990},
    NUMBER = {2},
     PAGES = {365--376}
}

@article {Naj10,
    AUTHOR = {Najman, F.},
     TITLE = {Complete classification of torsion of elliptic curves over quadratic cyclotomic fields},
   JOURNAL = {J. Number Theory},
  FJOURNAL = {Journal of Number Theory},
    VOLUME = {130},
      YEAR = {2010},
    NUMBER = {9},
     PAGES = {1964--1968}
}

\end{document}